\documentclass[twoside,11pt]{article}

\usepackage{amsmath, mathtools, amssymb, amsthm}
\usepackage{natbib}
\usepackage{microtype}
\usepackage{enumerate, multicol}
\usepackage{bbm}
\usepackage[linesnumbered,ruled,vlined]{algorithm2e}
\usepackage[font=small,labelfont=bf]{caption}
\usepackage{array} % For p{} and m{} column modifiers
\usepackage[table]{xcolor}
\usepackage{tikz}
\usetikzlibrary{positioning, arrows, shapes.geometric, calc}
\usepackage{fullpage}
\usepackage[colorlinks=true, linkcolor=blue, urlcolor=blue, citecolor=blue]{hyperref}

\definecolor{gaussianfill}{RGB}{220,230,250}
\definecolor{kssgfill}{RGB}{220,250,230}
\definecolor{usgfill}{RGB}{245,230,250}

\SetKwInput{KwInput}{Input}
\SetKwInput{KwOutput}{Output}
\SetKwRepeat{KwRepeat}{repeat}{until}

\theoremstyle{plain}
\newtheorem{theorem}{Theorem}[section]
\newtheorem{lemma}[theorem]{Lemma}

\newtheorem{remark}[theorem]{Remark}
\newtheorem{definition}[theorem]{Definition}
\numberwithin{equation}{section}

\newcommand{\EE}{\mathbb{E}}
\newcommand{\II}{\mathbb{I}}

\newcommand{\NN}{\mathbb{N}}
\newcommand{\PP}{\mathbb{P}}
\newcommand{\QQ}{\mathbb{Q}}
\newcommand{\RR}{\mathbb{R}}

\newcommand{\cC}{\mathcal{C}}
\newcommand{\cE}{\mathcal{E}}

\newcommand{\cL}{\mathcal{L}}
\newcommand{\cM}{\mathcal{M}}
\newcommand{\cN}{\mathcal{N}}
\newcommand{\cO}{\mathcal{O}}
\newcommand{\cP}{\mathcal{P}}
\newcommand{\cQ}{\mathcal{Q}}

\newcommand{\cT}{\mathcal{T}}
\newcommand{\cU}{\mathcal{U}}

\newcommand{\cMloc}{\mathcal{M}^{\operatorname{loc}}}
\newcommand{\cMKloc}{\mathcal{M}_K^{\operatorname{loc}}}

\newcommand{\diam}{\mathrm{diam}}
\newcommand{\loc}{\operatorname{loc}}

\newcommand{\argmin}{\mathop{\mathrm{argmin}}}

\def\T{{^{\mathrm{\scriptscriptstyle T}}}}

\allowdisplaybreaks

\begin{document}

\begin{center}
{\LARGE Information theoretic limits of robust sub-Gaussian mean estimation under star-shaped constraints}

{\large
\begin{center}
Akshay Prasadan and Matey Neykov
\end{center}}

{Department of Statistics \& Data Science, Carnegie Mellon University\\
Department of Statistics and Data Science, Northwestern University\\
[2ex]\texttt{aprasada@andrew.cmu.edu}, ~~~ \texttt{mneykov@northwestern.edu}}
\end{center}

\begin{abstract}We obtain the minimax rate for a mean location model with a bounded star-shaped set $K \subseteq \RR^n$ constraint on the mean, in an adversarially corrupted data setting with Gaussian noise. We assume an unknown fraction $\epsilon \le 1/2-\kappa$ for some fixed $\kappa\in(0,1/2]$ of $N$ observations are arbitrarily corrupted. We obtain a minimax risk up to proportionality constants under the squared $\ell_2$ loss of $\max(\eta^{*2},\sigma^2\epsilon^2)\wedge d^2$ with
\begin{align*}
    \eta^* = \sup \bigg\{\eta \ge 0 : \frac{N\eta^2}{\sigma^2} \leq \log \cMKloc(\eta,c)\bigg\}, 
\end{align*}
where $\log \cMKloc(\eta,c)$ denotes the local entropy of the set $K$, $d$ is the diameter of $K$, $\sigma^2$ is the variance, and $c$ is some sufficiently large absolute constant. A variant of our algorithm achieves the same rate for settings with known or symmetric sub-Gaussian noise, with a smaller breakdown point, still of constant order. We further study the case of unknown sub-Gaussian noise and show that the rate is slightly slower: $\max(\eta^{*2},\sigma^2\epsilon^2\log(1/\epsilon))\wedge d^2$. We generalize our results to the case when $K$ is star-shaped but unbounded.

\end{abstract}

\tableofcontents

\section{Introduction}
\label{section:robust_gsm:introduction}

Robust (multivariate) mean estimation is a fundamental task in statistics \citep{huber2011robust,werner2011robust}. In this paper, we study the information-theoretic limits of this problem under the assumptions that we have sub-Gaussian noise (i.e., light-tailed noise) and a prior knowledge of the mean in the form of a star-shaped constraint. Recently, there has been a spike of interest in robust mean estimation, owing to remarkable developments coming from the theoretical computer science community, allowing the (unconstrained) mean to be robustly evaluated in polynomial time \citep{lai2016agnostic, diakonikolas2019robust,diakonikolas2017being} \citep[see also the book][for summaries of recent robust computationally tractable algorithms]{diakonikolas2023algorithmic}. Notably, the near optimal computationally efficient algorithms often assume some knowledge of, or upper bound on the covariance matrix of the (uncorrupted) data, which we will not explicitly require.\footnote{Although we do not explicitly require a bound on the covariance matrix, for most settings we do assume knowledge (or an accurate upper bound) of the sub-Gaussian parameter which implies a bound on the maximal eigenvalue of the covariance matrix.} Furthermore, in contrast to this line of work, we discard computational tractability from consideration and focus solely on statistical (minimax) optimality. Unlike some classic and more recent work, here we study the more general problem of having prior knowledge on the mean in the form of a star-shaped constraint (see Definition \ref{definition:star-shaped} below), and to the best of our knowledge our work is the first to address this problem. Furthermore, we provide guarantees on the expected error of our outlier robust estimator, which is rare in the literature as most bounds are in high-probability. We now formalize the problem of interest.

To start with, we consider the following corrupted Gaussian location model: we draw $\tilde{X}_i = \mu + \xi_i$ for $\xi_i\sim \cN(0, \sigma^2 \II_n)$ for $i \in [N] = \{1,\ldots, N\}$ and $\sigma^2$ unknown, but we do not observe these points. We assume $\mu\in K$ for some known bounded star-shaped set $K\subset\RR^n$.  Then, for some fixed $\kappa\in(0,1/2]$ not depending on $N$ or $\sigma$, a (possibly unknown) fraction $\epsilon\le 1/2-\kappa$ of the $\tilde X_i$ are arbitrarily corrupted by some procedure $\cC$. The corruption scheme is adversarial in the sense that $\cC$ can depend on the original data, $\mu$, and the set $K$. Furthermore, $\cC$ is assumed to have infinite computational resources and oracle knowledge of our algorithm. Let $\tilde X =(\tilde X_1,\dots,\tilde X_N)\in\RR^{N\times n}$ and let $X=\cC(\tilde X)\in\RR^{N\times n}$ have in the $i$th row the vector $X_i=\cC(\tilde X_i)$ if that observation is corrupted and $X_i=\tilde X_i$ otherwise. Assume that we observe $X$. Our goal is to construct a minimax optimal way of estimating the mean $\mu$, where the minimax rate is defined as 
\begin{align*}
    \inf_{\hat \mu} \sup_{\mu \in K}\sup_{\cC} \EE_\mu \|\hat \mu(\cC(\tilde X)) - \mu\|^2,
\end{align*} and we take the infimum over all estimators $\hat\mu$ applied to $X$ and supremum over all corruption schemes corrupting fraction $\epsilon$ of the data.

As advertised in the first paragraph, we also generalize this setting for models of the type $\tilde{X}_i=\mu+\xi_i$ where the $\xi_1,\dots, \xi_N$ are i.i.d. centered sub-Gaussian random vectors with parameter $\sigma$. Unlike in the Gaussian case, in the sub-Gaussian case, we must assume some knowledge of $\sigma$ (i.e., an accurate upper bound). Furthermore, in the case with  unknown sub-Gaussian noise, we also need knowledge of the fraction of corruptions, $\epsilon$. In all cases, $\epsilon$ is  required to be smaller than a small absolute constant (strictly smaller than $1/2$). We uncover an interesting phenomenon in the sub-Gaussian case: if one knows the noise distribution, the minimax rate is faster in comparison to the situation when one is simply told the ``good'' samples have been generated with unknown sub-Gaussian noise.

For both the Gaussian and sub-Gaussian cases, our technique involves a nontrivial modification of the proof of \citet[Lemma II.5]{neykov2022minimax}, and then repeating an algorithm similar to that of \citet{neykov2022minimax} that iteratively constructs local packing sets. Importantly, though, we use a new criterion to select the updates following a tournament-style selection process rather than a simple distance minimizer. Moreover, we must perform an additional pruning step in the iterative packing procedure. For the unknown sub-Gaussian noise case, the tournament selection process uses the one-dimensional truncated mean estimator from \cite{mendelson_robust_mean} as a subroutine, which helps us to achieve the optimal rate in any dimension and any star-shaped constraint.  

We also extend to the case when $K$ is unbounded. Clearly, letting $K = \RR^n$ recovers the usual setting of a robust and unconstrained mean estimation, as considered in most recent papers, e.g., \cite{chen2018robust}. In that sense, the assumption of prior knowledge on $\mu$ should not be viewed as a restriction, but rather a generalization of the problem. As a closing example, we derive the minimax rates for robust sparse mean estimation, where we assume that at most $s \lesssim n$ of the $n$ coordinates of $\mu$ are nonzero. This is, of course, an example of an unbounded star-shaped set centered at $0$.

We summarize our results for the bounded case in Table \ref{table:assumptions}, including the different assumptions on the corruption rate or distribution. In the unbounded case, the assumptions for all models are the same except both $\epsilon$ and $\sigma$ must be known. The minimum with $d$, the diameter of $K$, is dropped in the unbounded minimax rate.

\begin{table}[t]
\caption{Summary of assumptions and results for the bounded constraint set case. The definition of $\eta^{\ast}$ is given in, e.g., Theorem \ref{theorem:robust:minimax:rate:attained:gaussian}. The minimax rate for the unknown sub-Gaussian case is defined in \eqref{def:unknown:sub:gaussian:minimax:rate}.}
\label{table:assumptions}
\renewcommand{\arraystretch}{1.6} % or try 1.5
\centering % Whole table flush left
\small
\begin{tabular}{|>{\raggedright\arraybackslash}p{2.4cm}|
                >{\raggedright\arraybackslash}p{2.5cm}|
                >{\raggedright\arraybackslash}p{3.7cm}|
                >{\raggedright\arraybackslash}p{4.2cm}|}
\hline
\rowcolor{gray!15}
\textbf{Noise Model} & \textbf{Corruption Rate} & \textbf{Distributional Assumptions} & \textbf{Minimax Rate} \\
\hline
\textsc{Gaussian} 
& $\epsilon$ unknown; $\epsilon < \frac{1}{2} - \kappa$ (fixed unknown $\kappa$) 
& Covariance $\sigma^2\mathbb{I}_n$; $\sigma$ unknown
& $\max\left( \eta^{\ast 2},\, \epsilon^2 \sigma^2 \right)\! \wedge d^2$ \\
\hline
\textsc{Known / Symmetric sub-Gaussian} 
& $\epsilon$ unknown; $\epsilon \leq$ absolute constant 
& Distribution known or symmetric; known upper bound on sub-Gaussian parameter $\sigma$ 
& $\lesssim\max\left( \eta^{\ast 2},\, \epsilon^2 \sigma^2 \right)\! \wedge d^2$ \\
\hline
\textsc{Unknown sub-Gaussian} 
& $\epsilon$ known; $\epsilon \leq \tfrac{1}{32}$ 
& Known upper bound on sub-Gaussian parameter $\sigma$ 
& $\max\left( \eta^{\ast 2},\, \epsilon^2 \log\left( \tfrac{1}{\epsilon} \right) \sigma^2 \right)\! \wedge d^2
$ \\
\hline
\end{tabular}
\end{table}

\subsection{Related Literature}

There are numerous contrasts between our work and existing literature, usually due to one or more of the following: an unconstrained setting; error bounds with high probability rather than expectation; sufficiently large sample size requirements; different notions of distance; a non-adversarial Huber contamination model; sub-optimal sub-Gaussian rates; symmetry assumptions on sub-Gaussian noise; and non-matching lower and upper bounds. In contrast, our bounds work in both unconstrained and constrained settings, with any sample size, fully adversarial corruption, very mild sub-Gaussian distributional requirements (including asymmetric noise), and achieve the minimax rate in expectation as a consequence of our high probability bounds. 

\citet{bateni2019confidence} explain that the popularity of high probability bounds rather than in expectation arises from the divergence of the minimax risk in the Huber contamination setting with unbounded sets $K$, primarily due to the number of outliers being random. In the adversarial setting, by contrast, we have a deterministic bound on the frequency of outliers and are able to obtain useful minimax bounds in expectation. \citet{dalalyan2022} add that high probability bounds lead to estimators adaptive to the corruption rate $\epsilon$ using Lepski's method.

An unconstrained and adversarial (as opposed to Huber) setting is studied in \citet{diakonikolas2017being},  \citet{minsker2018uniform}, \citet{mendelson_robust_mean}, \citet{bateni2022nearly}, \citet{depersin2022}, \citet{minasyan2023statistically}, and \citet{novikov_gleb_stefan_2023}, all of which are done in high probability. 
\citet{diakonikolas2017being}  give a polynomial time and sample algorithm that achieves the optimal unconstrained sub-Gaussian rate of $\epsilon^2\log(1/\epsilon)$, however, the authors assume an identity covariance matrix. Indeed, \citet{novikov_gleb_stefan_2023} remark that with unknown noise it is not known how to achieve the optimal $\epsilon^2\log(1/\epsilon)$ rate for general sub-Gaussian distributions in polynomial time and samples. Similarly, \citet{bateni2022nearly} give a computationally tractable algorithm that is nearly optimal but requires known sub-Gaussian noise. That our algorithm is computationally infeasible is of no surprise, given that it simultaneously handles unknown sub-Gaussian noise and a constrained setting. 

There has nonetheless been partial progress since. \citet{minsker2018uniform} includes the  robust mean estimation problem for heavy tailed distributions as a special case and recover the optimal rate under some moment conditions; the author's algorithm is not adaptive to $\epsilon$, requires a tuning parameter, and is unconstrained. The estimator also uses knowledge about the eigenvalues of the covariance matrix. \citet{mendelson_robust_mean} achieves a statistically optimal rate of $\epsilon$ in high probability for a large class of distributions, and obtain the same optimal $\epsilon^2\log(1/\epsilon)$ rate that we do for unknown sub-Gaussian noise (without constraints). However, consistent with the earlier remarks, their algorithm is not computationally feasible. \citet{depersin2022} obtained polynomial time and sample error guarantees in high probability but achieve suboptimal sub-Gaussian rates of $\epsilon$. \citet{novikov_gleb_stefan_2023} examine a rich class of symmetric noise distributions and efficiently achieve the optimal Gaussian rate with known noise and achieve it in (near) polynomial time in the unknown noise case. The phenomenon of symmetry assumptions leading to the Gaussian rate is also seen in our work in our symmetric and known-noise sub-Gaussian case. \citet{minasyan2023statistically} consider Gaussian noise and any sample size and achieve an optimal rate that is dimension-independent aside from the effective rank of the covariance matrix. \citet{abdalla2024covariance} achieve similar dimension-free rates but for covariance matrix estimation instead. More recently, \citet{diakonikolas2024sos} demonstrate the existence of a polynomial time algorithm to achieve $\epsilon^{2-2/t}$ error with unknown covariance where the authors take $t$ to be the number of certifiably bounded moments (see their Definition 1.3). This is an improvement to the $\epsilon$-rate previously known among all efficient algorithms but still shy of the optimal rate.

Several other papers consider the Huber contamination model rather than a fully adversarial one. \citet{chen_decision_theory_2016} use a very similar set-up  to our paper (albeit for density estimation rather than mean estimation) by constructing a robust testing procedure (compare the author's equation (5) with our Definition \eqref{definition:psi:gaussian}) and also derive Type I error bounds. Their resulting rates also depend on the geometry of the set. However, they use the Huber model and total variation distance rather than the $\|\cdot\|$-norm in their metric entropy calculations. Moreover, the authors constrain their parameter space to be totally bounded. The authors also do not provide a matching lower bound to their entropy term. \citet{pmlr-v108-prasad20a} is similar to our work in that the authors use a robust univariate tool to tackle a multivariate problem (for instance, we used the trimmed mean estimator of \citet{mendelson_robust_mean} in our sub-Gaussian setting). The authors also work with the Huber model, consider heavier tailed distributions, and derive results in high probability. \citet{bateni2019confidence} consider both several contamination settings, including an adversarial ones, and impose a shape constraint of the probability simplex.

Our results also strictly generalize \citet{neykov2022minimax} by taking $N = 1$ and $\epsilon=0$, i.e., the uncontaminated, convex constrained (sub)-Gaussian mean estimation problem. Moreover, the authors discovered that \citet{neykov2022minimax} only uses convexity in the proof that the local metric entropy is decreasing in $\eta$. But this property holds for star-shaped sets as well (see Lemma \ref{lemma:non:increasing:local:entropy}), so both \citet{neykov2022minimax} and our current work easily translate into a setting much broader than convex constraints. However, the adaptivity of the estimator to the true point requires convexity \citep[see Theorem IV.4]{neykov2022minimax}.

Our closing example of robust sparse mean estimation builds on a large body of existing work. Our derived minimax rate of $\max\left(\tfrac{\sigma^2s\log(1+n/s)}{N}, \sigma^2\epsilon^2\right)$ was established in the $\epsilon=0$, known $\sigma$ setting by \citet{donoho1992maximum}. With adversarial corruption, Theorem 88 of \citet{diakonikolas2022robust} (Gaussian setting) is the closest analogue, albeit their result is in high probability rather than expectation and holds only for sufficiently large $N$ (when the term $\sigma^2 \epsilon^2$ is dominant). \citet{balakrishnan2017computationally} also derive high probability rates with Huber contamination and Gaussian noise.

\subsection{Notation and Definitions}

 Let $\Phi$ denote the cumulative distribution function of the standard normal distribution which we denote $\cN(0,1)$, and let $\phi=\Phi'$ be its probability density function. The maximum of $a,b\in\RR$ is denoted $a\vee b$ and the minimum $a\wedge b$. We write $[N]$ for $\{1,2,\dots,N\}$ for any $N \in \NN$. We mean the natural logarithm when we write $\log(\cdot)$. We use $\|\cdot\|$ to denote the standard $\ell_2$-Euclidean norm on $\RR^n$. Define $B(\nu, r)=\{\mu\in\RR^n:\|\mu-\nu\|\le r\}$, and define the unit sphere $S^{n-1}=\{\mu\in\RR^n:\|\mu\|=1\}$. We denote a binomial random variable with parameters $N$ and $p$ as $\mathrm{Bin}(N,p).$ Indicator variables are denoted using $\mathbbm{1}(\cdot)$ or $\mathbbm{1}_{(\cdot)}$. We write $\II_n$ or just $\II$ for the $n\times n$ identity matrix. We say a mean 0 random vector $\xi\in\RR^n$ is sub-Gaussian with parameter $\sigma$ if for all $t\in\RR$ and $v\in S^{n-1}$ we have $\EE\exp(t\xi^T v)\le \exp(t^2\sigma^2/2)$.  Occasionally we will use $\Omega$ to denote dropping multiplicative absolute constants.

We now give a definition of local metric entropy of $K$, which is central to our results.

\begin{definition}[Local Metric Entropy] \label{definition:local:metric:entropy} Given a set $K$, define the $\eta$-packing number of $K$ as the maximum cardinality $M=\cM(\eta,K)$ of a set $\{\nu_1,\dots,\nu_{M}\}\subset K$ such that $\|\nu_i-\nu_j\|>\eta$ for all $i\ne j$. Given a fixed constant $c>0$, define the local metric entropy of $K$ as $\cMKloc(\eta,c) = \sup_{\nu\in K} \cM(\eta/c, B(\nu,\eta)\cap K)$.
\end{definition}

\begin{definition}[Star-Shaped Sets] \label{definition:star-shaped} A set $K$ is star-shaped with center $k^{\ast}$ if there exists $k^{\ast}\in K$ such that for any $k\in K$ and $\alpha\in[0,1]$, the point $\alpha k +(1-\alpha)k^{\ast}\in K$. We will refer to any such point $k^*$ as a center of $K$.

\end{definition} It is easy to see that convexity is equivalent to having the star-shaped property for all possible centers in the set.

\begin{figure}
\caption{Left: A bounded, non-convex, star-shaped set. Right: An unbounded, non-convex, star-shaped set.} 
\centering
\begin{tikzpicture}[x=1pt,y=1pt,yscale=-0.5,xscale=0.5]
%Shape: Polygon Curved [id:ds044401323076625365] 
%Shape: Polygon Curved [id:ds044401323076625365] 
%Shape: Polygon Curved [id:ds044401323076625365] 
\draw  [fill={rgb, 255:red, 155; green, 155; blue, 155 }  ,fill opacity=0.33 ] (124.79,31.36) .. controls (172.68,90.54) and (240.96,106.41) .. (290.91,80.48) .. controls (340.86,54.54) and (395.02,44.25) .. (393.04,105.23) .. controls (391.06,166.21) and (308.68,133.73) .. (261.28,230.82) .. controls (213.89,327.91) and (132.46,215.22) .. (47.06,198.11) .. controls (-38.33,180.99) and (76.89,-27.83) .. (124.79,31.36) -- cycle ;
%Shape: Ellipse [id:dp5284217612385487] 
\draw  [fill={rgb, 255:red, 0; green, 0; blue, 0 }  ,fill opacity=1 ] (174.16,139.7) .. controls (174.16,141.59) and (175.55,143.13) .. (177.25,143.13) .. controls (178.95,143.13) and (180.33,141.59) .. (180.33,139.7) .. controls (180.33,137.8) and (178.95,136.26) .. (177.25,136.26) .. controls (175.55,136.26) and (174.16,137.8) .. (174.16,139.7) -- cycle ;
%Straight Lines [id:da9318194525990158] 
\draw [color={rgb, 255:red, 218; green, 115; blue, 127 }  ,draw opacity=1 ][line width=1.5]  [dash pattern={on 1.69pt off 2.76pt}]  (119.48,26.45) -- (322.7,67.64) ;
\draw [shift={(322.7,67.64)}, rotate = 11.46] [color={rgb, 255:red, 218; green, 115; blue, 127 }  ,draw opacity=1 ][fill={rgb, 255:red, 218; green, 115; blue, 127 }  ,fill opacity=1 ][line width=1.5]      (0, 0) circle [x radius= 2.61, y radius= 2.61]   ;
\draw [shift={(119.48,26.45)}, rotate = 11.46] [color={rgb, 255:red, 218; green, 115; blue, 127 }  ,draw opacity=1 ][fill={rgb, 255:red, 218; green, 115; blue, 127 }  ,fill opacity=1 ][line width=1.5]      (0, 0) circle [x radius= 2.61, y radius= 2.61]   ;
%Curve Lines [id:da8680660206744186] 
\draw [fill={rgb, 255:red, 217; green, 217; blue, 217 }  ,fill opacity=1 ][line width=0.75]    (410.06,266.78) .. controls (410.37,244.97) and (409.56,196.47) .. (410,150) .. controls (410.46,101.08) and (479.54,98.92) .. (480,50) .. controls (480.46,1.08) and (579.46,0.08) .. (580,50) .. controls (580.54,99.92) and (649.54,98.92) .. (650,150) .. controls (650.44,198.78) and (650.46,246.58) .. (650.06,267.26) ;
\draw [shift={(650,270)}, rotate = 271.39] [fill={rgb, 255:red, 0; green, 0; blue, 0 }  ][line width=0.08]  [draw opacity=0] (8.93,-4.29) -- (0,0) -- (8.93,4.29) -- cycle    ;
\draw [shift={(410,270)}, rotate = 271.32] [fill={rgb, 255:red, 0; green, 0; blue, 0 }  ][line width=0.08]  [draw opacity=0] (8.93,-4.29) -- (0,0) -- (8.93,4.29) -- cycle    ;
%Straight Lines [id:da4763569311456244] 
\draw [line width=0.75]  [dash pattern={on 0.84pt off 2.51pt}]  (530,260) -- (530,298) ;
\draw [shift={(530,300)}, rotate = 270] [color={rgb, 255:red, 0; green, 0; blue, 0 }  ][line width=0.75]    (10.93,-3.29) .. controls (6.95,-1.4) and (3.31,-0.3) .. (0,0) .. controls (3.31,0.3) and (6.95,1.4) .. (10.93,3.29)   ;
%Straight Lines [id:da07302722799317851] 
\draw [line width=0.75]  [dash pattern={on 0.84pt off 2.51pt}]  (620,260) -- (620,298) ;
\draw [shift={(620,300)}, rotate = 270] [color={rgb, 255:red, 0; green, 0; blue, 0 }  ][line width=0.75]    (10.93,-3.29) .. controls (6.95,-1.4) and (3.31,-0.3) .. (0,0) .. controls (3.31,0.3) and (6.95,1.4) .. (10.93,3.29)   ;
%Straight Lines [id:da7817971275080731] 
\draw [line width=0.75]  [dash pattern={on 0.84pt off 2.51pt}]  (440,260) -- (440,298) ;
\draw [shift={(440,300)}, rotate = 270] [color={rgb, 255:red, 0; green, 0; blue, 0 }  ][line width=0.75]    (10.93,-3.29) .. controls (6.95,-1.4) and (3.31,-0.3) .. (0,0) .. controls (3.31,0.3) and (6.95,1.4) .. (10.93,3.29)   ;
%Straight Lines [id:da736882257373618] 
\draw [line width=0.75]  [dash pattern={on 0.84pt off 2.51pt}]  (580,260) -- (580,298) ;
\draw [shift={(580,300)}, rotate = 270] [color={rgb, 255:red, 0; green, 0; blue, 0 }  ][line width=0.75]    (10.93,-3.29) .. controls (6.95,-1.4) and (3.31,-0.3) .. (0,0) .. controls (3.31,0.3) and (6.95,1.4) .. (10.93,3.29)   ;
%Straight Lines [id:da8512314173530962] 
\draw [line width=0.75]  [dash pattern={on 0.84pt off 2.51pt}]  (480,260) -- (480,298) ;
\draw [shift={(480,300)}, rotate = 270] [color={rgb, 255:red, 0; green, 0; blue, 0 }  ][line width=0.75]    (10.93,-3.29) .. controls (6.95,-1.4) and (3.31,-0.3) .. (0,0) .. controls (3.31,0.3) and (6.95,1.4) .. (10.93,3.29)   ;
%Shape: Ellipse [id:dp15307358018250383] 
\draw  [fill={rgb, 255:red, 0; green, 0; blue, 0 }  ,fill opacity=1 ][line width=0.75]  (532.1,124.57) .. controls (532.05,126.46) and (533.38,128.04) .. (535.08,128.09) .. controls (536.79,128.15) and (538.21,126.65) .. (538.27,124.75) .. controls (538.33,122.86) and (537,121.28) .. (535.29,121.23) .. controls (533.59,121.18) and (532.16,122.67) .. (532.1,124.57) -- cycle ;
%Straight Lines [id:da8428148222288945] 
\draw [color={rgb, 255:red, 218; green, 115; blue, 127 }  ,draw opacity=1 ][line width=1.5]  [dash pattern={on 1.69pt off 2.76pt}]  (480,40) -- (424.24,115.56) ;
\draw [shift={(424.24,115.56)}, rotate = 126.43] [color={rgb, 255:red, 218; green, 115; blue, 127 }  ,draw opacity=1 ][fill={rgb, 255:red, 218; green, 115; blue, 127 }  ,fill opacity=1 ][line width=1.5]      (0, 0) circle [x radius= 2.61, y radius= 2.61]   ;
\draw [shift={(480,40)}, rotate = 126.43] [color={rgb, 255:red, 218; green, 115; blue, 127 }  ,draw opacity=1 ][fill={rgb, 255:red, 218; green, 115; blue, 127 }  ,fill opacity=1 ][line width=1.5]      (0, 0) circle [x radius= 2.61, y radius= 2.61]   ;

% Text Node
\draw (176.69,113.59) node [anchor=north west][inner sep=0.75pt]  [font=\large] [align=left] {$\displaystyle k^{\ast }$};
% Text Node
\draw (55.41,164.86) node [anchor=north west][inner sep=0.75pt]  [font=\LARGE] [align=left] {$\displaystyle K$};
% Text Node
\draw (421,182) node [anchor=north west][inner sep=0.75pt]  [font=\LARGE] [align=left] {$\displaystyle K$};
% Text Node
\draw (543,102) node [anchor=north west][inner sep=0.75pt]  [font=\large] [align=left] {$\displaystyle k^{\ast }$};
\end{tikzpicture}

\end{figure}

In the first portion of our paper, we assume that $K$ is a bounded set and denote its $\ell_2$ diameter by $d = \diam(K)$. We now argue that star-shaped sets contain line segments of length proportional to the diameter of the set. This property will be useful in selecting points sufficiently far apart (e.g., in our lower bound results), as well as making $\cMKloc(\eta,c)$ sufficiently large by increasing $c$ (when we show the minimax rate is attained).

\begin{lemma} \label{lemma:star:shaped:has:line:segment} Let $K$ be a star-shaped set. Then $K$ contains a line segment of length $\ge d/3$.
\end{lemma}
\begin{proof}
    Let $k^{\ast}$ be a center of $K$. Pick any $\delta>0$ and pick points $\nu_1,\nu_2\in K$ such that $\|\nu_1-\nu_2\|> d-\delta$. Then \[d-\delta <\|\nu_1-\nu_2\| \le \|\nu_1-k^{\ast}\|+\|\nu_2-k^{\ast}\|.\] Then at least one of $\|\nu_1-k^{\ast}\|$ or $\|\nu_2-k^{\ast}\|$ is $\ge (d-\delta)/2$, and both the line segments from $\nu_1$ to $k^{\ast}$ and from $\nu_2$ to $k^{\ast}$ are contained in $K$ by the star-shaped property. For sufficiently small $\delta$, $(d-\delta)/2\ge d/3$.
\end{proof}

Next we demonstrate that the non-increasing property of $\cMKloc(\cdot, c)$ still holds in a star-shaped setting. The proof in the convex setting was given in \citet[Lemma II.8]{neykov2022minimax}.

\begin{lemma}\label{lemma:non:increasing:local:entropy} For any star-shaped set $K$ and constant $c>0$, the map $\eta\mapsto \log \cMKloc(\eta, c)$ is non-increasing.

\end{lemma}
    \begin{proof}
        Let $K$ be star-shaped with center $k^{\ast}$. Pick $\eta>0$ and assume $\nu^{\ast}\in K$ achieves $\cMloc(\eta, c)=\sup_{\nu\in K} \cM(\eta/c, B(\nu,\eta)\cap K)$. Let $\nu_1,\dots,\nu_M$ be a maximal $\eta/c$-packing of $B(\nu^{\ast},\eta)\cap K$. Pick any $\alpha\in(0,1]$ and define $\nu_i' = \alpha\nu_i + (1-\alpha)k^{\ast}$ for each $i$. The $\nu_i'\in K$ by the star-shaped property. Let us show the $\nu_i'$ form a $(\alpha\eta/c)$-packing of $B(\alpha\nu^{\ast}+(1-\alpha)k^{\ast}, \alpha \eta)\cap K$, noting that $\alpha\nu^{\ast}+(1-\alpha)k^{\ast}\in K$ again by the star-shaped property. This would imply by definition that $\cMloc(\alpha\eta, c)\ge \cMloc(\eta, c)$ for any $\alpha\in(0,1]$. Then if $0<\eta_1\le \eta_2$, by setting $\alpha = \eta_1/\eta_2\in(0,1]$, we will have \[\cMloc(\eta_1, c)=\cMloc(\alpha\eta_2, c)\ge \cMloc(\eta_2, c),\] proving the non-increasing property.
        
        Returning to the $\nu_i'$, first observe that for $i\ne j$, $\|\nu_i'-\nu_j'\| = \|\alpha(\nu_i-\nu_j)\|> \alpha \eta/c$ since the $\nu_i$ are $\eta/c$-separated. Thus, the points $\nu_i'$ are $(\alpha\eta/c)$-separated. On the other hand, \begin{align*}
            \|\nu_i' - \alpha\nu^{\ast}-(1-\alpha)k^{\ast}\| &= \| \alpha\nu_i + (1-\alpha)k^{\ast}- \alpha\nu^{\ast}-(1-\alpha)k^{\ast}\| \\
            &= \alpha \|\nu_i-\nu^{\ast}\| \\
            &\le \alpha \eta,
        \end{align*} since the $\nu_i\in B(\nu^{\ast},\eta)$. Therefore, the points $\nu_i'$ are in $B(\alpha\nu^{\ast}+(1-\alpha)k^{\ast}, \alpha \eta)\cap K$, and we indeed have a packing of this set as desired. 
    \end{proof}

\subsection{Organization}

We will first establish lower bounds for both the Gaussian and sub-Gaussian settings in Section \ref{section:robust_gsm:lower:bound}. We combine well-known techniques with some adaptations of recent work in the robust literature. The main challenge is to match these lower bounds (up to absolute constant factors). 

We break down this matching procedure into several sections. We present the i.i.d. Gaussian upper bound in Section \ref{section:robust_gsm:gaussian:upper:bound}. The algorithm itself is explained in Section \ref{section:infinite:tree:construction} and \ref{section:robust:algorithm}, and the proof of convergence is given in Section \ref{section:gaussian:error:bounding}. We then proceed to the case with sub-Gaussian noise setting in Section \ref{section:robust_gsm:subgaussian:upper:bound}, which we partition into a case with symmetric or known sub-Gaussian noise in Section \ref{subsection:symmetric:subgaussian:upper:bound} followed by the general (i.e., unknown) sub-Gaussian case in Section \ref{subsection:assymmetric:subgaussian:upper:bound}. 

Finally, we extend our results to unbounded sets in Section \ref{section:robust_gsm:unbounded} and include an example with sparse robust mean estimation in Section \ref{section:robust_gsm:sparse}. Concluding remarks are provided in Section \ref{section:robust_gsm:discussion}.

The overall logic of the paper is sketched out in Figure \ref{fig:flowchart}, at least for the bounded constraint case (Sections 2-4). Each column corresponds to a different noise setting, whether Gaussian, symmetric or known sub-Gaussian, or unknown sub-Gaussian. The top row is the lower bound, while the remaining four rows build up our upper bound results. 

The upper bound in each noise setting has the following structure. First, we establish a `Testing Lemma' that bounds the Type I error of a hypothesis test that evaluates which small-ball the truth $\mu$ belongs to. Then, we prove a `Tournament Lemma' which shows that given a local covering that contains $\mu$, our tournament-style selection procedure will concentrate around $\mu$. From here we arrive at a bound on our algorithm after performing at least $J^{\ast}$ steps of our algorithm where $J^{\ast}$ satisfies some entropic equation. Finally, we combine this with our lower bound to prove that we have attained the minimax rate. Technical lemmas for each of these results are displayed above the boxes.

\begin{figure}
    \centering
    \begin{tikzpicture}[
  node distance=0.75 cm and 3.5 cm,
  every node/.style={draw, rounded corners, align=center, font=\small, inner sep=5pt},
  >=latex
]

% ========================================================
% Global Background: Lower Bound / Upper Bound Line + Labels
% ========================================================
\draw[dashed, thick] (0cm, -3.5cm) -- (15cm, -3.5cm);
\node[font=\large\bfseries, draw=none, align=center] at (8cm, -3.1cm) {Lower Bound};
\node[font=\large\bfseries, draw=none, align=center] at (8cm, -3.9cm) {Upper Bound};

% ========================================================
% Header Nodes (Top Section: Gaussian / KS-SG / U-SG)
% ========================================================
\node[fill=gaussianfill, very thick, font=\large\bfseries, rounded corners=0pt] (gaussian_header) at (3cm,0) {\textbf{Gaussian}};
\node[fill=kssgfill, very thick, font=\large\bfseries, rounded corners=0pt] (kssg_header) at (8cm,0) {\textbf{KS-SG}};
\node[fill=usgfill, very thick, font=\large\bfseries, rounded corners=0pt] (usg_header) at (13cm,0) {\textbf{U-SG}};

% ========================================================
% Lower Bound Nodes (first row below dotted line)
% ========================================================

% --- Invisible node to center thm21 and lem22 together under Gaussian ---
\node[draw=none, inner sep=0pt] (lowerbound_center) at (3cm, -2cm) {}; % this is horizontally aligned with Gaussian header

\node[fill=gaussianfill, anchor=east] (thm21) [left=0.2cm of lowerbound_center] {Lemma \ref{lemma:lower:bound:first:version} (Use \\ Fano's Inequality)};
\node[fill=gaussianfill, anchor=west] (lem22) [right=0.2cm of lowerbound_center] {Lemma \ref{corruptions:lower:bound} (Adapt \\ Huber-model proof)};

\node[fill=usgfill] (lem23) at (13cm, -2cm) {Lemma \ref{lemma:subgaussian:lower:bound} (Mixture \\ of point mass + Gaussian)};

% --- Supporting lemmas for Thm 2.1 ---
\node[font=\scriptsize\itshape, draw=none, align=center, yshift=-1.5mm] (thm21_support) [above=0cm of thm21] {Supporting Lemmas:\\\ref{lemma:fano}, \ref{lemma:mutual:info}};

% ========================================================
% Upper Bound Nodes (Gaussian Path)
% ========================================================
\node[fill=gaussianfill] (thm36) [below=5cm of gaussian_header] {Thm. \ref{theorem:main:testing:result} (Gaussian \\ Testing Lemma)};
\node[font=\scriptsize\itshape, draw=none, align=center, yshift=-1.5mm] (thm36_support) [above=0cm of thm36] {Supporting Lemmas:\\ \ref{lemma:function:g:technical:details}, \ref{lemma:existence:constants}, \ref{lemma:cdf:convex}, \ref{lemma:gaussian:tail:bound}};

\node[fill=gaussianfill] (lem37) [below=of thm36] {Lemma \ref{lemma:tournament} \\ (Tournament Lemma)};
\node[draw=none, inner sep=0pt, minimum size=0pt] (helper37) [below=of lem37] {}; % helper node to split arrows cleanly

\node[fill=gaussianfill] (thm38) [below=of helper37] {Thm. \ref{theorem:gaussian:version} (Gaussian Alg.\\ Error Bound at $J^*$)};
\node[font=\scriptsize\itshape, draw=none, align=center, yshift=-1.5mm] (thm38_support) [above=0cm of thm38] {Supporting Lemmas:\\ \ref{lemma:set:complement:induction}, \ref{lemma:for:theorem:gaussian}, \ref{lemma:for:theorem:gaussian:part2}};

\node[fill=gaussianfill] (thm310) [below=of thm38] {Thm. \ref{theorem:robust:minimax:rate:attained:gaussian} (Gaussian \\ Minimax Rate)};

% ========================================================
% Upper Bound Nodes (KS-SG Path)
% ========================================================
\node[fill=kssgfill] (thm41) [below=5cm of kssg_header] {Thm. \ref{theorem:subgaussian:main:testing:result} (KS-SG \\ Testing Lemma)};
\node[font=\scriptsize\itshape, draw=none, align=center, yshift=-1.5mm] (thm41_support) [above=0cm of thm41] {Supporting Lemmas:\\ \ref{lemma:function:g:technical:details}, \ref{lemma:local:clt}, \ref{lemma:technical:subgaussian}};

\node[fill=kssgfill] (rem43_1) [below=of thm41] {Remark \ref{remark:expectation:over:R} (KS-SG \\ Tournament Lemma)};
\node[draw=none, inner sep=0pt, minimum size=0pt] (helper43) [below=of rem43_1] {}; % helper node

\node[fill=kssgfill] (rem43_2) [below=of helper43] {Remark \ref{remark:expectation:over:R} (KS-SG Alg.\\ Error Bound at $J^*$)};
\node[font=\scriptsize\itshape, draw=none, align=center] (thm43_2_support) [above=0cm of rem43_2] {\\}; % intentionally empty

\node[fill=kssgfill] (rem43_3) [below=of rem43_2] {Remark \ref{remark:expectation:over:R} (KS-SG \\ Minimax Rate)};

% ========================================================
% Upper Bound Nodes (U-SG Path)
% ========================================================
\node[fill=usgfill] (thm45) [below=5cm of usg_header] {Thm. \ref{theorem:asymmetric:testing:result} (U-SG \\ Testing Lemma)};
\node[font=\scriptsize\itshape, draw=none, align=center, yshift=-1.5mm] (thm45_support) [above=0cm of thm45] {Supporting Lemmas:\\ \ref{lemma:function:g:technical:details}, \ref{lemma:technical:constants:subgaussian:asymmetric:version2}, \ref{lemma:trimmed:mean}, \ref{lemma:subgaussian:asymmetric:varrho:case}};

\node[fill=usgfill] (lem46) [below=of thm45] {Lemma \ref{lemma:tournament:asymmetric} (U-SG \\ Tournament Lemma)};
\node[draw=none, inner sep=0pt, minimum size=0pt] (helper46) [below=of lem46] {}; % helper node

\node[fill=usgfill] (thm47) [below=of helper46] {Thm. \ref{theorem:general:subgaussian:version} (U-SG Alg.\\ Error Bound at $J^*$)};
\node[font=\scriptsize\itshape, draw=none, align=center, yshift=-1.5mm] (thm47_support) [above=0cm of thm47] {Supporting Lemmas:\\\ref{lemma:for:theorem:asymmetric}, \ref{lemma:for:theorem:asymmetric:part2}};

\node[fill=usgfill] (thm48) [below=of thm47] {Thm. \ref{theorem:robust:minimax:rate:attained:subgaussian} (U-SG \\ Minimax Rate)};

% ========================================================
% Arrows (Connections between Nodes)
% ========================================================

% --- Gaussian path arrows ---
\draw[->] (thm36) -- (lem37);
\draw[->] (lem37) -- (helper37);
\draw[->] (thm38) -- (thm310);

% --- KS-SG path arrows ---
\draw[->] (thm41) -- (rem43_1);
\draw[->] (rem43_1) -- (rem43_2);
\draw[->] (rem43_2) -- (rem43_3);

% --- U-SG path arrows ---
\draw[->] (thm45) -- (lem46);
\draw[->] (lem46) -- (helper46);
\draw[->] (thm47) -- (thm48);

% --- Lower bound horizontal connector (Thm. 2.1 to Lemma 2.2) ---
\draw[-] (thm21.east) -- (lem22.west);

\end{tikzpicture}
    \caption{We illustrate the overall structure of our proofs of here. We write KS-SG to mean `Known or Symmetric sub-Gaussian' and U-SG to mean `Unknown sub-Gaussian.' The unbounded scenario is not shown but follows a similar structure. Many of the KS-SG results are not formally stated, but we explain in Remark \ref{remark:expectation:over:R} how they follow from Theorem \ref{theorem:subgaussian:main:testing:result}. The supporting lemmas \ref{lemma:for:theorem:gaussian}, \ref{lemma:for:theorem:gaussian:part2}, \ref{lemma:for:theorem:asymmetric}, and \ref{lemma:for:theorem:asymmetric:part2} in the respective algorithm error bound theorem critically rely on properties of our tree construction (Algorithm \ref{algorithm:directed:tree}) stated in Lemmas \ref{lemma:pruned:tree:properties}, \ref{lemma:bound:level:J:intersect:ball:mu}, and \ref{lemma:cauchy:sequence}.}
    \label{fig:flowchart}
\end{figure}

\section{Lower Bounds} 
\label{section:robust_gsm:lower:bound}

In our first two lower bounds---Lemmas \ref{lemma:lower:bound:first:version} and \ref{corruptions:lower:bound} below, we assume that the noise is Gaussian, i.e., we assume $\xi \sim \cN(0, \sigma^2 \II)$. We start by establishing a lower bound for the uncorrupted setting, which, unlike \citet{neykov2022minimax}, takes into account the fact that we have $N$ observations. Of course, our bound is expected, since in the Gaussian model the mean of the observations is a sufficient statistic. Nevertheless, for completeness, we give a full proof using Fano's inequality.

\begin{lemma} \label{lemma:lower:bound:first:version} Let  $c$ be the constant from Definition \ref{definition:local:metric:entropy} of the local metric entropy, fixed sufficiently large. Then for any $\eta$ satisfying $\log \cMKloc(\eta, c) > 4\left(\frac{N \eta^2}{2\sigma^2} \vee \log 2\right)$, we have
\begin{align*}
     \inf_{\hat \mu} \sup_{\mu \in K} \sup_{\cC}\EE \|\hat \mu(\cC(\tilde X)) - \mu\|^2  \geq  \inf_{\hat \mu} \sup_{\mu \in K} \EE \|\hat \mu(\tilde X) - \mu\|^2\geq \frac{\eta^2}{8 c^2}.
\end{align*}
\end{lemma}

Next, we provide a lower bound which captures the fact that we have adversarially corrupted fraction of the observations.

\begin{lemma}\label{corruptions:lower:bound} Suppose $\epsilon \geq \tfrac{1}{\sqrt{N}}$. Then we have
    \begin{align*}
        \inf_{\hat \mu} \sup_{\mu \in K}\sup_{\cC} \EE_\mu \|\hat \mu(\cC(\tilde X)) - \mu\|^2 \gtrsim \epsilon^2 \sigma^2 \wedge d^2.
    \end{align*}
\end{lemma}

The proof of Lemma \ref{corruptions:lower:bound} utilizes Theorem 5.1 of \cite{chen2018robust}, which shows a similar lower bound in the Huber contamination model. We add a twist allowing for the corruption scheme to corrupt at most fraction $\epsilon$  of the observations (note that in the Huber model the corruption scheme may corrupt any number between $0$ and $N$ of the observations, although the average number of corruptions remains fixed at $\epsilon N$).

\subsection{Unknown Sub-Gaussian Noise Lower Bound}

In the unknown sub-Gaussian noise setting, our minimax rate requires taking a supremum over the entire class of sub-Gaussian distributions with parameter $\leq \sigma$. Thus, our minimax quantity of interest becomes \begin{equation} \label{def:unknown:sub:gaussian:minimax:rate}
    \inf_{\hat\mu} \sup_{\mu\in K}\sup_{\xi}\sup_{\cC}\EE_{\mu}\|\hat{\mu}(\cC(\tilde{X}))-\mu\|^2,
\end{equation} where $\xi$ ranges over all sub-Gaussian distributions with parameter less than or equal to a given $\sigma$. Since Gaussian random variables are themselves sub-Gaussian, the lower bound from Lemma \ref{lemma:lower:bound:first:version} still holds.

The following lemma shows that the adversarial rate in the case of unknown sub-Gaussian noise is at least $(\sigma\epsilon\sqrt{\log 1/\epsilon})^2 \wedge d^2$. \cite{mendelson_robust_mean} provide a similar argument with high probability for the unconstrained univariate case. 

\begin{lemma}[Sub-Gaussian Lower Bound] \label{lemma:subgaussian:lower:bound} For any $\epsilon\in(0,1/2)$, \begin{align*}
    \inf_{\hat \mu} \sup_{\mu \in K}\sup_{\xi} \sup_{\cC} \EE_\mu \|\hat \mu(\cC(\tilde X)) - \mu\|^2 \gtrsim \sigma^2\epsilon^2\log(1/\epsilon) \wedge d^2.
\end{align*}
\end{lemma}

\section{Upper Bound with i.i.d. Gaussian Noise} \label{section:robust_gsm:gaussian:upper:bound}

In this section, we establish an upper bound on the minimax rate with our new local packing set algorithm. We do so in several parts. First, in Section \ref{section:infinite:tree:construction}, we construct a directed tree whose (infinitely many) nodes correspond to points of $K$. Each level of the tree forms a progressively finer packing and covering of the set $K$. Then in Section \ref{section:robust:algorithm}, we devise an algorithm that traverses the tree using the observed data. Section \ref{section:gaussian:error:bounding} proves that in this Gaussian case the algorithm achieves the minimax rate.

\subsection{Constructing an Infinite Tree of Points in \texorpdfstring{$K$}{K}}
\label{section:infinite:tree:construction}

We construct a directed tree $G$ with countably many nodes that densely populate $K$ using a local packing procedure as in \citet{neykov2022minimax}. The authors discovered a mistake in a proof of the prior paper, however, which we resolve via a pruning procedure in the tree construction.

Given a node $u\in G$, we define the parent set $\cP(u)$ as the set of nodes $u'$ with a directed edge from $u'$ to $u$ ($u' \rightarrow u$). We refer to a node $u$ as an offspring of a parent node $v$ if $v\in\cP(u)$. For a node $v$ let $\cO(v)$ denote the set of all offspring of $v$, i.e., $\cO(v) = \{u: v \in \cP(u)\}$. We now describe the tree construction in plain English. For the convenience of the reader, we also summarize our construction in Algorithm \ref{algorithm:directed:tree}.

Begin with the root node (level 1) by picking any $\bar\nu\in K$. Now construct a maximal $d/c$-packing of $B(\bar\nu, d)\cap K=K$ and draw a directed edge from $\bar\nu$ to each of these points to form level 2. Recall $c$ is the constant appearing in the definition of local metric entropy. To construct level 3, for each node $u$ in level 2, construct a maximal $d/(4c)$-packing of $B(u, d/2)\cap K$ and again draw a directed edge from $u$ to the resulting points. Here begins our first pruning step, which we repeat at every subsequent level.

Lexicographically order the points of level 3, say $u_1^3, u_2^3,\dots, u_{M_3}^3$, noting that each of these points has a single parent and some of these parent nodes may be different. Let $\cU_3$ be an ordered list of unprocessed nodes at level 3 (for bookkeeping purposes), initialized as $\cU_3=[u_1^3, u_2^3,\dots, u_{M_3}^3]$. Consider the first element $u_1^3$ from this list. Let $\cT_3(u_1^3)=\{u_j^3\in\cU_3:\|u_j^3-u_1^3\|  \leq d/(4c), j\ne 1\}$, which is possibly empty. Note that if $u_j^3\in \cT_3$, then it cannot share the same parent as $u_1^3$, because the offspring of $\cP(u_1^3)$ form a $d/(4c)$-packing. 

For each $u_j^3\in\cT_3(u_1^3)$, remove the directed edge from $\cP(u_j^3)$ to $u_j^3$, and add a directed edge from $\cP(u_j^3)$ to $u_1^3$. That is, $u_j^3$ is now a point with no edges connected to it, and the edge from $\cP(u_j^3)$ instead connects to $u_1^3$. We then delete each disconnected node $u_j^3$ and update $\cU_3$ by removing $\{u_1^3\}\cup \cT_3(u_1^3)$ from the vector. 

If $\cU_3$ is still nonempty, take the new node at the start of $\cU_3$, say $u_{l}^3$, and again construct $\cT_3(u_l^3)=\{u_j^3\in\cU_3 :\|u_j^3-u_l^3\| \leq d/(4c), j\ne l\}$. Importantly, since we already removed nodes at this level that are `close' to $u_1^3$, $\{u_l^3\}\cup \cT_3(u_l^3)$ will not contain any already processed nodes, i.e., points that we removed from $\cU_3$. Again, for each $u_j^3\in\cT_3(u_l^3)$, remove the directed edge from $\cP(u_j^3)$ to $u_j^3$, and add a directed edge from $\cP(u_j^3)$ to $u_l^3$.  Moreover, delete the node $u_j^3$. Update $\cU_3$ by removing $\{u_l^3\}\cup \cT_3(u_l^3)$. After finitely many steps, $\cU_3$ is empty, and we say level 3 nodes are the set of offspring of nodes from level $2$, noting this excludes former offspring nodes whose directed edge pointing to them was removed.

Now we may proceed to constructing level 4. For any node $u$ from level 3, form a maximal $d/(8c)$-packing of $B(u, d/4)\cap K$, noting the halving radius in both packing distance and the ball, and again repeat the pruning. Whenever we take the first element $u_l^4$ of our list of unprocessed nodes $\cU_4$, we define $\cT_4(u_l^4)=\{u_j^4 \in\cU_4:\|u_j^3-u_l^3\| \leq d/(8c), j\ne l\}$.

In general, for a level $k\ge 3$, assuming that the previous level is level $2$ (no pruning required) or is pruned, for each point $u$ in level $k-1$,  form  $\tfrac{d}{2^{k-1} c}$-packing of $B(u, \tfrac{d}{2^{k-2}}) \cap K$ of points from level $k-1$ to form our preliminary points forming level $k$, with directed edges from $u$ to its associated packing set points. Lexicographically order the points at level $k$, say $u_1^k,\dots, u_{M_k}^k$. Form the ordered list $\cU_k=[u_1^k,\dots, u_{M_k}^k]$ of unprocessed nodes. While $\cU_k$ is non-empty, take the first element of $\cU_k$, say $u_l^k$. Construct $\cT_k(u_l^k)=\{u_j^k \in\cU_k:\|u_j^k-u_l^k\|\leq \tfrac{d}{2^{k-1} c}, j\ne l\}$, noting $u_l^k\cup\cT_k(u_l^k)$ will not contain any already processed nodes (that were removed from $\cU_k$). For each element $u_j^k\in\cT_k(u_l^k)$, remove the directed edge from $\cP(u_j^k)$ to $u_j^k$ and instead add one from $\cP(u_j^k)$ to $u_l^k$, and remove $u_j^k$ from the graph. Once done, remove $u_l^k\cup\cT_k(u_l^k)$ from $\cU_k$. Once $\cU_k$ is empty, the resulting offspring nodes of points from level $k-1$ form the pruned level $k$. The pruned graph is then obtained by performing this procedure for all $k\ge 3$. We summarize the procedure in more concise terms in Algorithm \ref{algorithm:directed:tree}.  For convenience, we define $\cL(k)$ as the set of nodes forming the pruned graph at level $k$. 

\begin{algorithm} 
\SetKwComment{Comment}{/* }{ */}
\caption{Directed Tree Construction \label{algorithm:directed:tree}}
\KwInput{Root node $\bar\nu \in K$, graph $G$ composed of single  node $\bar\nu$}
$\cL(1)=\{\bar\nu\}$\;
Add directed edge from $\bar\nu$ to the points of a maximal $(d/c)$-packing of $B(\bar\nu, d)\cap K$, and set $\cL(2)=\cO(\bar\nu)$\;
$k \gets 3$\;
\While{TRUE} {
   For each $u\in \cL(k-1)$, add directed edge from $u$ to new nodes forming a maximal $\tfrac{d}{2^{k-1}c}$-packing of $B(u,\tfrac{d}{2^{k-2}})\cap K$\;
   Let $\cU_k$ be  lexicographically ordered list of nodes added in previous step\;
   \While{$\cU_k$ is non-empty} {
    Pick first element, say $u_l^k$, of $\cU_k$\;
    Set $\cT_k(u_l^k)=\{u_j^k \in\cU_k:\|u_j^k-u_l^k\|\leq \tfrac{d}{2^{k-1} c}, j\ne l\}$\;
    For each $u_j^k\in\cT_k(u_l^k)$, remove directed edge  $\cP(u_j^k)\rightarrow u_j^k$ and node $u_j^k$ from $G$, add edge $\cP(u_j^k)\rightarrow u_l^k$ \;
    Remove $u_l^k\cup \cT_k(u_l^k)$ from $\cU_k$\;
   }
   Set $\cL(k)=\bigcup_{u\in \cL(k-1)}\cO(u)$\;
   $k\gets k+1$\;
}
\Return{$G$} 
\end{algorithm}

 The following lemma establishes some covering and packing properties of the pruned graph at each level. As a corollary, we conclude that the pruned graph densely covers $K$ with its nodes.

\begin{lemma}\label{lemma:pruned:tree:properties} Let $G$ be the pruned graph from above and assume $c>2$. Then for any $J\ge 3$, $\cL(J)$ forms a $\tfrac{d}{2^{J-2}c}$-covering of $K$ and a $\tfrac{d}{2^{J-1}c}$-packing of $K$. In addition, for each $J\ge 2$ and any parent node $\Upsilon_{J-1}$ at level $J-1$, its offspring $\cO(\Upsilon_{J-1})$ form a $\tfrac{d}{2^{J-2}c}$-covering of the set $B(\Upsilon_{J-1}, \tfrac{d}{2^{J-2}}) \cap K$. Furthermore, the cardinality of $\cO(\Upsilon_{J-1})$ is upper bounded by $\cMloc(\tfrac{d}{2^{J-2}}, 2c)$ for $J \geq 2$.
\end{lemma}

Later when we prove our algorithm achieves the minimax rate, we will repeatedly intersect with the event that the random variable $\Upsilon_{J-1}$ (picked from $\cL(J-1)$) lies in $B(\mu, \tfrac{d}{2^{J-2}})$. We would like to bound the number of choices of $\Upsilon_{J-1}$ by bounding the cardinality of $\cL(J-1)\cap B(\mu, \tfrac{d}{2^{J-2}})$, so we prove the following:

\begin{lemma} \label{lemma:bound:level:J:intersect:ball:mu} Pick any $\mu\in K$. Then for $J\ge 2$, $\cL(J-1)\cap B(\mu, \tfrac{d}{2^{J-2}})$ has cardinality upper bounded by $\cMloc(\tfrac{d}{2^{J-2}}, c)\leq \cMloc(\tfrac{d}{2^{J-2}}, 2c)$. 
\end{lemma}

When we apply our algorithm with our observed data, we traverse a particular infinite path of the directed tree. The following result ensures that for any such path, the corresponding points form a Cauchy sequence.

\begin{lemma} \label{lemma:cauchy:sequence}  Let $[\Upsilon_1, \Upsilon_2,\dots]$ be the nodes of an infinite path in the graph $G$, i.e., $\Upsilon_{J + 1} \in \cO(\Upsilon_{J})$ for any $J \in \NN$. Then for any integers $J\ge J'\ge 1$, $\|\Upsilon_{J'}-\Upsilon_J\|\le  \frac{d(2+4c)}{c 2^{J'}}$.
\end{lemma}

\subsection{Robust Algorithm}

\label{section:robust:algorithm}

We begin by specifying a `robust' hypothesis test that checks which of two fixed points is closer to more than half the data. This will let us specify how to define a `winner' between two points.

\begin{definition} \label{definition:psi:gaussian} Given an ordered pair $(\nu_1,\nu_2)$ of points $\nu_1,\nu_2\in\RR^n$, define the test $\psi_{\nu_1,\nu_2}$ by \begin{align*}
    \psi_{\nu_1,\nu_2}(\{X_i\}_{i \in [N]}) =\mathbbm{1}(|\{i \in [N]: \|X_i - \nu_1\| \geq \|X_i - \nu_2\|\}| \geq N/2).
\end{align*} We drop the subscripts and write $\psi$ when the context is clear.
\end{definition}

\begin{definition} \label{definition:domination:of:point} Assume points $\nu_1$ and $\nu_2$ in $\mathbb{R}^n$ are in lexicographic order. If $\psi_{\nu_1,\nu_2}(\{X_i\}_{i \in [N]}) =0$, then we say $\nu_1$ dominates $\nu_2$ and write $\nu_1\succ\nu_2$ (or $\nu_2\prec \nu_1$). If $\psi_{\nu_1,\nu_2}(\{X_i\}_{i \in [N]}) =1$, we  say $\nu_2$ dominates $\nu_1$ and write $\nu_2\succ \nu_1$ (or $\nu_1\prec\nu_2$).
\end{definition}

Thus, given two distinct points $\zeta_1$ and $\zeta_2$ in $K$, we may order them lexicographically to obtain their ordered equivalent $\nu_1$ and $\nu_2$ and decide which point dominates the other given our data. It is important to note that for any two points $\zeta_1, \zeta_2 \in K$, exactly one of $\zeta_1 \succ \zeta_2$ or $\zeta_2 \succ \zeta_1$ holds. However, note that it may be possible to have three points $\zeta_1,\zeta_2,\zeta_3 \in K$ satisfying $\zeta_1 \succ \zeta_2 \succ \zeta_3 \succ \zeta_1$. This will not be an issue for our algorithm. 

Next, instead of taking the minimizer in $\ell_2$ norm as in \citet{neykov2022minimax}, we will run a tournament between the points, an idea originally due to \cite{lecam1973convergence} and \cite{birge1983approximation} and well-described in the book of \citet[Chapter 32.2.2, page 615]{polyanskiyinformation}. At any point $\nu$, given a radius $\delta>0$ and finite set $S\subset K$, define
\begin{align} \label{eq:tournament:definition}
    T(\delta, \nu, S) = \begin{cases}\max_{\nu' \in S'} \|\nu-\nu'\| \mbox{ if } S'\ne\varnothing \\
    0\mbox{ if } S'=\varnothing,
    \end{cases}
\end{align} where we define $ S'=\{\nu'\in S: \nu \prec \nu' \mbox { and } \|\nu - \nu'\| \geq C \delta\}$. For the rest of this section, we will define the constant $C = c/2-1$ (i.e., $c=2(C+1)$), where $c$ is our (absolute) constant from Definition \ref{definition:local:metric:entropy} of the local metric entropy. We assume $C>2$ (implied by assuming $c>6$). 

In our algorithm, we will take $S$ to be the set of offspring of the node in our tree constituting the latest update. To intuitively interpret $T$, take a point $\nu$ and consider any point $\nu'$ belonging to $S$ not in an immediate neighborhood of $\nu$. Discard all those contenders except those that are better representations of the observed data, i.e., dominate $\nu$. Then $T(\delta, \nu, S)$ measures the worst-case deviation from such a superior contender. If no such contenders exist, we set the quantity to 0. By choosing the $\nu$ with a minimal $T(\delta,\nu, S)$, we find the point whose preferred alternatives (if they exist) are as close as possible.

We now present Algorithm \ref{algorithm:robust}. We will denote our initial input (the root of the tree) as $\Upsilon_1$, and call $\Upsilon_{k+1}$ the output of $k$ iterations of the algorithm. For each $k\ge 1$, we will choose $\Upsilon_{k+1}$ as a point among the offspring $\cO(\Upsilon_k)$ that minimizes our tournament function $T(\delta, {}\cdot{} , \cO(\Upsilon_k))$. For each $k$ we will take $\delta = \frac{d}{2^k(C+1)}$, where $C=c/2-1$ with $c>6$. Although this procedure is clearly data dependent, the actual construction of packing sets from Algorithm \ref{algorithm:directed:tree} is wholly independent of the data. A particular realization $[\Upsilon_1,\Upsilon_2,\dots]$ corresponds to traversing down a particular, data-informed infinite path of a data-independent directed tree. Moreover, Lemma \ref{lemma:cauchy:sequence} ensures that the resulting sequence $[\Upsilon_1,\Upsilon_2,\dots]$ of updates form a convergent Cauchy sequence.

\begin{algorithm} 
\SetKwComment{Comment}{/* }{ */}
\caption{Robust Upper Bound Algorithm \label{algorithm:robust}}
\KwInput{A point $\Upsilon_1 \in K$}
$k \gets 1$\;
$\Upsilon \gets [\Upsilon_1]$\;
\While{TRUE} {
    $\Upsilon_{k+1} \gets \argmin_{\nu \in \cO(\Upsilon_{k})} T\left(\tfrac{d}{2^{k}(C+1)}, \nu, \cO(\Upsilon_{k})\right)$ as defined in \eqref{eq:tournament:definition}\Comment*[r]{Lexicographically sort points in the minimization to resolve ties by picking the smallest point}
    $\Upsilon$.append$(\Upsilon_{k+1})$\;
    $k \gets k + 1$\;
}
\Return{$\Upsilon=[\Upsilon_1,\Upsilon_2,\dots]$} 
\end{algorithm}

\subsection{Bounding the Error of our Algorithm}
\label{section:gaussian:error:bounding}

In this subsection, we begin by proving via Theorem \ref{theorem:main:testing:result} an analogous result to \citet[Lemma II.5]{neykov2022minimax}, in which we bound the Type I error of a hypothesis test $\psi$ that evaluates whether $\mu$ belongs to one of two fixed non-intersecting $\delta$-balls. This will later let us prove each update of our algorithm will contain $\mu$ in a $\delta$-ball with high probability. Our test picks the ball whose center is closer to more than half of the data.

\begin{theorem} \label{theorem:main:testing:result} Consider the Gaussian noise setting. Assume that at most $\epsilon \leq 1/2 - \kappa$ fraction of the observations are corrupted where $\kappa \in (0,1/2]$ is a fixed constant. Suppose also we are testing $H_0: \mu \in B(\nu_1,\delta)$ versus $H_A: \mu \in B(\nu_2,\delta)$ for $\|\nu_1 - \nu_2\| \geq C \delta$ for some fixed $C > 2$ with $\nu_1,\nu_2\in K$, and $\delta\geq C_1(\kappa)\epsilon \sigma$ where $C_1=C_1(\kappa)$ is the fixed constant from the proof of Lemma \ref{lemma:existence:constants} depending on $C$ and $\kappa$ alone. Then the test $\psi=\psi_{\nu_1,\nu_2}(\{X_i\}_{i \in [N]})$ from Definition \ref{definition:psi:gaussian} satisfies
\begin{align*}
    \sup_{\mu: \|\mu- \nu_1\| \leq \delta}\PP_{\mu}(\psi = 1) \vee \sup_{\mu: \|\mu- \nu_2\| \leq \delta}\PP_{\mu}(\psi = 0) \leq \exp\bigg(- C_3(\kappa) N \frac{\delta^2}{\sigma^2}\bigg),
\end{align*}
where $C_3(\kappa)$ is a fixed constant depending on $C$ and $\kappa$ alone.
%where $\eta(C), \kappa(C)$ are some constants which depend on $C$ -- I think we can explicitely quantify what it they are. (I think they will also depend on $\epsilon$ the fraction of contaminated points).
\end{theorem}

The presence of arbitrarily large outliers causes the proof of Theorem \ref{theorem:main:testing:result} to require substantially more effort than that of Lemma II.5. Our technique is to recognize that $\psi$ measures the outcome of a binomial random variable with a probability parameter bounded by $1-\Phi(\delta/\sigma)$. We further bound this probability in two cases: when $\delta/\sigma$ is small, we use the convexity of $z\mapsto 1-\Phi(z)$ for $z \geq 0$, and when $\delta/\sigma$ is large, we use a tail bound for Gaussian random variables. Then in both cases, we use binomial concentration inequalities to bound the Type I error. We also need the existence of some constants depending only on $C$ and $\kappa$. These bounds and constants are given in Lemmas \ref{lemma:existence:constants}, \ref{lemma:cdf:convex},  and \ref{lemma:gaussian:tail:bound} in the appendix.

The next lemma establishes that given a $\delta$-covering of our set, we may run our tournament procedure \eqref{eq:tournament:definition} and contain $\mu$ in a ball of radius $\delta$ (up to scaling) with high probability.

\begin{lemma} \label{lemma:tournament} Let $S=\{\nu_1,\dots,\nu_M\}$ be a $\delta$-covering set of $K'\subseteq K$ with $\mu\in K'$ and $\delta \ge C_1(\kappa)\sigma \epsilon$. Let $i^* \in \argmin_i T(\delta, \nu_i, S)$.  Then \begin{align*}
    \PP( \|\nu_{i^*} - \mu\| \geq (C + 1) \delta) \leq M \exp( - \tfrac{C_3(\kappa) N \delta^2}{\sigma^2}).
\end{align*}
\end{lemma}
\begin{proof}
    Without loss of generality, let $\nu_1$ be the closest point to $\mu$. Since we have a $\delta$ covering of $K'$ we have $\|\nu_1 - \mu\| \leq \delta$. Denote $T_i=T(\delta,\nu_i, S)$. By definition, we have $\max(T_i,T_j) \geq \|\nu_i - \nu_j\|$ if $\|\nu_i - \nu_j\| \geq C \delta$. It follows that
    \begin{align*}
        \mathbbm{1}_{\|\nu_{i^*} - \nu_1\| \geq C \delta} \leq \mathbbm{1}_{\max(T_{i^*}, T_1) \geq C \delta} = \mathbbm{1}_{T_1 \geq C \delta}.
    \end{align*} Now if $T_1\ge C\delta$, this means there is some $\nu_j$ such that $\|\nu_1-\nu_j\|\ge C\delta$ and the test $\psi$ from Theorem \ref{theorem:main:testing:result} applied to $\nu_1$ and $\nu_j$ has value 1 since $\nu_j\succ \nu_1$. Since  $\delta > C_1(\kappa)\epsilon\sigma$ and $\|\nu_1-\mu\|\le\delta$, we apply the theorem along with a union bound over the $M-1\le M$ possible $\nu_j$ to obtain
    \begin{align*}
        \PP( \|\nu_{i^*} - \nu_1\| \geq C \delta) \leq \PP(T_1 \geq C\delta) \leq M \exp( -  \tfrac{C_3(\kappa) N\delta^2}{\sigma^2}).
    \end{align*} The proof is completed by the triangle inequality.
\end{proof}

Finally, we may upper bound the error of Algorithm \ref{algorithm:robust} with the following theorem. The spirit of the proof is close to that of \citet[Theorem 2.10]{neykov2022minimax} when there is no corruption. However, we must adapt the proof to handle the case when the $\delta \ge C_1(\kappa)\sigma\epsilon$ condition from Lemma \ref{lemma:tournament} fails within the first $J^{\ast}$ iterations, where we are setting $\delta = \frac{d}{2^{J-1}(C+1)}$ at any step $J$. The analogous non-corrupted setting had no such condition, by contrast. We handle this second case by separately considering whether $\epsilon \lesssim \tfrac{1}{\sqrt{N}}$ or $\epsilon\gtrsim \tfrac{1}{\sqrt{N}}$.

\begin{theorem} \label{theorem:gaussian:version} Consider the Gaussian noise setting. For any $J\ge 1$, define  $\eta_J = \frac{d\sqrt{C_3(\kappa)}}{2^{J-1}(C+1)}$ where $C_3(\kappa)$ is from Theorem \ref{theorem:main:testing:result}. Let $c=2(C+1)$ be the constant used in the local metric entropy. Let $J^{\ast}$ be the maximal $J$ such that \begin{equation} \label{eq:robust:theorem:condition}
    \frac{N\eta_J^2}{\sigma^2} >2\log \left[\cMloc\left(\tfrac{c\eta_J}{\sqrt{C_3(\kappa)}}, 2c\right)\right]^2\vee \log 2,
\end{equation} and set $J^{\ast}=1$ if this condition never occurs. Then if $\nu^{\ast\ast}(X_1,\dots,X_N)$ is the output of at least $J^{\ast}$ iterations of Algorithm \ref{algorithm:robust}, we have $\EE_{X} \|\nu^{\ast\ast}-\mu\|^2 \lesssim \max\left(\eta_{J^{\ast}}^2, \epsilon^2\sigma^2\right)\wedge d^2$. 
\end{theorem}

\begin{remark} \label{remark:changing:2c:to:c} In the definition of local metric entropy, we specify some constant, say $\tilde c$, so that $\cMloc(\eta,\tilde c)$ is computed using $\eta/\tilde c$ packings of balls in $K$. In the lower bound result of Lemma \ref{lemma:lower:bound:first:version}, we could have used $\tilde c = 2c$ rather than $c$, and the resulting bound would be unchanged except by an absolute constant. The local metric entropy parameter in the lower bound can thus be chosen to match the $2c$ appearing in  \eqref{eq:robust:theorem:condition} of Theorem \ref{theorem:gaussian:version}. As a consequence, without loss of generality, we can assume that the same sufficiently large constant $c$ appears in both the lower and upper bound, and replace \eqref{eq:robust:theorem:condition} with the following: Let $J^{\ast}$ be the maximal integer such that \begin{equation} \label{eq:robust:condition:rescaled}
    \frac{N\eta_J^2}{\sigma^2} >2\log \left[\cMloc\left(\tfrac{c\eta_J}{2\sqrt{C_3(\kappa)}}, c\right)\right]^2\vee \log 2,
\end{equation} with $J^{\ast}=1$ if this never occurs.
\end{remark} 

Next, the following theorem adapts
 \citet[Theorem 2.11]{neykov2022minimax} to handle the corrupted data setting. We note that its proof is more involved than its counterpart from \citet{neykov2022minimax} as we need to consider four cases: two for $\{\eta^* \gtreqless \tfrac{1}{\sqrt{N}}\}$ and two for $\{\epsilon\gtreqless \tfrac{1}{\sqrt{N}}\}$.

\begin{theorem}[Robust Gaussian Minimax Rate] \label{theorem:robust:minimax:rate:attained:gaussian} Consider the Gaussian noise setting. Define $\eta^{\ast} = \sup\{\eta\ge 0:N\eta^2/\sigma^2\le \log \cMloc(\eta,c)\}$. Then, for  $c$ taken sufficiently large, the minimax rate is $\max({\eta^{\ast}}^2\wedge d^2, \epsilon^2\sigma^2\wedge d^2)$.
\end{theorem}

\begin{remark} \label{remark:lower:bound:from:lse} In \cite[Lemma 1.4]{prasadan2024some}\footnotemark{}, the authors show that $\eta^{*2} \wedge d^2 \gtrsim \tfrac{\sigma^2}{N} \wedge d^2$. This implies that when $\epsilon \lesssim \tfrac{1}{\sqrt{N}}$ the minimax rate is unaffected by the presence of the outliers which is rather remarkable! 
\end{remark}

\footnotetext{Technically the proof of \cite[Lemma 1.4]{prasadan2024some} requires picking many points along a diameter of a convex set $K$; however, in our star-shaped setting, this may not be possible. However, by Lemma \ref{lemma:star:shaped:has:line:segment}, we can just take a line segment of length $d/3$ in $K$ instead of $d$ and the proof carries through.}
\section{Sub-Gaussian case}
\label{section:robust_gsm:subgaussian:upper:bound}

In this section, we relax the Gaussian assumption on the noise vector from the previous section. We will first begin with a sign-symmetric assumption on the noise term as a warm-up before proceeding to the more challenging case when the noise is not necessarily sign-symmetric. Notably, if the noise distribution is not symmetric but is known to us, we can simply symmetrize the noise term by subtracting an independent copy of it and apply the sign-symmetric result (see also Remark \ref{remark:known:noise:to:symmetric}). 

After that, we tackle the problem of unknown sub-Gaussian noise. We demonstrate that the minimax rate for this problem is slightly slower than the known noise (or symmetric noise) counterpart. In both types of sub-Gaussian noise settings, we add independent multivariate Gaussian variables with covariance $\sigma^2\II$ to each observations. That is, we draw $R_i\sim \cN(0,\sigma^2\II)$ and then observe $X_i + R_i$ for $1\le i\le N$ where $X_1,\dots,X_N$ are the possibly corrupted data points. This operation ``smooths'' the distribution of the noise variables (at least the ones that are not corrupted), while maintaining that the sub-Gaussian parameter is still of order $\sigma$. In the sign-symmetric or known sub-Gaussian case, this enables us to use a local central limit theorem. For the unknown sub-Gaussian case, we do so as we use the results of \citet{mendelson_robust_mean}, who also add Gaussian noise; the authors remark, however, this is not strictly necessary and their proofs can be adapted. 

We will require knowledge of $\sigma$ or at least an upper bound in either sub-Gaussian setting. Additionally, in the unknown sub-Gaussian setting, we will need knowledge of $\epsilon$. For convenience, we will reset the definitions of our fixed constants  $C_1, C_2,\dots$ along with $\alpha,\beta, L, \varrho$ in both Section \ref{subsection:symmetric:subgaussian:upper:bound} and \ref{subsection:assymmetric:subgaussian:upper:bound}. When possible, we will match the context of their usage from the Gaussian noise setting. 

\subsection{Sign-Symmetric/Known Noise}

\label{subsection:symmetric:subgaussian:upper:bound}

In this subsection, we consider a symmetric sub-Gaussian setting with parameter $\sigma$ and demonstrate that the $\epsilon\sigma$ term in the minimax rate remains unchanged, a phenomenon also arising in
\citet{novikov_gleb_stefan_2023}.  However, the modified algorithm will have a lower breakdown point with respect to the corruption rate. This is because we will group then average our observations, so that the `new' observations have an effectively higher corruption rate. As a consequence, we can no longer tolerate a corruption rate of $\epsilon = 1/2-\kappa$ for arbitrarily small $\kappa$, but we can still sustain a constant fraction of observations being corrupted. 

Suppose we draw (but do not observe) $N$ data points of the form $\tilde X_i = \mu + \xi_i$ where $\xi_i$ are mean $0$, i.i.d. sign-symmetric sub-Gaussian vectors with parameter $\sigma$ in $\RR^n$. We only know that at most $\epsilon$ of the points have been arbitrarily corrupted after an adversary has been allowed to inspect all the original data points $\{\tilde X_i\}_{i \in [N]}$ resulting in $\{X_i\}_{i \in [N]}$, which is what we actually observe. Moreover, $\epsilon$ may be unknown to us. Suppose we know the true value of $\sigma$ or we have an accurate upper bound on it (if we are in the latter case, with a slight abuse of notation, we call the upper bound $\sigma$). We then draw $R_1,\dots, R_N\sim \cN(0,\sigma^2\II)$ and observe $X_1+R_1,\dots,X_N+R_N$.

Next, we group the observations in groups $G_j$ of cardinality $k$ where $k$ is fixed but sufficiently large. Note that the distribution $\tilde X_i + R_i - \mu$ is symmetric by construction. Also, for any uncorrupted group $G_j$ with $j \in [N/k]$, the random variable $k^{-1}\sum_{i \in G_j} (\tilde X_i + R_i - \mu)$ is symmetric and sub-Gaussian with parameter of order $\sigma / \sqrt{k} \asymp \sigma$. 

This leads us to the analogue of Theorem \ref{theorem:main:testing:result} in Theorem \ref{theorem:subgaussian:main:testing:result}, in which we test whether $\mu$ belongs to one of two non-intersecting balls of radius $\delta$ when $\delta/\sigma\gtrsim \epsilon$. However, we now assume $\epsilon\le \gamma/k$ for some $\gamma\in(0,1)$, and we compare whether $\nu_1$ or $\nu_2$ is closer to more than half of the new centers (following their convolution with the $R_i$ and averaging). When $\sqrt{k}\delta/\sigma$ is sufficiently small, we will bound the Type I error using a local central limit theorem presented in \citet[Chapter VII, Theorem 10]{petrov1976sums} (see Lemma \ref{lemma:local:clt}). Adding Gaussian noise and then taking $k$ sufficiently large ensures a density related to our test statistic exists and is lower bounded away from $0$. When $\sqrt{k}\delta/\sigma$ is large, we use a sub-Gaussian tail bound (similarly to the Gaussian case).  The choice of $k$, $\gamma$, and other constants are formalized in Lemma \ref{lemma:technical:subgaussian}.

\begin{theorem} \label{theorem:subgaussian:main:testing:result} Assume that at most $\epsilon \leq \gamma/k$ fraction of the observations are corrupted  where $k\in\NN$ and $\gamma\in(0,1)$ are absolute constants chosen in Lemma \ref{lemma:technical:subgaussian}. Suppose also we are testing $H_0: \mu \in B(\nu_1,\delta)$ versus $H_A: \mu \in B(\nu_2,\delta)$ for $\|\nu_1 - \nu_2\| \geq C \delta$ for some fixed $C > 2$, and $\delta\geq C_1\epsilon \sigma$ where $C_1$ is the fixed constant from the proof of Lemma \ref{lemma:technical:subgaussian}. Then the test \begin{align*}
    \psi_{\nu_1,\nu_2}(\{X_i\}_{i=1}^N) &= \mathbbm{1}\bigg(\bigg|\Big\{j \in [\tfrac{N}{k}]: \big\|k^{-1}\sum_{i\in G_j}(X_i + R_i) - \nu_1\big\|  \\ &\qquad \geq \big\|k^{-1}\sum_{i\in G_j}(X_i + R_i) - \nu_2\big\|\Big\}\bigg| \geq \tfrac{N}{2k}\bigg) 
\end{align*} satisfies
\begin{align*}
    \sup_{\mu: \|\mu- \nu_1\| \leq \delta}\PP_{\mu}(\psi = 1) \vee \sup_{\mu: \|\mu- \nu_2\| \leq \delta}\PP_{\mu}(\psi = 0) \leq \exp\bigg(- C_3 N \frac{\delta^2}{\sigma^2}\bigg)
\end{align*}
where $C_3$ is a fixed absolute constant (that may depend on the constants $k$ and $C$).
\end{theorem}
\begin{remark}\label{remark:known:noise:to:symmetric}
If the noise distribution is not necessarily sign-symmetric but known, it can clearly be symmetrized by subtracting a variable generated from the true noise distribution. Hence the above result, modulo the remark below, can also be applied in case of known noise.
\end{remark}

\begin{remark} \label{remark:expectation:over:R}
Since our bounds in Theorem \ref{theorem:gaussian:version} and Theorem \ref{theorem:robust:minimax:rate:attained:gaussian} do not depend on the assumption on Gaussian noise, but simply on the bound on the Type I error of $\psi$ as in Theorem \ref{theorem:subgaussian:main:testing:result}, they immediately extend to our estimator $\nu^*$. However, due to our convolution with Gaussian noise, the resulting bound will actually be on $\EE_R\EE_X\|\nu^{\ast\ast}-\mu\|^2$ where $\nu^{\ast\ast}$ is the output of at least $J^{\ast}$ steps. But by Jensen's inequality, we will have \begin{align*}
    \EE_X \|\EE_R \nu^{\ast\ast} - \mu\|^2\leq \EE_R \EE_X \|\nu^{\ast\ast} - \mu\|^2 \lesssim \max\left(\eta_{J^{\ast}}^2, \epsilon^2\sigma^2\right)\wedge d^2,
\end{align*} meaning our final estimator is $\hat\mu =\EE_R\nu^{\ast\ast}$ rather than $\hat\mu = \nu^{\ast\ast}$ as in the Gaussian case.
\end{remark}

\subsection{Unknown sub-Gaussian Noise}
\label{subsection:assymmetric:subgaussian:upper:bound}

Let us return to our original test $\psi$ from the Gaussian case (Definition \ref{definition:psi:gaussian}), and observe that $\psi$ can equivalently be viewed as computing the median (at least when $N$ is odd) of $V_1,\dots, V_N$ where we set $V_i=\tfrac{\|X_i-\nu_1\|^2-\|X_i-\nu_2\|^2}{\|\nu_2-\nu_1\|}$ (recalling $X_i$ is possibly corrupted), and comparing this quantity with $0$. That is, for odd $N$ we have \begin{align*}
    \psi_{\nu_1,\nu_2}(\{X_i\}_{i \in [N]}) &= \mathbbm{1}(|\{i \in [N]: \|X_i - \nu_1\| \geq \|X_i - \nu_2\|\}| \geq N/2) \\
     &= \mathbbm{1}(\mathrm{Median}(V_1,\dots,V_N) \geq 0). 
\end{align*} We then bounded the Type I error of this test. The use of a median-like quantity allowed us to robustly define a notion of points dominating other points (if closer to more than half of the data), so that we could capture the true mean $\mu$ in a $\delta$-ball with high probability (see Lemma \ref{lemma:tournament}).

But rather than using a median, we could use a different robust estimator of the mean (in one-dimension) to apply to our $V_1,\dots,V_N$, and then adapt our definition of what it means for $\nu\prec\nu'$, i.e., $\nu'$ dominating $\nu$. We therefore consider the trimmed mean estimator from \citet{mendelson_robust_mean}, in which we trim the tails of the data (hopefully the corrupted data) by using quantiles (which are functions of the corruption level and the desired confidence level). Depending on the size of $\delta$, we either define $\psi$ as we did in the Gaussian case (a median-like comparison) or as when $\mathrm{TM}_{\delta_0}(\{V_i\}_{i\in [N]})>0$, where we write $\mathrm{TM}_{\delta_0}$ to represent the trimmed mean with desired Type I error bound $\delta_0$, defined in \citet{mendelson_robust_mean} and again below. Theorem 1 of their paper demonstrates that $\mathrm{TM}_{\delta_0}(\{V_i\}_{i\in [N]})$ is close to the uncorrupted mean of the $V_i$ with high probability. 

We now reset all our absolute constants from both the Gaussian and sign-symmetric sub-Gaussian cases. We assume without loss of generality\footnote{Let us explain why this assumption is not binding if our data has an odd sample size. Assume for a moment that $\epsilon\le 1/64$. First, suppose $N\le 64$. If $N = 1$, then our assumption on the corruption rate forces $\epsilon=0$. In that case, an analogue to Theorem \ref{theorem:asymmetric:testing:result} holds by Lemma II.5 (and the sub-Gaussian remark after it) of \cite{neykov2022minimax}. If $1< N< 64$ and $N$ is odd, then we must again have $\epsilon=0$. We can thus return to the $N = 1$ case by averaging all of the observations. Suppose $N\ge 64$ but is odd. Then we can always add one sample at random (e.g., the null vector), at the expense of increasing the corruption rate to $(\epsilon N + 1)/(N+1) < 1/64 + 1/64 = 1/32$, as we required. Hence our theory will still go through.} that the sample size is now $2N$ (instead of $N$ as in the previous two sections). Assume at most fraction $\epsilon\in[0,1/32)$ of the $2N$ samples $\tilde X_i=\mu+\xi_i$ are corrupted and denote by $X_i$ the possibly corrupted data. Here $\epsilon$ must be known in order to use the trimmed mean estimator. Next, following the suggestion of \citet{mendelson_robust_mean} as well as satisfying the assumption of their theorem, we again add in Gaussian noise $R_i\sim\cN(0,\sigma^2\II)$ to each of the $X_i$, to ensure the resulting random variable has a density. However, the authors do note that a density is not required for their theorem to hold, but simply the fact that the variable is sub-Gaussian. By contrast, the reason we added the $R_i$ in the previous known or symmetric sub-Gaussian case was to justify a local central limit theorem application on the averages of the $X_i + R_i$. As in the symmetric or known sub-Gaussian case, this requires knowledge of an upper bound of $\sigma$ up to constants, but even if we did not add the $R_i$ we would need such knowledge to apply our robust algorithm (as is explained later in this section).

Before we proceed to define the trimmed mean estimator, let us verify the sub-Gaussian properties of the $V_i$. Consider testing whether $\mu$ belongs to a $\delta$-ball around $\nu_1$ versus $\nu_2$ when $\nu_1$ and $\nu_2$ are $C\delta$-separated. Following the computations from the proof of \citet[Lemma II.5]{neykov2022minimax}, if we assume $\|\nu_2-\nu_1\|\ge C\delta$ and $\|\mu-\nu_1\|\le \delta$ and denote  $v = (\nu_2-\nu_1)/\|\nu_2-\nu_1\|$, then an uncorrupted $V_i$ satisfies \begin{align} 
    V_i &= \tfrac{\|\tilde X_i+R_i-\nu_1\|^2-\|\tilde X_i+R_i-\nu_2\|^2}{\|\nu_2-\nu_1\|} = \tfrac{2 (\tilde X_i + R_i)^T(\nu_2 - \nu_1) + \|\nu_1\|^2 - \|\nu_2\|^2}{\|\nu_1 - \nu_2\|} \notag \\
    &= 2 (\tilde X_i + R_i-\mu)^Tv +2\mu^T v + \tfrac{\|\nu_1\|^2 - \|\nu_2\|^2}{\|\nu_1 - \nu_2\|} \notag\\
     &= 2 (\tilde X_i + R_i-\mu)^Tv +2(\mu-\nu_1)^Tv + \tfrac{\|\nu_1\|^2 - \|\nu_2\|^2+2\nu_1^T(\nu_2-\nu_1)}{\|\nu_1 - \nu_2\|} \label{eq:V_i:decomposition:before:inequality}\\
     &\le  2 (\tilde X_i + R_i-\mu)^Tv +2\|\mu-\nu_1\| - \|\nu_2-\nu_1\| \notag\\
     &\leq2(\tilde X_i+R_i-\mu)^Tv + (-1+ 2/C)\|\nu_2-\nu_1\| \notag \\
     &\le 2(\tilde X_i+R_i-\mu)^Tv - (C-2)\delta.\label{eq:V_i:decomposition}
\end{align} The second to last line used $\|\mu-\nu_1\|\le\delta\le C^{-1}\|\nu_2-\nu_1\|$. Inspecting \eqref{eq:V_i:decomposition:before:inequality} we see $V_i -\EE  V_i = 2(\tilde X_i+R_i-\mu)^Tv$ is sub-Gaussian with parameter $2\sqrt{2}\sigma$, and \eqref{eq:V_i:decomposition} demonstrates that the mean $\EE V_i$ is upper-bounded by $-(C-2)\delta.$ In the notation of \citet{mendelson_robust_mean}, the authors write $\sigma_X^2$ as the variance of our $V_i$, so by properties of sub-Gaussian random variables we have $\sigma_X \le 2\sqrt{2}\sigma$.

 We will set aside $V_{N+1},\dots,V_{2N}$ to compute quantiles and use $V_1,\dots,V_N$ for our new test statistic. Set $\delta_0 = \exp\left(-C_3 N\delta^2/\sigma^2\right)$, our desired Type I error bound, where $C_3>0$ is some absolute constant to be determined later. Set $\tilde\epsilon = 8\epsilon + 12\tfrac{\log(4/\delta_0)}{N}$. Now sort $V_{N+1},\dots, V_{2N}$ and let $q_1$ be the $\tilde\epsilon$-quantile of the sorted points and $q_2$ the $(1-\tilde\epsilon)$-quantile. Later we explain how we ensure $\tilde\epsilon\in(0,1/2)$ so that it is indeed a valid probability and $\tilde\epsilon < 1-\tilde\epsilon$. For $a<b$, define \[\phi_{a,b}(x) = \begin{cases}
    a & \text{ if } x<a, \\
 
     x & \text{ if } x\in[a,b], \\
    b & \text{ if } x>b.
\end{cases}\] The trimmed mean estimator is then defined by \[\mathrm{TM}_{\delta_0}(\{V_i\}_{i=1}^{2N})=N^{-1}\sum_{i=1}^N \phi_{q_1,q_2}(V_i),\] where the subscript $\delta_0$ indicates the dependence of $q_i$ on our $\sigma$ and $\delta$.

The authors require $\delta_0\ge e^{-N}/4$ in their theorem. In our setting, that corresponds to requiring $\delta^2/\sigma^2\le C_3^{-1}\left(1+N^{-1}\log 4\right)$, so it suffices to require $\delta^2/\sigma^2\le C_3^{-1}$. We will actually require a slightly refined bound, namely, $\delta^2/\sigma^2\le D_3^{-1}/4$ where $D_3 = 8 + 12\log 4 + 12 C_3$, noting that $D_3^{-1}/4\le C_3^{-1}$. Furthermore, we assume $\delta^2/\sigma^2 \ge 1/N$ (otherwise, $\delta$ is smaller than the minimax rate). Although this assumption is required in our Type I error bound in Theorem \ref{theorem:asymmetric:testing:result}, we explain in Theorem \ref{theorem:general:subgaussian:version} why it is not restrictive.

Let us ensure $\tilde\epsilon\in(0,1/2)$. Clearly, it is positive since $\delta_0<1$ (so that $\log(4/\delta_0)>0$). Note that $D_3\ge D_3-8>0$, so $\delta^2/\sigma^2\le D_3^{-1}/4\le (D_3-8)^{-1}/4$. Using $N^{-1}\le \delta^2/\sigma^2$ and this fact, we have \begin{align*}
    \tilde\epsilon &= 8\epsilon + \tfrac{12\log 4}{N}+\tfrac{12 C_3\delta^2}{\sigma^2} \le 8\epsilon + \tfrac{\delta^2}{\sigma^2}\left[12\log 4 + 12 C_3\right] \\
    &= 8\epsilon + \tfrac{\delta^2}{\sigma^2}\cdot (D_3-8) \le 8\epsilon + \tfrac{(D_3-8)^{-1}}{4}\cdot (D_3-8) \\ &= 8\epsilon + \tfrac{1}{4}.
\end{align*} It therefore suffices to require $\epsilon<1/32$ to have $\tilde\epsilon\in(0,1/2)$.

Thus, when $\delta^2/\sigma^2 \le D_3^{-1}/4$, we define our test $\psi$ in terms of the trimmed mean estimator, and when $\delta^2/\sigma^2\ge D_3^{-1}/4$, we will use our $\psi$ from before (the Type I error bound in this case does not require $\delta^2/\sigma^2\ge 1/N$). Note the dependence of $\psi$ on $\delta$ now.

\begin{definition} \label{eq:psi:definition:asymmetric} Let  $(\nu_1,\nu_2)$ be an ordered pair of points $\nu_1,\nu_2\in\RR^n$, and draw observations $\{X_i\}_{i=1}^{2N}$, $\{R_i\}_{i=1}^{2N}$. Set $V_i=\frac{\|X_i-\nu_1\|^2-\|X_i-\nu_2\|^2}{\|\nu_2-\nu_1\|}$ for each $1\le i \le 2N$. Let $\delta>0$. Then define the test $\psi_{\nu_1,\nu_2,\delta}(\{X_i\}_{i=1}^{2N})$ by
    \begin{align*} 
    \begin{cases}
    \mathbbm{1}(\mathrm{TM}_{\delta_0}(\{V_i\}_{i=1}^{2N})>0) &  \text{ if } \tfrac{\delta^2}{\sigma^2}\le \tfrac{D_3^{-1}}{4}\\
    \mathbbm{1}(|\{i \in [2N]: \|X_i + R_i - \nu_1\| \geq \|X_i + R_i - \nu_2\|\}| \geq N)  & \text{ if } \tfrac{\delta^2}{\sigma^2} > \tfrac{D_3^{-1}}{4}.
\end{cases}
\end{align*} 
\end{definition} 

We may still use  Definition \ref{definition:domination:of:point} to define the notion of points dominating each other, but we are now using a different test $\psi_{\nu_1,\nu_2,\delta}$. Consequently, the definition of $\succ$ also depends on $\delta$. Similarly,  our tournament definition in \eqref{eq:tournament:definition} is unchanged except we use a more complex definition of $\succ$. Thus, our algorithm gets a new update rule once $\delta$ is sufficiently small, recalling we take $\delta$ to be of the form $\tfrac{d}{2^{J-1}(C+1)}$ for each $J$. Moreover, we will require $\delta\gtrsim \epsilon\sigma\sqrt{\log(1/\epsilon)}$, a change from the Gaussian or sign-symmetric sub-Gaussian case where we required $\delta\gtrsim\epsilon\sigma$.

At this point the reader may ask why the trimmed mean estimator is necessary in the first place, or why we can't only use the trimmed mean estimator and have a switching condition. The reason boils down to the use of symmetry in the proof of our testing theorem. In the Gaussian testing lemma (Theorem \ref{theorem:main:testing:result}), we had defined an event $A_J$ to be the event the uncorrupted $\tilde X_i$ is closer to $\nu_2$ than $\nu_1$. For small $\delta/\sigma$, we prove that $\PP_{\mu}(A_J)\le 1/2 - \Omega(\delta/\sigma)$, where $\Omega(\cdot)$ omits some constants. We obtain a similar bound for the sign-symmetric or known sub-Gaussian case. This crucially involved symmetry of the noise distribution, or at least the ability to symmetrize it. In the unknown sub-Gaussian case, we lose this ability, so we bypass the requirement using the trimmed mean estimator. But the trimmed mean estimator also requires $\delta/\sigma$ to be small, so we continue to use our median-like estimator when $\delta/\sigma$ is large.

We now proceed to our results in this noise setting. Applying Lemmas \ref{lemma:trimmed:mean} and \ref{lemma:subgaussian:asymmetric:varrho:case} from the appendix, we obtain our main Type I error bound. 

\begin{theorem} \label{theorem:asymmetric:testing:result} There exist sufficiently large absolute constants $C>2,C_1>0,C_3>0, C_5>0$ with the following properties. Let $\sigma>0$ and $\epsilon\in(0,1/32)$. Suppose we are testing $H_0: \mu \in B(\nu_1,\delta)$ versus $H_A: \mu \in B(\nu_2,\delta)$ for $\|\nu_1 - \nu_2\| \geq C \delta$. Let $\psi=\psi_{\nu_1,\nu_2}$ be defined as in Definition \ref{eq:psi:definition:asymmetric}. Suppose $\|\nu_1-\nu_2\|\ge C\delta$ for $\nu_1,\nu_2\in K$. Suppose $\delta>0$ is such $\delta/\sigma\ge C_1\epsilon\sqrt{\log(1/\epsilon)}$ and additionally that $\delta^2/\sigma^2\ge N^{-1}$. Then \begin{align*}
    \sup_{\mu: \|\mu- \nu_1\| \leq \delta}\PP_{\mu}(\psi = 1) \vee \sup_{\mu: \|\mu- \nu_2\| \leq \delta}\PP_{\mu}(\psi = 0) \leq \exp\bigg(-  \frac{C_5 N\delta^2}{\sigma^2}\bigg).
\end{align*}
\end{theorem}

This leads us to the analogue of Lemma \ref{lemma:tournament} where we simply swap out the $\delta\ge C_1\sigma\epsilon$ condition with $\delta \ge C_1\sigma \epsilon\sqrt{\log(1/\epsilon)}$ and include an extra constraint. The proof is identical, so we omit it.

\begin{lemma} \label{lemma:tournament:asymmetric} Let $C_1, C_5$ be given from Theorem \ref{theorem:asymmetric:testing:result}. Let $S=\{\nu_1,\dots,\nu_M\}$ be a $\delta$-covering set of $K'\subseteq K$ with $\mu\in K'$ and $\delta \ge C_1\sigma \epsilon\sqrt{\log(1/\epsilon)}$.  Assume also that $\delta^2/\sigma^2\ge N^{-1}$. Let $i^* \in \argmin_i T(\delta,\nu_i, S)$.  Then \begin{align*}
    \PP( \|\nu_{i^*} - \mu\| \geq (C + 1) \delta) \leq M \exp\bigg( -  \frac{C_5 N\delta^2}{\sigma^2}\bigg).
\end{align*}
\end{lemma}

Thus, we arrive at Theorem \ref{theorem:general:subgaussian:version} which upper bounds the error of our algorithm after at least $J^{\ast}$ steps. The procedure is very similar to proof of Theorem \ref{theorem:gaussian:version}, except now we have a more complex breakdown of cases. Although we still apply Algorithm \ref{algorithm:robust}, note that our tournament selection procedure $T$ changes after sufficiently many iterations.

\begin{theorem} \label{theorem:general:subgaussian:version} For any $J\ge 1$, define $\eta_J = \frac{d\sqrt{C_5}}{2^{J-1}(C+1)}$ where $C_5$ is from Theorem \ref{theorem:asymmetric:testing:result}. Let $c=2(C+1)$ be the constant used in the local metric entropy. Let $J^{\ast}$ be the maximal $J$ such that \begin{equation} \label{eq:asymmetric:robust:theorem:condition}
    \frac{N\eta_J^2}{\sigma^2} > 2\log \left[\cMloc\left(\frac{c\eta_J}{\sqrt{C_5}}, 2c\right)\right]^2\vee \log 2,
\end{equation} and set $J^{\ast}=1$ if this condition never occurs. Then if $\nu^{\ast\ast}(X_1,\dots,X_{2N})$ is the output of at least $J^{\ast}$ iterations of Algorithm \ref{algorithm:robust}, we have $\EE_R\EE_{X} \|\nu^{\ast\ast}-\mu\|^2 \lesssim \max\left(\eta_{J^{\ast}}^2, \epsilon^2\log(1/\epsilon)\sigma^2\right)\wedge d^2$. 
\end{theorem}

We now demonstrate that our algorithm attains the minimax rate. In the previous Gaussian version, we had to handle cases depending on whether $\epsilon\gtrsim \tfrac{1}{\sqrt{N}}$. The proof will now simplify since we already established in Section \ref{section:robust_gsm:lower:bound} that the minimax rate in the unknown sub-Gaussian case is $\gtrsim \sigma^2\epsilon^2\log(1/\epsilon) \wedge d^2$ for all $\epsilon\in(0,1/2)$. 
 We skim over some details due to similarities with Theorem \ref{theorem:robust:minimax:rate:attained:gaussian}. Recall also that Remark \ref{remark:changing:2c:to:c} allows us to replace the $2c$ in \eqref{eq:asymmetric:robust:theorem:condition} of Theorem \ref{theorem:general:subgaussian:version} with $c$.

 \begin{theorem}[Robust Sub-Gaussian Minimax Rate] \label{theorem:robust:minimax:rate:attained:subgaussian} Let $\epsilon\in(0,1/32)$. Define \[\eta^{\ast} = \sup\{\eta\ge 0:N\eta^2/\sigma^2\le \log \cMloc(\eta,c)\}.\] Then for   sufficiently large $c$, the minimax rate is $\max({\eta^{\ast}}^2, \sigma^2\epsilon^2\log(1/\epsilon)) \wedge d^2$.
\end{theorem}

\section{Extension to unbounded sets}
\label{section:robust_gsm:unbounded}

We now assume $K$ is an unbounded star-shaped set and that $\sigma$ (or a valid upper bound of $\sigma$ which we can call $\sigma$ with a slight abuse of notation) is known. We initially consider the Gaussian and sub-Gaussian case simultaneously before splitting into cases. Recall in the bounded case that we did not need knowledge of $\epsilon$ for the Gaussian or symmetric or known sub-Gaussian noise setting, but we did for the unknown sub-Gaussian case. In the unbounded case, either distributional assumption will require knowledge of $\epsilon$.

 \subsection{Lower Bound}
 \label{subsection:lower:bound:unbounded}
By the argument in \citet[Section V.A]{neykov2022minimax}, the lower bound from Lemma \ref{lemma:lower:bound:first:version} still applies. We similarly claim the $\log 2$ term is irrelevant. This is so since any unbounded star-shaped set contains arbitrarily long line segments, so by taking $c$ sufficiently large and picking points distance $\eta/c$-apart, we ensure $\log \cMKloc(\eta, c)> 4\log 2$. Thus, for all $\eta$ satisfying $\log \cMKloc(\eta, c) > \frac{2N \eta^2}{\sigma^2} $, we have a lower bound of $\frac{\eta^2}{8 c^2}.$

The extension to Lemma \ref{corruptions:lower:bound} is also similar. There we had selected two points $\theta_1,\theta_2\in K$ such that \[(d/3)^2\wedge 4\sigma^2 \frac{(\epsilon')^2}{(1 - \epsilon')^2} \le \|\theta_1-\theta_2\|^2 \le4\sigma^2 \frac{(\epsilon')^2}{(1 - \epsilon')^2}.\] Now we simply require $\|\theta_1-\theta_2\|^2\asymp 4\sigma^2 \frac{(\epsilon')^2}{(1 - \epsilon')^2}$ and the proof otherwise carries over. The reason such $\theta_1,\theta_2$ exist follows from the star-shaped property (implying arbitrarily long line segments fully contained in $K$). The lower bound becomes $\sigma^2\epsilon^2$ for $\epsilon\ge \tfrac{1}{\sqrt{N}}$.

Similarly, for the sub-Gaussian lower bound in Lemma \ref{lemma:subgaussian:lower:bound}, we repeat the same argument except require $\|\mu-\nu\|\asymp \sigma \sqrt{\log(1/\epsilon)}$, noting that the mixture $\cQ_{\epsilon}$ from our proof is still sub-Gaussian. Again, we can do so by the star-shaped property we used for the previous lower bound extensions. The new lower bound is $\sigma^2\epsilon^2\log(1/\epsilon)$. 

\subsection{Upper Bound}

Our upper bound requires a more sophisticated construction for which there is little hope of computational implementation.  For a fixed choice of $\mu$, we will show with high probability that the true data is within a deterministic distance of $\mu$.  After that, we construct an appropriately coarse packing and covering of the set $K$ (which has a countable cardinality; see Lemma \ref{lemma:mpacking:2mcovering}). We then perform the directed tree construction from the bounded section countably many times (rooted at different point and using a sufficiently large diameter at the first step). With high probability, we will capture $\mu$ in one of these sets and deduce our result using our previous work. 

In this section, we will assume that the star-shaped set $K$ is closed. If it is not closed, we will apply our procedure to the closure of $K$, denoted $\overline{K}$. We will argue later (see Remark \ref{remark:closure:K}) that our algorithm applied to $\overline{K}$ matches the lower bound for $K$ from the previous section.

Recall our notation of $\tilde X_i=\mu+\xi_i$ to represent the data pre-corruption. First, we show that $\mu$ is contained in $B(\tilde X_i, R)$ with high probability using a simple sub-Gaussian tail bound.

\begin{lemma} \label{lemma:R:tail:bound} For any $R>0$, we have $\PP(\|\tilde X_i-\mu\|> R)\le 5^n\exp(- \tfrac{R^2}{8\sigma^2})$.
\end{lemma}

Now we impose three requirements on $R$, while maintaining that it is some absolute constant that depends only on $\sigma$, $n$, and $\epsilon$, all of which are assumed to be known. Let $\gamma\in(0,1)$ be some constant, which for reasons clarified later, is assumed to satisfy \[\gamma > \max\left( 1 -\tfrac{C_3(\kappa)}{6(C+1)^2\log 2}, 1-\tfrac{4}{(C+1)^2C_1^2(\kappa)}\right)\] in the Gaussian case and \[\gamma > \max\left(1 -\tfrac{C_5}{6(C+1)^2\log 2},1-\tfrac{1}{(C+1)^2}, 1- \tfrac{1}{C_1^2(C+1)^2}\right)\] in the sub-Gaussian case, where the constants are chosen from the respective bounded scenario. Let $g$ be defined as in Lemma \ref{lemma:function:g:technical:details}. Then we assume \begin{equation} \label{eq:R:conditions}
    R > \max\left(\sqrt{\frac{8n\sigma^2\log 5}{1-\gamma}}, \sqrt{\frac{-32 \sigma^2 g(\epsilon)}{\gamma(1/2-\epsilon)}}, \sqrt{\frac{512\sigma^2\log 2}{\gamma(1/2-\epsilon)}}\right).
\end{equation}

Note that $-\tfrac{2g(\epsilon)}{1/2-\epsilon}\in(\log 4,\infty)$ by Lemma \ref{lemma:function:g:technical:details}, so the second bound takes the square root of a positive number. Assume moreover $R > 2(C+1)$.

The first bound in \eqref{eq:R:conditions} implies that \begin{equation} \label{eq:unbounded:tilde:X_i:bound:post:R}
    \PP(\|\tilde X_i-\mu\|\ge R) \le 
 \exp\left(-\tfrac{\gamma R^2}{8\sigma^2}\right).
\end{equation}

 For each fixed $\nu$ and $1\le i \le N$, let $E_{\nu,i}$ be the event $\|X_i-\nu\|> R$. Similarly, let $\tilde E_{\nu,i}$ be the event $\|\tilde X_i-\nu\|> R$. Now, let $S=S(R):=\{\nu\in K: \sum_{i=1}^{N} \mathbbm{1}(E_{\nu,i}) \le N/2-1\}$.  In other words, we isolate the random subset of $K$ in which strictly more than half of the possibly corrupted data points are within a reasonable distance. This set may be empty, but we will explain how to handle this scenario. We will now show both that $S$ has diameter $2R$ and that $S$ contains $\mu$ with probability $1-\exp(-\Omega(N R^2/\sigma^2))$, writing $\Omega$ to denote the omission of some absolute constants in the exponential.

\begin{lemma} We have $\mathrm{diam}(S) \le 2R$. \label{lemma:diam:S:2R}
\end{lemma}
    \begin{proof}
        The result is trivial if $S=\varnothing$, so suppose not. Pick any distinct $\nu,\nu'\in S$. To show that $\|\nu-\nu'\|\le 2R$, it suffices to obtain an $i\in [N]$ such that simultaneously $\|X_i-\nu\|\le R$ and $\|X_i-\nu'\|\le R$ (by applying the triangle inequality). In our notation, this means that we want $E_{\nu,i}^c$ and $E_{\nu',i}^c$ to occur simultaneously for some $i$. Suppose that such an index does not exist. Then for all $i$, at least one of $E_{\nu,i}$ or $E_{\nu',i}$ occurs, so that $\mathbbm{1}(E_{\nu,i}) + \mathbbm{1}(E_{\nu',i}) \ge 1$. Hence $\sum_{i=1}^N \mathbbm{1}(E_{\nu,i}) + \sum_{i=1}^N \mathbbm{1}(E_{\nu',i}) \ge N$. Then one of the summations must be $\ge N/2$. Thus, either $\nu\not\in S$ or $\nu'\not\in S$, leading to a contradiction. Hence an $i$ exists such that $\|X_i-\nu\|\le R$ and $\|X_i-\nu'\|\le R$, completing the proof.
    \end{proof}

 \begin{lemma} \label{lemma:unbounded:mu:S} We have $\PP(\mu \in S) \ge 1- \exp\left(-\frac{N(1/2-\epsilon)\gamma R^2}{32\sigma^2}\right) $.
 \end{lemma}

Note that the set $S$ depends on the choice of $R$. In our final result about the set $S$, we argue that the set-valued mapping $R\mapsto S(R)$ has a nesting property as $R$ increases.

\begin{lemma} \label{lemma:nesting:sets:S}
    $S(R) \subseteq S(R')$ for $R' > R$.
\end{lemma}
\begin{proof}
    Write \begin{align*}
        S(R) &= \bigcup_{\substack{J\subset [N]\\ |J|\ge N/2 + 1}} \bigcap_{i\in J} \{\nu\in K:\|X_i-\nu\|\le R\} = \bigcup_{\substack{J\subset [N]\\ |J|\ge N/2 + 1}} \bigcap_{i\in J} B(X_i,R)\cap K,
    \end{align*} then note that increasing $R$ clearly yields a superset.
\end{proof}

We now argue that for any closed subset of $\RR^n$ we can find a countable $m$-packing and $2m$-covering. A similar result for the bounded case is discussed in \citet[Section 3.3]{aamari_stability_2018} and could potentially be extended to yield the lemma below. For completeness we add a standalone proof in the appendix.

\begin{lemma} \label{lemma:mpacking:2mcovering} Let $T\subset \RR^n$ be an arbitrary closed set. There exists a countable $m$-packing and $2m$-covering of $T$ comprised of points belonging to $T$ for any $m > 0$. Here by a countable $2m$-cover we mean a countable set $N_T$ such that for any $x \in T$ we have $\min_{p \in N_T} \|x - p\| \leq 2m$, and by countable $m$-packing we mean that for any $x,y \in N_T$ such that $x \neq y$, we have $\|x-y\| > m$.
\end{lemma}

  Now, take $m = \tfrac{R}{c-1}$.  Since $m$ depends on $R$ which in turn depends on both $\sigma,\epsilon$, we require $\sigma$ and $\epsilon$ to be known in the unbounded case (for both Gaussian and sub-Gaussian noise).  Then form a (countable) $m$-packing $S_m$ of $K$ (assuming $K$ is closed or otherwise using $\overline K$ instead of $K$) that is also a countable $2m$-covering of $K$, which exists by Lemma \ref{lemma:mpacking:2mcovering}. At each point and independently of the data, we perform the directed tree construction (Algorithm \ref{algorithm:directed:tree}) on the set $K$ with $\Upsilon_1$ taken to be the corresponding point in $S_m$. In detail, once $\Upsilon_1 \in S_m$ has been selected, the second step is a maximal $d_m/c$-packing of $B(\Upsilon_1, d_m)\cap K$ where we define $d_m=2m+2R$ (in place of $d$ from the bounded case).  Then at each point in that packing, say $\Upsilon_2$, we form a maximal $\tfrac{d_m}{4c}$-packing of  $B(\Upsilon_2, \tfrac{d_m}{2})\cap K\cap B(\Upsilon_1, d_m)$ (pruning as needed); in general from step $J\ge 3$ onward, given any choice of $\Upsilon_{J-1}$ from the previous level, perform a maximal
 $\tfrac{d_m}{2^{J-1}c}$-packing of $B(\Upsilon_{J-1}, \tfrac{d_m}{2^{J-2}})\cap K\cap B(\Upsilon_1, d_m)$ and prune as before.

  \begin{lemma}\label{lemma:pruned:tree:properties:unbounded} Let $G$ be the pruned graph formed at a fixed $\Upsilon_1=s$.  Let $\cL^s(j)$ to be the $j$th level of the directed tree rooted at $s$.  Then for any $J\ge 3$, $\cL^s(J)$ forms a $\tfrac{d_m}{2^{J-2}c}$-covering of $K\cap B(s, d_m)$ and a $\tfrac{d_m}{2^{J-1}c}$-packing of $K\cap B(s, d_m)$. In addition, for each $J\ge 2$ and any parent node $\Upsilon_{J-1}$ at level $J-1$, its offspring $\cO(\Upsilon_{J-1})$ form a $\tfrac{d_m}{2^{J-2}c}$-covering of the set $B(\Upsilon_{J-1}, \tfrac{d_m}{2^{J-2}}) \cap K\cap B(s, d_m)$. Furthermore, the cardinality of $\cO(\Upsilon_{J-1})$ is upper bounded by $\cMKloc(\tfrac{d_m}{2^{J-2}}, 2c)$ for $J \geq 2$. 
\end{lemma}
    \begin{proof}
        Note that the proof of the first packing and covering claim in the bounded case (Lemma \ref{lemma:pruned:tree:properties}) did not rely on the star-shapedness of $K$, so the same logic gives the result, but instead for $K\cap B(\Upsilon_1, d_m)$. Similarly, the argument about $\cO(\Upsilon_{J-1})$ forming a covering still holds.

        For the last claim, $|\cO(\Upsilon_{J-1})|$ is bounded by the cardinality of a maximal packing set of $B(\Upsilon_{J-1}, \tfrac{d_m}{2^{J-2}})\cap K\cap B(\Upsilon_1, d_m)$ which by definition is bounded by $\cM_{K\cap B(\Upsilon_1,d_m)}^{\loc}(\tfrac{d_m}{2^{J-2}}, c)$. But this is bounded by $\cMKloc(\tfrac{d_m}{2^{J-2}}, c)\le \cMKloc(\tfrac{d_m}{2^{J-2}}, 2c)$ (the last part using the star-shapedness of $K$ and a similar argument as Lemma \ref{lemma:non:increasing:local:entropy}). For $J=2$, the cardinality recall we form level 2 with a maximal $d_m/c$-packing of $B(\Upsilon_1,2m+2R)\cap K$, which is bounded by $\cM_{K\cap B(\Upsilon_1,d_m)}^{\loc}(d_m, c)\le \cMKloc(d_m, 2c)$ by similar logic.
    \end{proof}

\begin{lemma} \label{lemma:bound:level:J:intersect:ball:mu:unbounded} Consider the same set-up as in Lemma \ref{lemma:pruned:tree:properties:unbounded}. Pick any $\mu\in K$. Then for $J\ge 2$, the cardinality of $\cL^s(J-1)\cap B(\mu, \tfrac{d_m}{2^{J-2}})$ is upper bounded by $\cMKloc(\tfrac{d_m}{2^{J-2}}, c)\leq \cMKloc(\tfrac{d_m}{2^{J-2}}, 2c)$. 
\end{lemma}
    \begin{proof} Lemma \ref{lemma:pruned:tree:properties:unbounded} states that for $J\ge 4$, $\cL^s(J-1)$ forms a $\tfrac{d_m}{2^{J-2}c}$-packing of $K\cap B(\Upsilon_1, d_m)$. Taking the intersection with $B(\mu, \tfrac{d_m}{2^{J-2}})$, we obtain a (possibly empty)  $\tfrac{d_m}{2^{J-2}c}$-packing of $B(\mu, \tfrac{d_m}{2^{J-2}})\cap K\cap B(\Upsilon_1, d_m)$. By definition, the cardinality of this packing set is bounded by \[\cM_{K\cap B(\Upsilon_1, d_m)}^{\loc}(\tfrac{d_m}{2^{J-2}}, c)\le \cMKloc(\tfrac{d_m}{2^{J-2}}, c).\]

    For $J=2$, note there is only a single element in $\cL^s(1)$, so the cardinality is bounded by 1 and thus $\cMKloc(d_m,c)$. For $J=3$, recall that $\cL(2)$ was formed  by a $d_m/c$-covering of $B(\Upsilon_1, d_m)\cap K$, and this is bounded by definition by \[\cM_{K\cap B(\Upsilon_1, d_m)}^{\loc}(d_m,c)\le \cMKloc(d_m,c) \le \cMKloc(d_m,2c).\]
    \end{proof}

We still have the Cauchy sequence property for any infinite path along the directed graph $G$. The proof is similar to the bounded version (Lemma \ref{lemma:cauchy:sequence}) except we just add an intersection with $B(\Upsilon_1, d_m)$ where $\Upsilon_1$ is the chosen root node from the countable packing $S_m$, and replace $d$ with $d_m$ everywhere.

\begin{lemma} \label{lemma:cauchy:sequence:unbounded} Let $[\Upsilon_1, \Upsilon_2,\dots,]$ be the nodes of an infinite path in the graph $G$, i.e., $\Upsilon_{J + 1} \in \cO(\Upsilon_{J})$ for all $J \in \NN$. Then for any integers $J\ge J'\ge 1$, $\|\Upsilon_{J'}-\Upsilon_J\|\le  \frac{d_m(2+4c)}{c 2^{J'}}$.
\end{lemma}
   
Finally, we may now describe the actual estimator. In the edge case where $S=\varnothing$, our estimator is the lexicographically smallest point in the set $S(\hat R)$ where $\hat R = \min_{t > R} \{t: S(t) \neq \varnothing\}$.\footnote{We have a $\min$ here because the sets $S(t)$ are nested and closed; hence for any monotone decreasing sequence $t_i \rightarrow t_0$, if each of the sets $S(t_i)$ is non-empty, then $S(t_0)  = \cap_{i \in \NN} S(t_i)$ is also non-empty.} We will return to this edge case later.

Otherwise, if $S\ne\varnothing$, we proceed with a similar iterative estimation procedure. That is, let the random variable $\Upsilon_1\in S_m$ be the point in the $m$-packing $S_m$ that is closest to the random set $S$, breaking ties lexicographically. Note that one of the directed trees that we constructed has a root at $\Upsilon_1$. Observe that $S$ is contained in $B(\Upsilon_1,d_m)\cap K$. To see this: if $\Upsilon_1\in S$, then since $\mathrm{diam}(S)\le 2R < d_m$, clearly $S\subset B(\Upsilon_1, d_m)\cap K$. If $\Upsilon_1\not\in S$, there exists some $s\in S$ such that $\|\Upsilon_1-s\|\le 2m$ since $\Upsilon_1$ is the closest point in the $2m$-covering of $K$ (which contains $S$). But for any $s'\in S$, we have $\|s-s'\|\le 2R$ by the diameter result, hence $\|\Upsilon_1-s'\|\le d_m$ by the triangle inequality. Thus, $S\subset B(\Upsilon_1, d_m)$. So by Lemma \ref{lemma:unbounded:mu:S},  
\begin{align} 
    \PP(\mu\in B(\Upsilon_1,d_m)\cap K | S \neq \varnothing) &\ge \PP(\mu\in S | S \neq \varnothing) \ge \PP(\mu\in S)\notag \\ &\geq 1-\exp\left(-\frac{N(1/2-\epsilon)\gamma R^2}{32\sigma^2}\right), \label{eq:unbounded:mu:K_m}
\end{align} 
where we used that $\{\mu \in S\}\subset \{S \neq \varnothing\}$ to drop the conditioning.

 Now, apply Algorithm \ref{algorithm:robust} to the directed tree rooted at $\Upsilon_1$  to obtain a sequence $[\Upsilon_1,\Upsilon_2,\dots]$. As before, we use $d_m =2m + 2R$ in place of $d$. We use our modified directed tree properties in Lemmas \ref{lemma:pruned:tree:properties:unbounded} and \ref{lemma:bound:level:J:intersect:ball:mu:unbounded}. We still have the Cauchy sequence property from Lemma \ref{lemma:cauchy:sequence:unbounded} except with $d_m$. Our previous tournament bound in Lemma \ref{lemma:tournament} still holds. Then we need some modified auxiliary lemmas for the upper bound. From now on, we must separately consider the Gaussian and sub-Gaussian cases, but all that changes is swapping out the condition on $\delta$.  We re-use the same (absolute) constants from either bounded case, e.g., $C_1(\kappa), C_3(\kappa)$ in the Gaussian case.

\begin{theorem} \label{theorem:unbounded:version}   Consider the Gaussian noise setting. Let $c=2(C+1)>2$ be the constant used in the local metric entropy.  For any $J\ge 1$, define $\eta_J = \frac{d_m\sqrt{C_3(\kappa)}}{2^{J-1}(C+1)}$ where $C_3(\kappa)$ is from Theorem \ref{theorem:main:testing:result} and $d_m=2m + 2R$ with $m=\tfrac{R}{c-1}$ (and $R$ chosen as in \eqref{eq:R:conditions}). Let $J^{\ast}$ be the maximal $J$ such that \begin{equation} \label{eq:robust:unbounded:theorem:condition}
    \frac{N\eta_J^2}{\sigma^2} >2\log \left[\cMloc\left(\tfrac{c\eta_J}{\sqrt{C_3(\kappa)}}, 2c\right)\right]^3\vee \log 2,
\end{equation} and set $J^{\ast}=1$ if this condition never occurs. If $S(R)\ne\varnothing$, write $\nu^{\ast\ast}(X_1,\dots,X_N)$ as the output of at least $J^{\ast}$ iterations of Algorithm \ref{algorithm:robust}. If $S(R)=\varnothing$, write $\nu^{\ast\ast}(X_1,\dots,X_N)$ as $\hat p$, where $\hat p$ is the lexicographically smallest point in $S(\hat R)$ defined previously. Then we have $\EE_{X} \|\nu^{\ast\ast}-\mu\|^2 \lesssim \max\left(\eta_{J^{\ast}}^2, \epsilon^2\sigma^2\right)$. 
\end{theorem}

We proceed to prove our algorithm in the unbounded Gaussian setting attains the minimax lower bound. While we do not state it formally, our symmetric or known sub-Gaussian noise setting will have a similar guarantee. This is due to our Remark \ref{remark:expectation:over:R}.

 \begin{theorem}[Robust Unbounded Gaussian Minimax Rate] \label{theorem:robust:minimax:rate:attained:unbounded:gaussian} Let $\epsilon\in(0,1/2)$. Define \[\eta^{\ast} = \sup\{\eta \ge 0:N\eta^2/\sigma^2\le \log \cMloc(\eta,c)\}.\] Then for   sufficiently large $c$, the minimax rate is $\max({\eta^{\ast}}^2, \sigma^2\epsilon^2)$.
\end{theorem}

We now state the (general) sub-Gaussian equivalents of the unbounded algorithm error.  Note also that we reset our absolute constants from the Gaussian section, and we re-use the absolute constants from the bounded, sub-Gaussian setting.

\begin{theorem} \label{theorem:general:subgaussian:version:unbounded} Consider the sub-Gaussian noise setting where $\epsilon\in(0,1/32)$. Let $c=2(C+1)$ be the constant used in the local metric entropy.  For any $J\ge 1$, define $\eta_J = \frac{d_m\sqrt{C_5}}{2^{J-1}(C+1)}$ where $C_5$ is from Theorem \ref{theorem:asymmetric:testing:result} and $d_m=2m + 2R$ with $m=\tfrac{R}{c-1}$ and $R$ chosen as in \eqref{eq:R:conditions}. Let $J^{\ast}$ be the maximal $J$ such that \begin{equation} \label{eq:asymmetric:robust:theorem:condition:subgaussian:unbounded}
    \frac{N\eta_J^2}{\sigma^2} > 2\log \left[\cMloc\left(\frac{c\eta_J}{\sqrt{C_5}}, 2c\right)\right]^3\vee \log 2,
\end{equation} and set $J^{\ast}=1$ if this condition never occurs. If $S(R)\ne\varnothing$, we set $\nu^{\ast\ast}(X_1,\dots,X_N)$ as the output of at least $J^{\ast}$ iterations of Algorithm \ref{algorithm:robust}. If $S(R)=\varnothing$, set $\nu^{\ast\ast}(X_1,\dots,X_N)$ as $\hat p$, where $\hat p$ is the lexicographically smallest point in $S(\hat R)$ defined previously. Then we have $\EE_R\EE_{X} \|\nu^{\ast\ast}-\mu\|^2 \lesssim \max\left(\eta_{J^{\ast}}^2, \epsilon^2\log(1/\epsilon)\sigma^2\right)$. 
\end{theorem}

 \begin{theorem}[Robust Unbounded Sub-Gaussian Minimax Rate] \label{theorem:robust:minimax:rate:attained:unbounded:subgaussian} Let $\epsilon\in(0,1/32)$. Define \[\eta^{\ast} = \sup\{\eta \ge 0:N\eta^2/\sigma^2\le \log \cMloc(\eta,c)\}.\] Then for   sufficiently large $c$, the minimax rate is $\max({\eta^{\ast}}^2, \sigma^2\epsilon^2\log(1/\epsilon))$.
\end{theorem}

\begin{remark} \label{remark:closure:K} In this remark, we show that if $K$ is not closed, the minimax rate obtained by Theorems \ref{theorem:robust:minimax:rate:attained:unbounded:gaussian} and \ref{theorem:robust:minimax:rate:attained:unbounded:subgaussian} remains valid even though we used $\overline K$ in place of $K$ in our algorithms. To that end, we will first argue that $\cMloc_{\overline{K}}(\eta, c) \leq \cMloc_{K}(\tfrac{\eta(2c + 2)}{2c+1}, 2c)$. To see this, assume that $\cMloc_{\overline{K}}(\eta, c)$ is achieved at a point $\nu \in \overline{K}$ (if the supremum is not achieved, then we can take a limiting sequence and repeat the same steps for that sequence). For brevity, denote $M = \cMloc_{\overline{K}}(\eta, c)$, and let $\{\nu_1,\ldots, \nu_M\} \subset \overline{K} \cap B(\nu, \eta)$ be an $\eta/c$-packing. Observe that since $\nu, \nu_1,\ldots, \nu_M \in \overline{K}$ we can find points from $K$ arbitrarily close to any of these points. Suppose that $\nu', \nu_1', \ldots, \nu_M' \in K$ are such that $\|\nu - \nu'\| \vee \max_{i \in [M]} \|\nu_i - \nu'_i\| \leq \kappa$ for some small $\kappa > 0$. Now we claim that $\nu_i' \in B(\nu', \eta + 2\kappa)\cap K$ and that the $\nu_i'$ form a $(\tfrac{\eta}{c}-2\kappa)$-packing of this set. Both claims follow from the triangle inequality. Since $\kappa$ is arbitrary, we can select $\kappa = \tfrac{\eta}{4c + 2}$ and prove the claimed local entropy inequality. 

Hence
\begin{align*}
    \eta^* & = \sup\{\eta \ge 0: N\eta^2/\sigma^2 \leq \log \cM_{\overline{K}}^{\loc}(\eta,c )\} \\
    & \leq \sup\{\eta \ge 0: N\eta^2/\sigma^2 \leq \log \cMKloc(\tfrac{\eta(2c + 2)}{2c+1}, 2c)\} =: \hat \eta.
\end{align*}
Recall Theorems \ref{theorem:robust:minimax:rate:attained:unbounded:gaussian} and \ref{theorem:robust:minimax:rate:attained:unbounded:subgaussian} argued that the worst case rate of our algorithm is $\lesssim \max(\eta^{*2}, \sigma^2\epsilon^2)$ in the Gaussian case (and known or symmetric sub-Gaussian noise) and  $\lesssim \max(\eta^{*2}, \sigma^2\epsilon^2\log 1/\epsilon)$ in the unknown sub-Gaussian noise case. Next, by Lemma 1.5 of \cite{prasadan2024some}, $\hat \eta$ (defined above) satisfies:
\begin{align*}
    \hat \eta \asymp \sup\{\eta \ge 0: N\eta^2/\sigma^2 \leq \log \cMKloc(\eta, 2c)\} =: \tilde \eta.
\end{align*} Note that Lemma 1.5 of that paper uses the monotonicity property of local metric entropy of convex sets, but we showed in Lemma \ref{lemma:non:increasing:local:entropy} this property holds in the star-shaped setting. 
Thus we conclude that the worst case rate of our procedure is upper bounded by $\lesssim \max(\tilde \eta^{2}, \sigma^2\epsilon^2)$ in the Gaussian case (and in the known or symmetric sub-Gaussian case) and $\lesssim \max(\tilde \eta^{2}, \sigma^2\epsilon^2\log 1/\epsilon)$ in the unknown sub-Gaussian noise case. On the other hand, $\tilde\eta$ is a valid lower bound by the argument in Section \ref{subsection:lower:bound:unbounded}, as are $\sigma^2\epsilon^2$ and $\sigma^2\epsilon^2 \log(1/\epsilon)$ for the Gaussian and sub-Gaussian settings, respectively. 
\end{remark}

\subsection{Example: Sparse Robust Mean Estimation}
\label{section:robust_gsm:sparse}

We close our paper with an application to the well-studied setting of sparse mean estimation and capture the impact of adversarial corruption. In this setting, the statistician does not know which predictors are irrelevant, and the adversary could in principle try to confuse the statistician by zeroing out the `wrong' ones. 

Let $K$ be the set of vectors in $\RR^n$ where it is known at most $s$ of the coordinates are non-zero for some integer $1\le s\le n/8$. This is clearly an unbounded star-shaped set with the $0$ vector as the center.

\begin{lemma} \label{lemma:sparse:varshamov} $\log \cMKloc(\delta,c) \asymp s\log(1+\tfrac{n}{2s})$.
\end{lemma}

As a consequence, in the Gaussian (or symmetric or known sub-Gaussian) noise setting, applying Theorem \ref{theorem:robust:minimax:rate:attained:unbounded:gaussian}, we have ${\eta^{\ast}}^2 \asymp\frac{\sigma^2}{N}\log \cMloc(\eta^{\ast}, c)\asymp \frac{\sigma^2}{N}\cdot s\log(1+\tfrac{n}{2s})$. The minimax rate is thus given by \[\max\left( \frac{\sigma^2}{N}\cdot s\log(1+\tfrac{n}{2s}), \sigma^2\epsilon^2\right).\] Similarly, in the unknown sub-Gaussian setting, the minimax rate is \[\max\left( \frac{\sigma^2}{N}\cdot s\log(1+\tfrac{n}{2s}), \sigma^2\epsilon^2\log(1/\epsilon)\right)\] by Theorem \ref{theorem:robust:minimax:rate:attained:unbounded:subgaussian}. Our algorithm in all cases achieves the minimax rate, although recall that we must have knowledge of $\sigma^2$ in this unbounded setting as well as the corruption rate $\epsilon$ and the sparsity $s$. 

Several comments are in order about the novelty of the above result. When there are no outliers present with known $\sigma$, the above result was first established by \cite{donoho1992maximum}. When outliers are present, there are multiple previous related works, including \citep{balakrishnan2017computationally, diakonikolas2019outlier, cheng2022outlier, diakonikolas2022robust}. The closest result to our characterization above is Theorem 88 of \cite{diakonikolas2022robust} which states that under the case of Gaussian noise, if $N \gtrsim (s \log n)/\epsilon^2$ the minimax rate is $\sigma^2 \epsilon^2$ in high probability. Observe that even in this simple Gaussian case, this result fails to capture the full landscape of the rate.

\section{Discussion and Future Work}
\label{section:robust_gsm:discussion}

In this paper, we presented a minimax optimal estimation procedure for the  adversarially corrupted (sub)-Gaussian location model with star-shaped constraints. Our results were inspired by the work of \cite{neykov2022minimax} on the minimax rate of the uncontaminated, convex constrained Gaussian sequence model. The proofs presented in this manuscript use deep results from probability theory (such as local CLTs) to obtain the first definitive robust minimax rate for all sub-Gaussian distributions under any kind of star-shaped constraint, regardless of whether the noise distribution is known or unknown. Interestingly, as it turns out, if one knows the sub-Gaussian noise distribution, the minimax optimal rate is faster than if the distribution is unknown. The obvious drawback of this work is that the algorithms are hard to implement, but we do hope our paper inspires efforts to devise computationally efficient algorithms achieving the same performance under various constraints for the mean. 

We now comment on the i.i.d. Gaussian case. We believe that our results can be used to understand what happens with the minimax rate when $\kappa \rightarrow 0$, that is, the number of outliers approaches $1/2$. In a recent paper by \cite{kotekal2024optimal}, the authors demonstrated (see Appendix E therein), that the minimax rate for the problem when $K = \RR^1$ is $\sigma^2 \log (\tfrac{en}{n(1 - 2 \epsilon)})$ when $\epsilon$ approaches $1/2$. We think that by carefully studying the behavior of the constant $C_1$ in the proof of our Theorem \ref{theorem:main:testing:result}, and more specifically Lemma \ref{lemma:existence:constants}, and some of our results in Section \ref{section:robust_gsm:unbounded}, one can obtain a similar result for the much more general multivariate case under constraints. We leave the details of this exercise for future exploration.

Another avenue for future research is to understand the case when the noise can be heavy tailed, e.g., it is only assumed that $\sup_{v\in S^{n-1}}\EE (v\T\xi_i)^2 \leq D$ for some $D > 0$. We believe that our ideas can be useful for deriving the information-theoretic limits of this setting as well. However, the challenge there is that we no longer have simple tail bounds for heavy-tailed random variables as opposed to the sub-Gaussian variables considered in this work. Yet another potential exploration is using Lepski's method to make our algorithm adaptive to unknown $\epsilon$ in the unknown sub-Gaussian case (by contrast, our Gaussian case permits unknown $\epsilon$).
\section{Acknowledgements}
\label{section:acknowledgements}

The authors appreciate the input of an Associate Editor and two anonymous referees for suggestions that fixed typos and greatly improved the presentation of the manuscript.

\bibliographystyle{abbrvnat}
\bibliography{robust_gsm_ref}

\appendix

\section{Proofs for Section \ref{section:robust_gsm:lower:bound}}

To prove our uncorrupted lower bound, we first state Fano's inequality and subsequently prove a bound on the mutual information.

\begin{lemma}[Fano's inequality] \label{lemma:fano} Let $\mu^1, \ldots, \mu^m$ be a collection of $\eta$-separated points in the parameter space in Euclidean norm. Suppose $J$ is uniformly distributed over the index set $[m]$, and for each $i\in[N]$ and $j\in[m],$ we have $(\tilde X_i | J = j) = \mu^j + \xi_i$ where $\xi_i \sim \cN(0, \II\sigma^2)$. Let the mutual information $I(\tilde X;J)$ be defined as the average Kullback–Leibler (KL) divergence of the $\prod_i(\tilde X_i | J = j)$ and their mixture across $J$ as defined in \citet[Equation(15.30)]{wainwright2019high}. Then 
\begin{align*}
     \inf_{\hat \mu} \sup_{\mu} \EE \|\hat \mu(\tilde X) - \mu\|^2 \geq \frac{\eta^2}{4}\bigg(1 - \frac{I(\tilde X; J) + \log 2}{\log m}\bigg).
\end{align*}
\end{lemma}

\begin{lemma}[Bounding the Mutual Information] \label{lemma:mutual:info} For any $\nu\in K$ we have \[I(\tilde X;J) \le \frac{N}{2\sigma^2} \max_j \|\mu^j-\nu\|^2.\] 
\end{lemma}
\begin{proof}
    We use \citet[equation (15.52)]{wainwright2019high}, the independence of $\tilde X_1,\dots,\tilde X_n$, as well as properties of KL-divergence $D_{\mathrm{KL}}$ for Gaussian random variables.
    Let $\PP_{\tilde X_i}^{\mu^j}$ be the distribution of $(\tilde X_i |J = j)$, which is $\cN(\mu^j,\sigma^2\II)$, and let $\PP_{\tilde X_i}^{\nu}$ be the distribution of $\tilde X_i=\nu+\xi_i$ which is $\cN(\nu,\sigma^2\II)$. Then we obtain 
    \begin{align*}
        I(\tilde X;J) 
        &\le \frac{1}{m} \sum_j D_{\mathrm{KL}}\left(\prod_i\PP_{\tilde X_i}^{\mu^j} \parallel \prod_i       \PP_{\tilde X_i}^{\nu}\right) \\
        &= \frac{1}{m} \sum_j \sum_i D_{\mathrm{KL}}\left(\PP_{\tilde X_i}^{\mu^j} \parallel        \PP_{\tilde X_i}^{\nu}\right) \\
        &= \frac{1}{m} \sum_j \sum_i D_{\mathrm{KL}}\left(\cN(\mu^j,\sigma^2\II) \parallel        \cN(\nu,\sigma^2\II)\right)  \\
        &= \frac{1}{2\sigma^2 m}\sum_j \sum_i \|\mu^j-\nu\|^2 \\
        &= \frac{N}{2\sigma^2 m}\sum_j \|\mu^j-\nu\|^2 \\
        &\le \frac{N}{2\sigma^2}\max_j \|\mu^j-\nu\|^2.
\end{align*}
\end{proof}

\begin{proof}[Proof of Lemma \ref{lemma:lower:bound:first:version}]
    The first inequality is obvious, since if the adversary does nothing to the original observations, we obtain the middle term. We now show the second inequality. Pick a $\nu\in K$ that achieves \[M:=\cMKloc(\eta, c):=\sup_{\nu \in K}\cM(\eta/c, B(\nu, \eta) \cap K)\] (otherwise, repeat the following argument by picking a limiting sequence of functions). Let $\nu_1,\dots,\nu_M$ be a maximal $\eta/c$-packing set for $B(\nu, \eta) \cap K$, so that $\|\nu_i-\nu_j\|>\eta/c$ for all $i\ne j$. Because $\nu_1,\dots,\nu_M\in B(\nu, \eta) \cap K$, it follows for each $j\in[M]$ that $\|\nu_j-\nu\| \le \eta$. Letting $J$ be uniformly distributed on $[M]$, we have from Lemma \ref{lemma:mutual:info} that \begin{align*}
        I(\tilde X;J) &\le \tfrac{N}{2\sigma^2}\cdot\max_j  \|\nu_j-\nu\|^2 \le \tfrac{N\eta^2}{2\sigma^2}.
    \end{align*} Applying Lemma \ref{lemma:fano}, we obtain \begin{align*}
        \inf_{\hat \mu} \sup_{\mu} \EE\|\hat{\mu}(\tilde X)-\mu\|^2 &\geq \frac{\eta^2}{4c^2}\bigg(1 - \frac{\frac{N\eta^2}{2\sigma^2} + \log 2}{\log M}\bigg).
    \end{align*}
    Choose $\eta$ such that $\log M$ is bigger than the middle term in \[\log M\ge 4\left(\frac{N\eta^2}{2\sigma^2} \vee \log 2\right) \ge 2\left(\frac{N\eta^2}{2\sigma^2}+\log 2\right),\] and this proves the minimax risk is lower-bounded by $\frac{\eta^2}{8c^2}$.
\end{proof}

\begin{proof}[\hypertarget{proof:corruptions:lower:bound}{Proof of Lemma \ref{corruptions:lower:bound}}] 
     First, take $\epsilon'=\epsilon - \tfrac{1}{\sqrt{2N}}<1/2-\kappa$ and observe that $\epsilon'\asymp \epsilon$. Next, pick $\theta_1,\theta_2\in K$ such that \[(d/3)^2\wedge  \tfrac{4\sigma^2{\epsilon'}^2}{(1 - \epsilon')^2} \le \|\theta_1-\theta_2\|^2 \le \tfrac{4\sigma^2{\epsilon'}^2}{(1 - \epsilon')^2},\] which we can do by the star-shaped property in  Lemma \ref{lemma:star:shaped:has:line:segment}.
    
    Now consider the two Gaussian measures $\PP_{\theta_1} \sim \cN(\theta_1, \sigma^2 \II_p)$ and $\PP_{\theta_2} \sim \cN(\theta_2, \sigma^2 \II_p)$, and observe that $\operatorname{TV}(\PP_{\theta_1}, \PP_{\theta_2}) \leq \frac{\epsilon'}{1 - \epsilon'}$. This follows since $\|\theta_1 - \theta_2\|^2 \leq  \frac{4\sigma^2(\epsilon')^2}{(1 - \epsilon')^2}$ by Pinsker's inequality, which states that \[\operatorname{TV}(\PP_{\theta_1}, \PP_{\theta_2})  \leq \sqrt{\frac{1}{2} D_{\operatorname{KL}}(\PP_{\theta_1}\| \PP_{\theta_2})}.\] 

     Following \cite{chen2018robust}, we can pick $\epsilon'' \leq \epsilon'$ such that $\operatorname{TV}(\PP_{\theta_1}, \PP_{\theta_2}) = \frac{\epsilon''}{1 - \epsilon''}$ since $\operatorname{TV}(\PP_{\theta_1}, \PP_{\theta_2})\in(0,\tfrac{\epsilon'}{1 - \epsilon'}]$. Write $p_{\theta_1}$ and $p_{\theta_2}$ as the densities of $\PP_{\theta_1}, \PP_{\theta_2}$ with respect to the Lebesgue measure and define the distributions $Q_1, Q_2$ with densities $q_1$, $q_2$  by
    \begin{align*}
        q_1 = \frac{(p_{\theta_2} - p_{\theta_1}) \mathbbm{1}(p_{\theta_2} \geq p_{\theta_1})}{\operatorname{TV}(\PP_{\theta_1}, \PP_{\theta_2})}, ~~~~~~ q_2 = \frac{(p_{\theta_1} - p_{\theta_2}) \mathbbm{1}(p_{\theta_1} \geq p_{\theta_2})}{\operatorname{TV}(\PP_{\theta_1}, \PP_{\theta_2})}.
    \end{align*}
   One can verify (see \cite{chen2018robust} for a proof) that $q_1$ and $q_2$ are indeed probability densities. Now consider the mixture measures $R_1=(1 - \epsilon'')  P_{\theta_1} + \epsilon'' Q_1$ and $R_2=(1 - \epsilon'')  P_{\theta_2} + \epsilon'' Q_2$ with densities $r_1,r_2$. Using that $\operatorname{TV}(\PP_{\theta_1}, \PP_{\theta_2}) = \frac{\epsilon''}{1 - \epsilon''}$, their densities with respect to the Lebesgue measure are:
    \begin{align*}
         r_1  &= (1 - \epsilon'') p_{\theta_1} + (1-\epsilon'')(p_{\theta_2} - p_{\theta_1}) \mathbbm{1}(p_{\theta_2} \geq p_{\theta_1}) \\
         &= (1 - \epsilon'') p_{\theta_1}\mathbbm{1}(p_{\theta_1} \geq p_{\theta_2}) + (1-\epsilon'')p_{\theta_2}  \mathbbm{1}(p_{\theta_2} \geq p_{\theta_1}) \\
        & = (1 - \epsilon'') (p_{\theta_1} - p_{\theta_2})\mathbbm{1}(p_{\theta_1} \geq p_{\theta_2}) + (1-\epsilon'')p_{\theta_2} \\ 
        & = r_2.
    \end{align*}
    Hence the two measures $R_1,R_2$ are the same and $\operatorname{TV}(R_1^{\otimes N}, R_2^{\otimes N}) = 0.$

    We now expand the TV distance $\operatorname{TV}(R_1^{\otimes N}, R_2^{\otimes N})$. Let us define some useful notation. Define $\binom{[N]}{s}  = \{(i_1,\ldots, i_s): i_1, \ldots, i_s \in [N], i_1 < i_2 < \ldots i_s\}$. Define the quantities \begin{align*}
        S_1^L &= \sum_{s = 0}^{N(\epsilon'' + 1/\sqrt{2N})} \sum_{I \in \binom{[N]}{s}}(1 - \epsilon'')^{N-s} (\epsilon'')^{s}  \prod_{i \in [N]\setminus I} p_{\theta_1}(x_i)\prod_{i \in I} q_1(x_i)  \\
       S_2^L &=  \sum_{s = 0}^{N(\epsilon'' + 1/\sqrt{2N})} \sum_{I \in \binom{[N]}{s}}(1 - \epsilon'')^{N-s} (\epsilon'')^{s}  \prod_{i \in [N]\setminus I} p_{\theta_2}(x_i)\prod_{i \in I} q_2(x_i) \\
        S_1^U &= \sum_{s = N(\epsilon'' + 1/\sqrt{2N}) + 1}^{N} \sum_{I \in \binom{[N]}{s}}(1 - \epsilon'')^{N-s} (\epsilon'')^{s} \prod_{i \in [N]\setminus I} p_{\theta_1}(x_i)\prod_{i \in I} q_1(x_i)  \\
       S_2^U &=  \sum_{s = N(\epsilon'' + 1/\sqrt{2N}) + 1}^{N} \sum_{I \in \binom{[N]}{s}}(1 - \epsilon'')^{N-s} (\epsilon'')^{s} \prod_{i \in [N]\setminus I} p_{\theta_2}(x_i)\prod_{i \in I} q_2(x_i),
    \end{align*} assuming the summation bounds are integers for simplicity. By applying the triangle inequality and integrating out the densities in each of the $S_i^U$ terms, so that one is left with terms of the form $\sum_{s = N(\epsilon'' + 1/\sqrt{2N}) + 1}^{N} \binom{N}{s}(1 - \epsilon'')^s (\epsilon'')^{N-s}$, we obtain \begin{equation} \label{eq:recognizing:a:binomial}
        \frac{1}{2}\int|S_1^U - S_2^U| \mathrm{d}\mathbf{x}\le \PP\left(\mathrm{Bin}(N, \epsilon'') > N(\epsilon'' + 1/\sqrt{2N})\right).
    \end{equation}
    Then expanding out the product of densities, and then splitting the sum into $S_1^L,S_2^L$ and $S_1^U, S_2^U$ with the triangle inequality, we have \begin{align*}
       \MoveEqLeft \operatorname{TV}(R_1^{\otimes N}, R_2^{\otimes N})  \\ &= \frac{1}{2} \int \bigg|\prod_{i \in [N]}[(1 - \epsilon'')  p_{\theta_1}(x_i) + \epsilon'' q_1(x_i)] \\ &\qquad\qquad -  \prod_{i \in [N]}[(1 - \epsilon'')  p_{\theta_2}(x_i) + \epsilon'' q_2(x_i)] \bigg| \mathrm{d}\mathbf{x} \\
        &=  \frac{1}{2} \int \bigg|\sum_{s = 0}^{N} \sum_{I \in \binom{[N]}{s}}(1 - \epsilon'')^{N-s} (\epsilon'')^{s} \Big[ \prod_{i \in [N]\setminus I} p_{\theta_1}(x_i)\prod_{i \in I} q_1(x_i) \\ &\qquad\qquad - \prod_{i \in [N]\setminus I} p_{\theta_2}(x_i)\prod_{i \in I} q_2(x_i)\Big] \bigg|  \mathrm{d}\mathbf{x}\\
        &\ge  \frac{1}{2} \int |S_1^L - S_2^L|  \mathrm{d}\mathbf{x} - \frac{1}{2}\int|S_1^U - S_2^U| \mathrm{d}\mathbf{x} \\
        &\ge \frac{1}{2} \int |S_1^L - S_2^L| \mathrm{d}\mathbf{x} -\PP\left(\mathrm{Bin}(N, \epsilon'') > N(\epsilon'' + 1/\sqrt{2N})\right) \\
        &\ge \frac{1}{2} \int |S_1^L - S_2^L| \mathrm{d}\mathbf{x} - \exp(-1)
    \end{align*} where we used \eqref{eq:recognizing:a:binomial} and Hoeffding's inequality in last two steps. Since the TV is 0, we have $\frac{1}{2} \int |S_1^L - S_2^L| \mathrm{d}\mathbf{x} \le \exp(-1)$. Note also that \[\PP\left(\mathrm{Bin}(N, \epsilon'') \le N(\epsilon'' + 1/\sqrt{2N})\right) \ge 1-\exp(-1).\] Therefore, \begin{align*}
        \frac{\tfrac{1}{2} \int |S_1^L - S_2^L| \mathrm{d}\mathbf{x}}{\PP(\mathrm{Bin}(N, \epsilon'') \le N(\epsilon'' + 1/\sqrt{2N}))} \le \frac{\exp(-1)}{1-\exp(-1)}<1.
    \end{align*}

    Now, upon further inspection, one can see that the left-hand side is the $\mathrm{TV}$ between conditional mixture measures defined as follows: Given $N$ observations of $R_1$, let $W_1$ be the number of times $Q_1$ is drawn (in place of $P_{\theta_1}$), which is a $\mathrm{Bin}(N, \epsilon'')$ random variable. Define $\tilde R_1= R_1^{\otimes N}|\{W_1\le N( \epsilon'' + 1/\sqrt{2N})\}$. Similarly, let $W_2$ be the number of times $Q_2$ is drawn (in place of $P_{\theta_2}$) in $N$ observations of $R_2$, and define $\tilde R_2= R_2^{\otimes N}|\{W_2\le N( \epsilon'' + 1/\sqrt{2N})\}$. In other words, for each $i\in\{1,2\}$, $\tilde R_i$ is the distribution $P_{\theta_i}^{\otimes N}$ conditional on  $N(\epsilon''+1/\sqrt{2N})$ (which is $\leq N\epsilon$) of the observations corrupted into having distribution $Q_i$. We have just shown $\mathrm{TV}(\tilde R_1,\tilde R_2)<1$. 

    Thus, by Le Cam's two point lemma \citep[Chapter 15]{wainwright2019high} we have
    \begin{align*}
        \inf_{\hat\mu} \sup_{\mu\in K} \sup_{\cC} \EE_{\mu}\|\hat\mu(C(\tilde X))-\mu\|^2 &\gtrsim \|\theta_1-\theta_2\|^2\cdot \frac{1}{2}(1 - \mathrm{TV}(\tilde R_1,\tilde R_2)) \\  &\gtrsim \sigma^2 \epsilon^2 \wedge d^2.
    \end{align*}
\end{proof}

\begin{proof}[\hypertarget{proof:lemma:subgaussian:lower:bound}{Proof of Lemma \ref{lemma:subgaussian:lower:bound}}] 
Let $\mu$ be the center of $K$ (from the definition of a star-shaped set). Consider drawing data from a mixture distribution between a Gaussian and a constant, i.e., $X_i\sim \cQ_{\epsilon}$ where $\cQ_{\epsilon}=(1-\tfrac{\epsilon}{2}) \cdot \cN(\mu, \sigma^2 \II) + (\tfrac{\epsilon}{2}) \nu$. We require the mean $\mu_{\cQ}:=(1-\tfrac{\epsilon}{2})\mu + (\tfrac{\epsilon}{2}) \nu$ to be in $K$, while not necessarily requiring $\nu\in K$. 

Let us first ensure the distribution is sub-Gaussian. We equivalently characterize $X_i\sim\cQ_{\epsilon}$ as follows: for a $\mathrm{Ber}(\tfrac{\epsilon}{2})$ random variable $W_i$, let $X_i\sim \cN(\mu, \sigma^2 \II)$ when $W_i=0$ and $X_i=\nu$ when $W_i=1$. For any $t \geq 0$ and any unit vector $v\in S^{n-1}$, observe that \begin{align*}
    \PP(v^T(X_i-\mu_{\cQ}) \ge t) &= \PP(W_i=0)\PP\left(v^T(X_i-\mu_{\cQ}) \ge t | W_i= 0\right)+ \\ &\quad\quad \PP(W_i=1)\PP\left(v^T(X_i-\mu_{\cQ}) \ge t | W_i =1\right) \\
    &=  (1-\tfrac{\epsilon}{2})\PP\left(v^T(\cN(\mu, \sigma^2\II)-(1-\tfrac{\epsilon}{2})\mu - (\tfrac{\epsilon}{2}) \nu) \ge t\right)
    + \\ &\quad\quad (\tfrac{\epsilon}{2})\PP\left(v^T(\nu-(1-\tfrac{\epsilon}{2})\mu - (\tfrac{\epsilon}{2}) \nu) \ge t\right).
\end{align*}
Thus, we require for each $t \geq 0$ and any $v\in S^{n-1}$ that
\begin{align*}
    \MoveEqLeft (1-\tfrac{\epsilon}{2})\PP\left(v^T \cN(0,\sigma^2 \II) + (\tfrac{\epsilon}{2}) v^T (\mu - \nu) \geq t\right) + (\tfrac{\epsilon}{2}) \PP\left(v^T (1-\tfrac{\epsilon}{2})(\nu - \mu) \geq t\right) \\ &\leq \exp(-\tfrac{t^2}{2{\sigma'}^{2}}),
\end{align*}
for some $\sigma'$ of the same order as $\sigma$ (which is equivalent to the moment generating definition of a sub-Gaussian random variable by \citet[Proposition 2.5.2]{vershynin2018high}). By a standard Gaussian tail bound (see Lemma \ref{lemma:gaussian:tail:bound}), the left-hand side is smaller than
\begin{align} \label{eq:subgaussian:lower:gaussian:tail}
    \frac{1-\tfrac{\epsilon}{2}}{2}\exp\left(-\frac{(t- \tfrac{\epsilon}{2} v^T (\mu - \nu))^2}{2\sigma^2}\right)  + (\tfrac{\epsilon}{2}) \PP\left(v^T (1-\tfrac{\epsilon}{2})(\nu - \mu) \geq t\right).
\end{align} It thus suffices to bound \eqref{eq:subgaussian:lower:gaussian:tail} with $\exp(-\tfrac{t^2}{2{\sigma'}^{2}})$.

We will now require that $\|\mu - \nu\| \asymp \sigma\sqrt{\log 1/\epsilon} \wedge \tfrac{d}{\epsilon}$ (which is possible due to Lemma \ref{lemma:star:shaped:has:line:segment}) to simultaneously ensure the mixture is sub-Gaussian and also sufficiently separated. Now consider two cases:

\textsc{Case 1:} $v^T (\mu - \nu) \ge 0$. Then the second term in \eqref{eq:subgaussian:lower:gaussian:tail} is $0$ and the first term is $\leq \frac{1-\epsilon/2}{2}$. For $t \leq \sqrt{2}\sigma' \sqrt{\log \tfrac{2}{1-\epsilon/2}}$, the desired bound on \eqref{eq:subgaussian:lower:gaussian:tail} by $\exp(-\tfrac{t^2}{2{\sigma'}^2})$ will automatically hold. On the other hand, suppose $t \geq \sqrt{2}\sigma' \sqrt{\log \tfrac{2}{1-\epsilon/2}}$. Noting that $\sqrt{\log \tfrac{2}{1-\epsilon/2}} \ge \sqrt{\log 2}\ge \epsilon \sqrt{\log 1/\epsilon}$, we have $t\geq \sqrt{2}\sigma'(\epsilon\sqrt{\log 1/\epsilon} \wedge d)\gtrsim \epsilon\|\mu-\nu\|$. Thus, in the first term of \eqref{eq:subgaussian:lower:gaussian:tail}, we have \[\tfrac{\epsilon}{2} v^T (\mu - \nu) \le \epsilon \|\mu-\nu\|\lesssim t  \]
and hence we can fold $\epsilon v^T (\mu - \nu)$ into $t$. So our desired bound will hold for $\sigma'$ slightly bigger than $\sigma$. 

\textsc{Case 2:} Suppose now $ v^T (\mu - \nu) \le 0$. Then the first term in \eqref{eq:subgaussian:lower:gaussian:tail} is smaller than $\frac{1-\epsilon/2}{2} \exp(-\tfrac{t^2}{2\sigma^2}).$ The second term is smaller than $\tfrac{\epsilon}{2}$ for $t \lesssim \sigma\sqrt{\log 1/\epsilon} $ and is $0$ otherwise. Hence it suffices to argue that 
\begin{align*}
 \tfrac{1-\epsilon/2}{2} \exp(-\tfrac{t^2}{2\sigma^2})+ \tfrac{\epsilon}{2}  &\leq \exp(-\tfrac{t^2}{2{\sigma'}^{2}})
\end{align*}
for $t \lesssim \sigma\sqrt{\log 1/\epsilon}$ (since otherwise the sum is clearly $\le \exp(-\tfrac{t^2}{2{\sigma'}^{2}})$). Well, since $\sigma' > \sigma$ this is implied if $\epsilon \leq (1 + \epsilon/2)\exp(-\tfrac{t^2}{2{\sigma'}^{2}})$. But this holds for $t \leq \sigma'\sqrt{\log 1/\epsilon}$. This completes the proof that the mixture is sub-Gaussian for both cases.

Having verified this is a sub-Gaussian setting, we now pick a corruption procedure $\tilde{\cC}$ that lets us lower bound the minimax rate. First, fix an estimator $\hat\mu$. Now observe that \begin{align} \label{eq:subgaussian:lower:bound:first:decomposition}
    \sup_{\mu\in K}\sup_{\xi}\sup_{\cC} \EE_\mu \|\hat \mu(\cC(\tilde X)) - \mu\|^2 &\geq \tfrac{1}{2} \sup_{\cC}\EE_{\tilde X \sim \cN(\mu, \sigma^2\II)^{\otimes N}} \|\hat \mu(\cC(\tilde X) - \mu\|^2 \notag \\
    &\quad\quad + \tfrac{1}{2} \sup_{\cC}\EE_{\tilde X \sim \cQ_{\epsilon}^{\otimes N}}  \|\hat \mu(\cC(\tilde X)) - \mu_{\cQ}\|^2,
\end{align} since the worst case risk across all $\mu\in K$ and sub-Gaussian distributions $\xi$ is larger than the average of the risks when $\tilde X \sim \cN(\mu, \sigma^2\II)^{\otimes N}$ and when  $\tilde X \sim \cQ_{\epsilon}^{\otimes N}$. 

Next, our adversary, given knowledge of which scenario is presented, will do nothing in the fully Gaussian setting, where $\tilde X \sim \cN(\mu, \sigma^2\II)^{\otimes N}$. On the other hand, in the mixture setting $\tilde X \sim \cQ_{\epsilon}^{\otimes N}$, the adversary does the following. Let $W=\sum_i W_i$ count the number of $\nu$ that are drawn, in place of a Gaussian. The adversary, given knowledge of the chosen distribution, observes $W$ and if $W\leq \epsilon N$, substitutes each observed $X_i=\nu$ with a freshly drawn Gaussian observation. If $W > \epsilon N$, the adversary does nothing. Hence, the adversary converts an i.i.d. mixture to either i.i.d. Gaussians or leaves the data unchanged, depending on the outcome of $W\sim \operatorname{Bin}(N,\tfrac{\epsilon}{2})$. Furthermore, note that the number of outliers is bounded by $\epsilon N$ so this is a valid corruption procedure $\tilde{\cC}$. 

Now, by \cite[Theorem 1]{kaas1980mean}, given $\epsilon<1/2$, any $k$ for which $(N-1)\epsilon/2 \leq k \leq (N + 1) \epsilon/2$ will satisfy a median property: $\PP(\operatorname{Bin}(N,\tfrac{\epsilon}{2}) \geq k) > 1/2$ and $\PP(\operatorname{Bin}(N,\tfrac{\epsilon}{2}) \leq k) > 1/2$. In our setting, we take $k = \tfrac{\epsilon}{2}N$.

Thus, the risk with this corruption procedure $\tilde{\cC}$ is lower than the most effective $\cC$, so applying $\tilde{\cC}$ by replacing the mixture with Gaussian draws when $W\le k$ (and discarding the portion of the risk when $W>k$), we have \begin{align*}
    \tfrac{1}{2} \sup_{\cC}\EE_{\tilde X \sim \cQ_{\epsilon}^{\otimes N}}  \|\hat \mu(\cC(\tilde X)) -\mu_{\cQ}\|^2 &\ge \tfrac{1}{2} \EE_{\tilde X \sim \cQ_{\epsilon}^{\otimes N}}  \|\hat \mu(\tilde{\cC}(\tilde X)) -\mu_{\cQ}\|^2 \\
    &\ge \tfrac{1}{2}\PP(W\le k)\EE_{\tilde X \sim \cN(\mu,\sigma^2\II)^{\otimes N}}  \|\hat \mu(\tilde X) -\mu_{\cQ}\|^2 \\
    &> \tfrac{1}{4}\EE_{\tilde X \sim \cN(\mu,\sigma^2\II)^{\otimes N}}  \|\hat \mu(\tilde X) -\mu_{\cQ}\|^2.
\end{align*} Note also that for the already Gaussian scenario, by comparing to a scenario with no corruption, we have \[\tfrac{1}{2} \sup_{\cC}\EE_{\tilde X \sim \cN(\mu, \sigma^2\II)^{\otimes N}} \|\hat \mu(\cC(\tilde X) - \mu\|^2 \ge \tfrac{1}{2} \EE_{\tilde X \sim \cN(\mu, \sigma^2\II)^{\otimes N}} \|\hat\mu (\tilde X) - \mu\|^2.\]

Returning to \eqref{eq:subgaussian:lower:bound:first:decomposition}, the quantity $\sup_{\mu\in K}\sup_{\xi}\sup_{\cC} \EE_\mu \|\hat \mu(\cC(\tilde X)) - \mu\|^2$ is thus lower bounded by\begin{align*}
        \MoveEqLeft \tfrac{1}{2} \EE_{\tilde X \sim \cN(\mu, \sigma^2\II)^{\otimes N}} \|\hat\mu (\tilde X) - \mu\|^2 + \tfrac{1}{4}\EE_{\tilde X \sim \cN(\mu,\sigma^2\II)^{\otimes N}}  \|\hat \mu(\tilde X) -\mu_{\cQ}\|^2 \\
        &\ge \tfrac{1}{8}\EE_{\tilde X \sim \cN(\mu,\sigma^2\II)^{\otimes N}} \left(\|\hat\mu(\tilde X) - \mu\| + \|\hat\mu(\tilde X) - \mu_{\cQ}\|\right)^2 \\
    & \gtrsim \|\mu-\mu_{\cQ}\|^2 \asymp \epsilon^2 \|\mu - \nu\|^2 \asymp \sigma^2\epsilon^2\log(1/\epsilon) \wedge d^2,
\end{align*} recalling our earlier requirement on $\mu$ and $\nu$. The second line follows from $(x+y)^2\le 2(x^2+y^2)$ for any $x,y\in\RR$. Taking the infimum over all estimators, we obtain the result.
\end{proof}

\section{Proofs for Section \ref{section:robust_gsm:gaussian:upper:bound}}

\begin{proof}[\hypertarget{proof:lemma:pruned:tree:properties}{Proof of Lemma \ref{lemma:pruned:tree:properties}}] 
    First, note that at level $1$, we have both a minimal $d$-covering and maximal $d$-packing of $K$ by our root node $\bar\nu$, and at level $2$, we have a maximal $d/c$-packing of $B(\bar\nu,d)\cap K=K$ and hence a $d/c$-covering of $K$. Moreover, there is no pruning at these steps.

    Let us now prove our claim about a $\tfrac{d}{2^{J-2}c}$-covering of $K$ for $J\ge 3$. We proceed by induction. We start with our base case $J=3$. Pick any point $x\in K$. Because at $J=2$ we had a $d/c$-covering of $K$, there exists a node $u\in \cL(2)$ such that $\|x-u\|\le d/c$. Recall to form the next level, we first created a maximal $\tfrac{d}{4c}$-packing set of $B(u, d/2)\cap K$. Since $d/c<d/2$, it follows $x\in B(u, d/2)\cap K$. Hence for some $u'\in B(u,d/2)\cap K$ that either belongs to $\cL(3)$ or was pruned from $\cL(3)$, we have $\|x - u'\|\le \tfrac{d}{4c}$. If $u'$ was pruned, then there exists some $u''\in\cL(3)$ so that $\|u-u''\|\le \tfrac{d}{4c}$, which by the triangle inequality implies $\|x - u''\|\le \tfrac{d}{2c}$. In either case, there exists a node in $\cL(3)$ within distance $\tfrac{d}{2c}$ to $x$, proving the covering claim for $J=3$.

    Suppose we have shown for some $J\ge 3$ that  $\cL(J)$ forms a $\tfrac{d}{2^{J-2}c}$-covering of $K$. Pick any $x\in K$. Then for some $u\in \cL(J)$, $\|x-u\|\le \tfrac{d}{2^{J-2}c}$. Our algorithm then constructs a maximal $\tfrac{d}{2^{J}c}$-packing of $B(u, \tfrac{d}{2^{J-1}})\cap K$. Since $\tfrac{d}{2^{J-2}c}<\tfrac{d}{2^{J-1}}$, it follows $x\in B(u, \tfrac{d}{2^{J-1}})\cap K$. Then for some $u' \in B(u, \tfrac{d}{2^{J-1}})\cap K$ that either belongs to $\cL(J+1)$ or was pruned from $\cL(J+1)$, $\|x -u'\|\le \tfrac{d}{2^{J}c}$ since we make a maximal packing. If $u'$ was pruned, there is some $u''\in \cL(J+1)$ such that $\|u'-u''\|\le \tfrac{d}{2^{J}c}$, which implies by the triangle inequality that $\|x - u''\|\le\tfrac{d}{2^{J-1}c}$. In either case, there exists a point in $\cL(J+1)$ within distance $\tfrac{d}{2^{J-1}c}$ of $x$. This completes the induction.

    Now let us verify the packing claim. Suppose at any $J\ge 3$ we have distinct points $u,u'$ in $\cL(J)$ such that $\|u-u'\|\le\tfrac{d}{2^{J-1}c}$. By definition of $\cL(J)$, neither $u$ nor $u'$ was ever pruned, hence we never constructed a set $\cT_J(u'')$ that contained $u$ or $u'$ since otherwise we would have removed those points. Suppose without loss of generality that $u$ is less than $u'$ lexicographically. This means at some point in our pruning algorithm, $u$ was the first element of our list of unprocessed nodes $\cU_J$. But because $\|u-u'\|\le \tfrac{d}{2^{J-1}c}$ and $u'$ appears after $u$ in $\cU_J$, this by definition implies $u'\in\cT_J(u)$, contradicting the fact $u'$ was not pruned. Thus, for any points $u,u'\in \cL(J)$, $\|u-u'\|> \tfrac{d}{2^{J-1}c}$, verifying that we have a $\tfrac{d}{2^{J-1}c}$-packing as claimed for all $J\ge 3$.

    Now we prove the covering property of $\cO(\Upsilon_{J-1})$ for $J\ge 2$. The $J=2$ case follows from the base case remarks earlier. Assume $J\ge 3$, consider a parent node $\Upsilon_{J-1}\in \cL(J-1)$, and let $\cO(\Upsilon_{J-1})$ be its offspring. Recall that pre-pruning we construct a maximal $\tfrac{d}{2^{J-1}c}$-packing of $B(\Upsilon_{J-1}, \tfrac{d}{2^{J-2}})\cap K$, which means it is also a $\tfrac{d}{2^{J-1}c}$-covering. Pick any $x\in B(\Upsilon_{J-1}, \tfrac{d}{2^{J-2}})\cap K$ and let us show there exists an element of $\cO(\Upsilon_{J-1})$ within distance $\tfrac{d}{2^{J-1}c}$ of $x$. Well for some $u$ in this maximal packing, we have $\|x - u\|\le \tfrac{d}{2^{J-1}c}$. If $u$ does not get pruned, then $u\in\cO(\Upsilon_{J-1})$ and the claim is proven. Suppose $u$ is pruned. That means for some other node $u'$ whose parent set $\cP(u')$ initially does not contain $\Upsilon_{J-1}$, we have $\|u'-u\|\le \tfrac{d}{2^{J-1}c}$. But recall that we then draw a directed edge from $\Upsilon_{J-1}$ to $u'$, so $u'\in\cO(\Upsilon_{J-1})$. By the triangle inequality, we have $\|x - u'\|\le \tfrac{d}{2^{J-2}c}$. In either case, we have a $\tfrac{d}{2^{J-2}c}$-covering of $B(\Upsilon_{J-1}, \tfrac{d}{2^{J-2}})\cap K$ by elements of $\cO(\Upsilon_{J-1})$.

    To see the final implication of the lemma suppose first that $J\geq 3$. Note that the cardinality of $\cO(\Upsilon_{J-1})$ is upper bounded by the cardinality of the original maximal packing set of $B(\Upsilon_{J-1},\tfrac{d}{2^{J-2}})\cap K$ at a distance $\tfrac{d}{2^{J-1}c}$, which completes the proof upon invoking the definition of local entropy. For $J = 2$ the cardinality is upper bounded by $\cMloc(d,c) \leq \cMloc(d,2c)$ as claimed.
 \end{proof}

    \begin{proof}[\hypertarget{proof:lemma:bound:level:J:intersect:ball:mu}{Proof of Lemma \ref{lemma:bound:level:J:intersect:ball:mu}}] 
    From Lemma \ref{lemma:pruned:tree:properties}, we know for $J\ge 4$ that $\cL(J-1)$ forms a $\tfrac{d}{2^{J-2}c}$-packing of $K$. Taking the intersection with $B(\mu, \tfrac{d}{2^{J-2}})$, we obtain a  $\tfrac{d}{2^{J-2}c}$-packing of $B(\mu, \tfrac{d}{2^{J-2}})\cap K$. By definition, this quantity is bounded by $\cMloc(\tfrac{d}{2^{J-2}}, c)$.

    For $J=2$, note there is only a single element in $\cL(1)$, so the cardinality is bounded by 1 and thus $\cMloc(d,c)$. For $J=3$, recall that $\cL(2)$ was formed  by a $d/c$-covering of $B(\bar\nu, d)\cap K$ where $\bar\nu$ is our root node, and this is bounded by definition by $\cMloc(d,c) \leq \cMloc(d/2,c)$ (where the last inequality follows from Lemma \ref{lemma:non:increasing:local:entropy}) which is in turn bounded by $\cMloc(d/2,2c)$.      \end{proof}

    \begin{proof}[\hypertarget{proof:lemma:cauchy:sequence}{Proof of Lemma \ref{lemma:cauchy:sequence}}] 
    Recall $\Upsilon_1$ is just our root node, and $\Upsilon_2$ is chosen from a $d/c$-packing of $B(\Upsilon_1, d)\cap K$ (and we need not prune). Then for each $k\ge 3$, observe that $\Upsilon_{k}$ is chosen by first taking a $\tfrac{d}{2^{k-1}c}$-packing of $B(\Upsilon_{k-1}, \tfrac{d}{2^{k-2}})\cap K$ and then possibly pruning the tree. If $\Upsilon_k$ was pruned, that means for some $u$ in this $\tfrac{d}{2^{k-1}c}$-packing of $B(\Upsilon_{k-1}, \tfrac{d}{2^{k-2}})\cap K$, we have $\|\Upsilon_k-u\|\le \tfrac{d}{2^{k-1}c}$. By the triangle inequality, we have \begin{align*}
        \|\Upsilon_k-\Upsilon_{k-1}\|\le \|\Upsilon_k-u\|+\|u-\Upsilon_{k-1}\| \le \tfrac{d}{2^{k-1}c}+\tfrac{d}{2^{k-2}}.
    \end{align*} If $\Upsilon_k$ was not pruned, $\|\Upsilon_k-\Upsilon_{k-1}\|\le \tfrac{d}{2^{k-2}}$. In summary, $\|\Upsilon_2-\Upsilon_1\|\le d$, and for any $k\ge 3$, $\|\Upsilon_k-\Upsilon_{k-1}\|\le \tfrac{d}{2^{k-1}c} +\tfrac{d}{2^{k-2}}$.

    If $J\ge J'\ge 2$, then we have \begin{align*}
        \|\Upsilon_J-\Upsilon_{J'}\| &\le \sum_{k=J'+1}^J \|\Upsilon_k-\Upsilon_{k-1}\| \le \sum_{k=J'+1}^J \left[\tfrac{d}{2^{k-1}c} +\tfrac{d}{2^{k-2}}\right] \\ &=\tfrac{2d}{c}\left[\tfrac{1}{2^{J'}} - \tfrac{1}{2^{J}}\right] + 4d\left[\tfrac{1}{2^{J'}} - \tfrac{1}{2^J}\right] \le \left(\tfrac{2d}{c}+4d\right)\cdot \tfrac{1}{2^{J'}} \\
        &= \tfrac{d(2+4c)}{c 2^{J'}}.
    \end{align*}
    
    If $J'=1$, then clearly $\|\Upsilon_J-\Upsilon_{1}\| \le d$ since both points belong to $K$, but $d\le \frac{d(2+4c)}{c\cdot 2}$ clearly holds. Therefore, we conclude for all $J\ge J'\ge 1$, we have $\|\Upsilon_J-\Upsilon_{J'}\|\le  \frac{d(2+4c)}{c 2^{J'}}$.
    \end{proof}

Each setting of the paper requires checking the existence of different absolute constants that appear in the main theorems, namely, Lemma \ref{lemma:existence:constants}, Lemma \ref{lemma:technical:subgaussian}, and Lemma \ref{lemma:technical:constants:subgaussian:asymmetric:version2}. It turns out to repeatedly involve properties of the functions defined in the following lemma.

\begin{lemma} \label{lemma:function:g:technical:details} Define the functions $g\colon [0,1/2)\to\RR$ and $h\colon[0,1/2)\to\RR$ by \[g(t)=(\tfrac{1}{2}+t) \log (\tfrac{1}{2}+t)+ (\tfrac{1}{2}-t) \log(1-2t)\] and $h(t) = t\left(1-\exp(\tfrac{2g(t)}{1/2-t})\right)$. Then the following properties hold:
    \begin{multicols}{2}
        \begin{enumerate}[(i)]
            \item $g, h$ are both continuous \label{enum:g:continuous}
            \item $g(t)<0$ for $t\in[0,1/2)$ \label{enum:g:negative}
            \item $\lim_{t\downarrow 0} \tfrac{-2g(t)}{1/2-t} = \log 4$ \label{enum:g:limit:0}
            \item $\lim_{t\uparrow 1/2} \tfrac{-2g(t)}{1/2-t} = \infty$ \label{enum:g:limit:1/2}
            \item $\tfrac{-2g(t)}{1/2-t}$ is increasing on $(0,1/2)$ \label{enum:g:increasing}
            \item $\lim_{t\downarrow 0}h(t)=0$\label{enum:h:limit:0}
            \item $\lim_{t\uparrow 1/2} h(t)=1/2$ \label{enum:h:limit:1/2}
            \item $h(t)/t\in(0,1)$ for $t\in(0,1/2)$ \label{enum:h:over:t}
            \item $h$ is a surjection onto $(0,1/2)$ \label{enum:h:surjection}
        \end{enumerate}
    \end{multicols}
\end{lemma}
    \begin{proof}
    Property \eqref{enum:g:continuous} is evident from standard logarithm  and exponential properties. \eqref{enum:g:negative} follows by separately checking the sign of each term of $g$. Substitute $t=0$ and apply \eqref{enum:g:continuous} to obtain \eqref{enum:g:limit:0}. To verify \eqref{enum:g:limit:1/2}, expand the definition of $g$ and one obtains a sum of terms, the first tending to $0$, the second tending to $\infty$. To obtain \eqref{enum:g:increasing}, differentiate to obtain $-8\log(t+1/2)/(1-2t)^2$ and this is positive on $(0,1/2)$.
    \eqref{enum:h:limit:0} and \eqref{enum:h:limit:1/2} then follow by applying  \eqref{enum:g:limit:0} and  \eqref{enum:g:limit:1/2}. \eqref{enum:h:over:t} holds since the input to the exponential is negative. \eqref{enum:h:surjection} follows from \eqref{enum:g:continuous}, \eqref{enum:h:limit:0}, \eqref{enum:h:limit:1/2}, and the intermediate value theorem.
    \end{proof}

Now we proceed to proving our Gaussian Type I error bound in Theorem \ref{theorem:main:testing:result}. We first prove a purely technical result which establishes that the absolute constants we use actually exist. We then state two tail bounds on Gaussian random variables that we use conditional on the size of $\delta/\sigma$.
    
\begin{lemma} \label{lemma:existence:constants} Let $\kappa\in(0,1/2]$ and $C > 2$ be fixed constants. Denote $C'=\frac{C-2}{2\sqrt{2\pi}}>0$ and let $g$ be as defined in Lemma \ref{lemma:function:g:technical:details}. Then there exist constants  $\alpha\in(0,1/2)$, and $L>0$ depending on $C$ and $\kappa$ only with the following properties. Assume $\epsilon \le 1/2-\kappa$.  Let $\delta>0$ and $\sigma>0$ be given. Denote $\varrho=\exp(-\tfrac{{C'}^2\delta^2}{\sigma^2})$. Now assume $C'\delta/\sigma > L$. Then
    \begin{multicols}{2}
    \begin{enumerate}[(i)]
        \item $\epsilon < \alpha(1-\varrho)$
        \item $(1/2-\alpha)\log(1/\varrho) \ge -2g(\alpha)$.
    \end{enumerate}
    \end{multicols} 
\end{lemma}
    \begin{proof}
    First, select $\alpha\in [0,1/2)$ such that \begin{equation} \label{eq:alpha:choice}
        1/2-\kappa < \alpha\left(1 - \exp\left(\tfrac{2g(\alpha)}{1/2-\alpha}\right)\right).
    \end{equation} Such an $\alpha$ exists since the right-hand side is a surjection onto $(0,1/2)$ as proven in \eqref{enum:h:surjection} of Lemma \ref{lemma:function:g:technical:details}.

    Next, pick $\beta>0$ such that \begin{equation} \label{eq:beta:choice}
        \Phi^{-1}(\tfrac{1}{2} + \tfrac{\beta}{C'}) > \sqrt{-\tfrac{2g(\alpha)}{(1/2-\alpha)} }.
    \end{equation} Note that the term inside the square root is strictly positive since $g(\alpha)<0$ for any $\alpha\in[0,1/2)$, using \eqref{enum:g:negative} of Lemma \ref{lemma:function:g:technical:details}.
    
    Then, we set \begin{equation}
        \label{eq:C_1:choice}C_1 = 2\Phi^{-1}(\tfrac{1}{2} + \tfrac{\beta}{C'} )/\beta.
    \end{equation} Now set $L=\beta C_1/2$ and  \begin{equation}
        \label{eq:C_2:choice}C_2 = \tfrac{C'(\Phi(L)-1/2)}{L},
    \end{equation} which is positive since $\Phi(L)>1/2$ when $L>0$. Note that \eqref{eq:C_1:choice} rearranges to $\Phi(\beta C_1/2)-1/2=\beta/C'$ so that \[C_1C_2 = \underbrace{\tfrac{2L}{\beta}}_{=C_1}\cdot \underbrace{\tfrac{C'(\Phi(L)-1/2)}{L}}_{=C_2} = 2\cdot\tfrac{C'}{\beta}\cdot\underbrace{(\Phi(\beta C_1/2)-1/2)}_{=\beta/C'}=2.\] We will use this calculation later in the paper.

    Now assume $C'\delta/\sigma>L$. This implies \begin{equation} \label{eq:varrho:relation}
        1-\varrho > 1-\exp(-L^2)=1-\exp(-\beta^2 C_1^2/4).
    \end{equation} Moreover, \eqref{eq:beta:choice} and \eqref{eq:C_1:choice} ensure \begin{equation} \label{eq:beta:c1}
        \beta C_1 = 2\Phi^{-1}(\tfrac{1}{2} + \tfrac{\beta}{C'}) > 2\sqrt{-\tfrac{2g(\alpha)}{(1/2-\alpha)}}. 
    \end{equation} Combining \eqref{eq:varrho:relation} and \eqref{eq:beta:c1} yields $1-\varrho > 1-\exp\left(\frac{2g(\alpha)}{1/2-\alpha}\right)$. Then using \eqref{eq:alpha:choice} we have $1/2-\kappa< \alpha(1-\varrho)$, so assuming $\epsilon\leq1/2-\kappa$ ensures (i) holds. Finally, $1-\varrho\ge 1-\exp\left(\frac{2g(\alpha)}{1/2-\alpha}\right)$ rearranges to yield (ii). 
    \end{proof}

    \begin{lemma} \label{lemma:cdf:convex} Fix a constant $L>0$. Then for all $z\in[0,L]$, we have \begin{equation*}
   1- \Phi(z)\le \tfrac{1}{2} -\tfrac{z}{L}\cdot \left(\Phi(L) - \tfrac{1}{2}\right).
\end{equation*}
\end{lemma}
    \begin{proof}
        Observe that the function $z\mapsto 1-\Phi(z)$ is convex for $z \geq 0$, as its second derivative is $(-\phi(z))' = z \phi(z) > 0$ for $z > 0$, where $\phi:=\Phi'$. Thus, for any value of $z \in [0, L]$ we have \[
        1- \Phi(z) \leq \left(1-\tfrac{z}{L}\right)\cdot\tfrac{1}{2} + \tfrac{z}{L}\cdot (1-\Phi(L)) = \tfrac{1}{2} -\tfrac{z}{L}\cdot \left(\Phi(L) - \tfrac{1}{2}\right).\]
    \end{proof}

\begin{lemma}[{\cite[equation (2.122)]{wozencraft_principles_1965}}] \label{lemma:gaussian:tail:bound} For $z\ge 0$, \[1-\Phi(z)\le \frac{1}{2}\exp(-z^2/2).\]
\end{lemma}

\begin{proof}[\hypertarget{proof:theorem:main:testing:result}{Proof of Theorem \ref{theorem:main:testing:result}}] 

We first bound $\sup_{\mu: \|\mu- \nu_1\| \leq \delta}\PP_{\mu}(\psi = 1)$, and at the end, argue that the full claim follows by symmetry. Using Lemma \ref{lemma:existence:constants}, pick constants $\alpha \in(0,1/2)$, and $L>0$ depending on $C$ and $\kappa$ only, such that conditions (i) and (ii) hold when $C'\delta/\sigma > L$. Recall from the lemma that we defined $C' = (C-2)/(2\sqrt{2\pi})>0$ and  $\varrho=\exp\left(-\frac{{C'}^2\delta^2}{\sigma^2}\right)\in(0,1)$. Also, we introduced constants $C_1$ and $C_2$ in the proof of that lemma, specifically \eqref{eq:C_1:choice} and \eqref{eq:C_2:choice}.

We introduce some additional shorthand notation.  Let $A_i$ be the event that $\|\tilde X_i - \nu_1\| \geq \|\tilde X_i - \nu_2\|$, and let $B_i$ be the event that the possibly corrupted data $X_i$ satisfies $\|X_i - \nu_1\| \geq \| X_i - \nu_2\|$. In our notation $\psi=1$ is the event that at least $N/2$ of $B_1,\dots, B_N$ occur. We define another indicator random variable $\tilde\psi$ in two different ways: when $C'\delta/\sigma\le L$ we let $\tilde\psi=1$ be the event that at least $N/2-C_2N\delta/(2\sigma)$ of $A_1,\dots,A_N$ occur, and when $C'\delta/\sigma > L$, we let $\tilde\psi=1$  be the event at least $N/2 - N\alpha(1-\varrho)$ of $A_1,\dots, A_N$ occur.

For both of these cases, we use the following strategy: For each $\mu\in\RR^n$ satisfying $\|\mu-\nu_1\|\le\delta$, we will upper bound $\PP_{\mu}(A_i)$ and as a consequence $\PP_{\mu}(\tilde\psi = 1)$ using binomial concentration results. It will turn out that $\PP_{\mu}(\psi = 1)\le \PP_{\mu}(\tilde\psi = 1)$, so we will have bounded the desired quantity.

The proof of \citet[Lemma II.5]{neykov2022minimax} demonstrated that the original data $\tilde X_i$ satisfies \[\sup_{\mu: \|\mu - \nu_1\|\leq \delta}\PP_{\mu}(A_i) \le \PP(N(m,\tau^2)\ge 0)= \PP(\cN(0,1)\ge -m/\tau),\] where $m = (-1+2/C)\|\nu_1-\nu_2\|^2<0$ and $\tau^2=4\sigma^2\|\nu_1-\nu_2\|^2$. But $|m/\tau|\ge C'\sqrt{2\pi}\delta/\sigma$, hence we obtain  \begin{align} \label{eq:gaussian:CDF:bound}
        \sup_{\mu: \|\mu - \nu_1\|\leq \delta}\PP_{\mu}(A_i) \le 1-\Phi\left(\tfrac{C'\sqrt{2\pi}\delta}{\sigma}\right)\le  1- \Phi\left(\tfrac{C'\sqrt{2}\delta}{\sigma}\right) \le 1- \Phi\left(\tfrac{C'\delta}{\sigma}\right).
    \end{align} 

    Fix any $\mu\in\RR^n$ such that $\|\mu-\nu_1\|\le \delta$, and let us now consider two cases. First, assume  $C'\delta/\sigma \le L$, where $L$ is the fixed constant from the previous lemma. Then apply our bound on $1-\Phi(z)$ from Lemma \ref{lemma:cdf:convex} with $z=C'\delta/\sigma$ and the right-most bound in \eqref{eq:gaussian:CDF:bound}, we obtain
\begin{align} \label{eq:p(A_i^c)}
        \PP_{\mu}(A_i) &\leq  1- \Phi\left(\tfrac{C'\delta}{\sigma}\right)\le \tfrac{1}{2} - \tfrac{C'\delta}{\sigma}\cdot \tfrac{\Phi(L) - \tfrac{1}{2}}{L} = \tfrac{1}{2} - \tfrac{C_2\delta}{\sigma}.
\end{align} In the last step we substituted $C_2$ from  \eqref{eq:C_2:choice} in Lemma \ref{lemma:existence:constants}.

 Let us show  $\PP_{\mu}(\psi=1) \le \PP_{\mu}(\tilde\psi = 1)$. It suffices to show the event $\tilde\psi = 0$ is a subset of the event $\psi = 0$. Suppose $\tilde\psi = 0$, i.e., no more than $N/2-C_2N \delta/(2\sigma)$ of $A_1,\dots,A_N$ occur. Now we corrupt at most $N\epsilon$ of the $\tilde X_i$. Observe that if $\tilde X_i$ is uncorrupted, then the event $A_i$ occurs if and only if $B_i$ occurs. It is only possible that $A_i$ occurs but not $B_i$ or vice versa if the $i$th datapoint is corrupted. In the worst case,   no more than \[N/2-C_2 N\delta/(2\sigma) + N\epsilon\le N/2+N\epsilon(1-C_1C_2/2)= N/2\] of $B_1,\dots,B_N$ occur, where we used the assumption $\delta/\sigma\ge C_1\epsilon$ and that $C_1C_2= 2$ as computed in Lemma \ref{lemma:existence:constants}. This implies $\psi = 0.$  Thus, $\PP_{\mu}(\psi=1) \le \PP_{\mu}(\tilde\psi = 1)$.

Let us now bound $\PP_{\mu}(\tilde\psi = 1)$ using our bound on $\PP_{\mu}(A_i).$ The event $\tilde\psi=1$ corresponds to no more than $N/2 + C_2\delta N/2\sigma$ of $A_1^c,\dots, A_N^c$ being true. The probability of this event is bounded by the tail probability of a binomial random variable, i.e., $\PP(\mathrm{Bin}(N,\PP_{\mu}(A_i^c)) \le N/2+C_2\delta N/(2\sigma))$. Set $p=  1/2 + C_2\delta/\sigma$ and $\zeta =C_2\delta/(2\sigma)$ so that $p-\zeta= 1/2 + C_2\delta/(2\sigma)$. Then note that $\PP_{\mu}(A_i^c)\ge p$ using \eqref{eq:p(A_i^c)}.
Using the Hoeffding bound, we can write
\begin{align*}
   \PP_{\mu}(\tilde\psi=1)&\le \PP(\mathrm{Bin}(N,\PP_{\mu}(A_i^c)) \le \overbrace{N/2 +C_2\delta N/(2\sigma)}^{=N(p-\zeta)})  \\ &\le \PP(\mathrm{Bin}(N,\PP_{\mu}(A_i^c)) \leq N (\PP_{\mu}(A_i^c) - \zeta)) \\ &\leq \exp(-2N \zeta^2) \\
   &= \exp(-\tfrac{C_2^2 N\delta^2}{2\sigma^2}).
\end{align*}
Taking the supremum over $\mu\in\RR^n$ such that $\|\mu-\nu_1\|\le \delta$ finishes the claim in this first case.

For our second case, suppose $C'\delta/\sigma > L$. Then we have by the penultimate bound in \eqref{eq:gaussian:CDF:bound} and the normal tail bound in Lemma \ref{lemma:gaussian:tail:bound} that 
\begin{align*}
    \PP_{\mu}(A_i) \leq 1- \Phi\left(\tfrac{C'\sqrt{2}\delta}{\sigma}\right) \leq \tfrac{1}{2}\exp\left(-\tfrac{{C'}^2\delta^2}{\sigma^2}\right) = \tfrac{\varrho}{2}.
\end{align*}

Recall for this second case that $\tilde{\psi}=1$ if at least $N/2-N\alpha(1-\varrho)$ of the $A_i$ occur, where $\alpha,\varrho$ are from Lemma \ref{lemma:existence:constants}. Let us show $\PP_{\mu}(\psi = 1) \le \PP_{\mu}(\tilde\psi = 1)$ by again showing the event $\tilde \psi = 0$ is a subset of $\psi=0$. Suppose no more than $N/2-N\alpha(1-\varrho)$ of the $A_i$ occur, i.e., $\tilde\psi=0$. By the same logic as before, in the worst case, no more than \[N/2-N\alpha(1-\varrho)+N\epsilon = N/2+N(\epsilon-\alpha(1-\varrho)) \le N/2\] of the $B_i$ occur, where we used $\epsilon < \alpha(1-\varrho)$ from (i) in Lemma \ref{lemma:existence:constants}. Thus $\psi=0$. Hence $\PP_{\mu}(\psi = 1) \le \PP_{\mu}(\tilde\psi = 1)$. 

Let us now bound $\PP_{\mu}(\tilde\psi = 1)$. This is bounded by $\PP(\mathrm{Bin}(N,\PP_{\mu}(A_i^c)) \le N/2 +  N\alpha(1-\varrho))$. Now set  $p=1-\varrho/2$ and $\zeta =(1/2-\alpha)(1-\varrho)$. Then $p-\zeta = \frac{1}{2}+\alpha-\alpha\varrho$, and also note $\PP_{\mu}(A_i^c)\ge p$. Observe that \begin{align}
    \PP(\mathrm{Bin}(N,\PP_{\mu}(A_i^c)) \le N/2+ N\alpha(1-\varrho)) &= \PP(\mathrm{Bin}(N,\PP_{\mu}(A_i^c)) \le N(p-\zeta)) \nonumber\\
    &\le \PP(\mathrm{Bin}(N,p) \le N(p-\zeta)) \nonumber\\
    &\le e^{ -N\cdot D((p-\zeta) \| p)} \label{chernoff:bound},
\end{align} where in the next to last inequality we used that $\mathrm{Bin}(N, p) \leq_{\operatorname{st}} \mathrm{Bin}(N, \PP_{\mu}(A_i^c))$\footnote{Here we use $\leq_{\operatorname{st}}$ to denote (first order) stochastic dominance, which means that $\PP(\mathrm{Bin}(N, p) \geq x) \leq \PP(\mathrm{Bin}(N, \PP_{\mu}(A_i^c)) \geq x)$ for any $x \in \RR$.} since they use the same sample size and $p \leq \PP_{\mu}(A_i^c)$ \citep[Theorem 1(a)]{klenke2010stochastic}, and in the last inequality we used a Chernoff bound for the binomial distribution \citep[see e.g.,][Section 1.3]{dubhashi_panconesi_2009} where we define $D(q\|p) = q \log \tfrac{q}{p} + (1-q) \log \tfrac{1-q}{1-p}$. Let us now lower bound $D(p-\zeta \|p)$. We have
\begin{align*}
    D(p-\zeta \| p ) & = (1/2+\alpha-\alpha\varrho) \overbrace{\log \textstyle\frac{1/2+\alpha-\alpha\varrho}{1-\varrho/2}}^{\le 0} + (1/2-\alpha+\alpha\varrho) \overbrace{\log\textstyle\frac{1/2-\alpha+\alpha\varrho}{\varrho/2}}^{\ge 0} \\
    &\ge (1/2+\alpha) \log \textstyle\frac{1/2+\alpha-\alpha\varrho}{1-\varrho/2}+ (1/2-\alpha) \log\textstyle\frac{\frac{1}{2}-\alpha+\alpha\varrho}{\varrho/2}\\ &\ge
    (1/2+\alpha) \log (1/2+\alpha)+ (1/2-\alpha) \log\textstyle\frac{1/2-\alpha+\alpha\varrho}{\varrho/2}  \\
    &\ge  (1/2+\alpha) \log (1/2+\alpha)+ (1/2-\alpha) \log\textstyle\frac{1/2-\alpha}{\varrho/2} \\
    &= g(\alpha)+ (1/2-\alpha)\log(1/\varrho). 
\end{align*} 

The first line is from substituting $p$ and $\zeta$ into $D(p-\zeta\|p)$. The second uses the respective sign on the logarithmic terms and sets $\varrho=0$ in the terms multiplied to the logarithm. The third follows by setting $\varrho=0$ in the first logarithm term by noting $\frac{d}{d\varrho} \frac{\frac{1}{2}+\alpha-\alpha\varrho}{1-\varrho/2} = \frac{1-2\alpha}{(2-\varrho)^2}>0$ so the entire logarithmic term is increasing in $\varrho$. The fourth line uses monotonicity of the logarithm and sets $\varrho=0$ in the numerator of the argument to the second logarithm term. Then we split up the logarithm terms to isolate the portion depending on $\varrho$, recalling the definition of $g(\alpha)$ from Lemma \ref{lemma:function:g:technical:details}. 

Now recall since $C'\delta/\sigma>L$, we have from (ii) of  Lemma \ref{lemma:existence:constants} that $(1/2-\alpha)\log(1/\varrho)\ge -2g(\alpha)>0$. Returning to $D(p-\zeta \| p )$, we have \begin{align*}
       D(p-\zeta \| p )  &\ge g(\alpha)+ (1/2-\alpha)\log(1/\varrho) \\
        &\ge (1/2)(1/2-\alpha)\log(1/\varrho). 
\end{align*} 
To complete the argument, note that $\log(1/\varrho) = {C'
}^2\delta^2/\sigma^2$ and use \eqref{chernoff:bound}.

Having bounded $\sup_{\mu:\|\mu-\nu_1\|\le \delta} \PP_{\mu}(\psi=1)$, we claim the same bound for $\sup_{\mu:\|\mu-\nu_2\|\le \delta} \PP_{\mu}(\psi=0)$ follows by symmetry.  Recall $\psi=1$ precisely when at least $N/2$ of $X_i$ satisfy $\|X_i - \nu_2\| \le \|X_i - \nu_1\|$. On the other hand,  $\psi =0 $ is equivalent to at least $N/2$ of the $X_i$ satisfying the reverse inequality, $\|X_i - \nu_1\| < \|X_i - \nu_2\|$. We can set $A_i$ to be the event $\|\tilde X_i -\nu_1\|< \|\tilde X_i -\nu_2\|$ and $B_i$ the event $\|X_i -\nu_1\|< \| X_i -\nu_2\|$. Our Gaussian tail bound for $\tilde X_i$ adapted from \citet{neykov2022minimax} is unaffected by strict inequalities versus weak ones. Specifically, we have the modified bound: \[\sup_{\mu: \|\mu - \nu_2\|\leq \delta}\PP_{\mu}(A_i) \le \PP(N(m,\tau^2) > 0)= \PP(N(m,\tau^2) \ge 0)=\PP(\cN(0,1)\ge -m/\tau).\] The remainder of the proof can proceed without issue.
\end{proof}

We now verify our Algorithm \ref{algorithm:robust}'s error bound in Theorem \ref{theorem:gaussian:version}. We start with some auxiliary lemmas that handle the scenario where our $\delta/\sigma$ requirements are met for the first $\tilde{J}$ steps. We will require some set intersection properties---one of which we use here and the other later in the unbounded case.

 \begin{lemma} \label{lemma:set:complement:induction} Let $J\ge 2$ be an integer and let $A_1,A_2,\dots,A_J$ be a sequence of events. Then \begin{enumerate}[(i)]
    \item $\PP(A_J) \le \PP(A_1)+\PP(A_1^c\cap A_2) + \sum_{j=3}^J \PP(A_1^c\cap A_{j-1}^c\cap A_j)$, \label{enum:set:intersection:general}
     \item  $\PP(A_J) \le \PP(A_1)+\sum_{j=2}^J \PP(A_j\cap A_{j-1}^c)$. \label{enum:set:intersection:simple}
 \end{enumerate}
\end{lemma}
    \begin{proof} 
        Write $AB$ to denote the intersection of sets $A$ and $B$. Observe \begin{align}
            A_J &\subseteq A_1 \cup (A_1^c A_2) \cup (A_1^c A_2^c A_3) \cup (A_1^c A_2^c A_3^c A_4) \cup \ldots \cup(A_1^c A_2^c \ldots A_{J-1}^c A_J) \notag \\
            &= A_1\cup (A_1^cA_2) \cup \bigcup_{j=3}^J \left[ A_1^c\dots A_{j-1}^c A_j\right] \notag \\
            &\subseteq A_1\cup (A_1^cA_2)\cup \bigcup_{j=3}^J \left[ A_1^c A_{j-1}^c A_j\right] \label{eq:set:intersection:general} \\
            &\subseteq A_1\cup (A_1^cA_2)\cup \bigcup_{j=3}^J \left[A_{j-1}^c A_j\right]. \label{eq:set:intersection:simple}
        \end{align} Applying the union bound to \eqref{eq:set:intersection:general} yields \eqref{enum:set:intersection:general}, and similarly the union bound with \eqref{eq:set:intersection:simple} yields \eqref{enum:set:intersection:simple}.
    \end{proof}

\begin{lemma} \label{lemma:for:theorem:gaussian} Let $\eta_J$ be defined as in Theorem \ref{theorem:gaussian:version} and $C=c/2-1$. Suppose $\tilde J$ is such that  \eqref{eq:robust:theorem:condition} holds and also $\frac{d}{2^{\tilde J-1}(C+1)} \ge C_1(\kappa)\epsilon\sigma$. Then for each such $1\le J \le \tilde{J}$ we have \[ \PP\left(\|\Upsilon_{J} - \mu\| > \tfrac{d}{2^{J-1}}\right)
             \le 2\cdot\mathbbm{1}(J>1)\exp(-\tfrac{N\eta_{J}^2}{2\sigma^2}).\]
\end{lemma}
    \begin{proof}
    Note that if $\tilde J$ satisfies \eqref{eq:robust:theorem:condition} and also $\frac{d}{2^{\tilde J-1}(C+1)} \ge C_1(\kappa)\epsilon\sigma$, it is clear that any $1\le J\le \tilde{J}$ satisfies these conditions. Observe that for $3 \le j\le \tilde J$, if $\|\Upsilon_{j-1} - \mu\| \leq \tfrac{d}{2^{j-2}}$, then $\Upsilon_{j-1}=u$ for some $u\in \cL(j-1)\cap B(\mu, \tfrac{d}{2^{j-2}})$.  Applying a union bound, setting $\delta = \frac{d}{2^{j-1}(C+1)}$, substituting the update rule for $\Upsilon_j$ from Algorithm \ref{algorithm:robust}, and finally dropping the intersection, we have
        \begin{align*}
            \MoveEqLeft \PP(\|\Upsilon_j - \mu\| > \tfrac{d}{2^{j-1}}, \|\Upsilon_{j-1} - \mu\| \leq \tfrac{d}{2^{j-2}})\\
            &\leq \sum_{u \in \cL(j-1) \cap B(\mu,  \tfrac{d}{2^{j-2}})} \!\!\!\!\!\!\!\PP(\|\Upsilon_j - \mu\| > \tfrac{d}{2^{j-1}}, \Upsilon_{j-1} = u) \\
            & = \sum_{u \in \cL(j-1) \cap B(\mu,  \tfrac{d}{2^{j-2}})} \!\!\!\!\!\!\!\PP\Big(\big\|\argmin_{\nu \in \cO(u)} T(\delta, \nu, \cO(u)) - \mu\big\| > (C+1)\delta, \Upsilon_{j-1} = u\Big)\\
            & \leq \sum_{u \in \cL(j-1) \cap B(\mu, \tfrac{d}{2^{j-2}})}\!\!\!\!\! \PP\Big(\big\|\argmin_{\nu \in \cO(u)} T(\delta, \nu, \cO(u)) - \mu\big\| > (C+1)\delta\Big).
            %&\leq \sum_{p \in \cL(j-1) \cap B(\mu, d/2^{j-2})} \PP(\|\Upsilon_j - \mu\| \geq d/2^{j-1} | \Upsilon_{j-1} = p)
        \end{align*}
        Note $\cL(j-1) \cap B(\mu, \tfrac{d}{2^{j-2}})$ has cardinality upper bounded by $\cMloc(\tfrac{d}{2^{j-2}},2c)$ by Lemma \ref{lemma:bound:level:J:intersect:ball:mu}.  Set $K'=B(u,\tfrac{d}{2^{j-2}})\cap K\subseteq K$, and recall from Lemma \ref{lemma:pruned:tree:properties} that $\cO(u)$ forms a $\tfrac{d}{2^{j-2}c} = \tfrac{d}{2^{j-1}(C+1)}=\delta$-covering of $K'$ with cardinality bounded by $\cMloc(\tfrac{d}{2^{j-2}},2c)$. Moreover, $\delta \ge C_1(\kappa)\sigma \epsilon$ holds by assumption. Applying Lemma \ref{lemma:tournament} and noting the summands become constant in $u$, we may bound the probability term above: 
        \begin{align*}
            \MoveEqLeft \PP(\|\Upsilon_j - \mu\| > \tfrac{d}{2^{j-1}}, \|\Upsilon_{j-1} - \mu\| \leq \tfrac{d}{2^{j-2}}) \\ &\le \Big[\cMloc(\tfrac{d}{2^{j-2}}, 2c)\Big]^2 \cdot \exp\Big(-\tfrac{C_3(\kappa) N\delta^2}{\sigma^2}\Big) \\
&=\Big[\cMloc\Big(\tfrac{2(C+1)\eta_j}{\sqrt{C_3(\kappa)}}, 2c\Big)\Big]^2\cdot \exp\Big(-\tfrac{N \eta_j^2}{\sigma^2} \Big).
        \end{align*}
    Here we used the definition $\eta_j = \tfrac{d\sqrt{C_3(\kappa)}}{2^{j-1}(C+1)}=\sqrt{C_3(\kappa)}\delta$. 
    
    Now define the event $A_j = \{\|\Upsilon_j-\mu\| >\tfrac{d}{2^{j-1}}\}$ and observe for any $1\le J \le\tilde J$ that \[\PP(A_J) \le \PP(A_1)+ \PP(A_2\cap A_1^c) +\sum_{j=3}^{J} \PP(A_j\cap A_{j-1}^c),\] a proof of which is given in Lemma \ref{lemma:set:complement:induction}. We have already bounded $\PP(A_j\cap A_{j-1}^c)$ for $3\le j\le \tilde J$.  Note that $\PP(A_1)=0$ since $\Upsilon_1$ and $\mu$ both belong to a set of diameter $d$. Recall that for the second level, we construct a maximal $d/c$-packing of $B(\bar\nu, d)\cap K$, which is therefore a $d/c$-covering of $B(\bar\nu, d)\cap K$ (without pruning). So we may apply Lemma \ref{lemma:tournament} and conclude \begin{align*}
        \PP(A_2\cap A_1^c) &=\PP(A_2) \le \cMloc(d, c)\cdot \exp\left(-\tfrac{C_3(\kappa) N\delta^2}{\sigma^2}\right) \\ &\le \left[\cMloc(\tfrac{2(C+1)\eta_2}{\sqrt{C_3(\kappa)}}, 2c)\right]^2\cdot \exp\left(-\tfrac{N \eta_2^2}{\sigma^2} \right).
    \end{align*} Note that $\eta_J$ is decreasing in $J$ while $\cMloc$ is a non-increasing function
    , so we may bound $\cMloc(\tfrac{2(C+1)\eta_j}{\sqrt{C_3(\kappa)}}, 2c)$ with $\cMloc(\tfrac{2(C+1)\eta_J}{\sqrt{C_3(\kappa)}}, 2c)$ for any $j\le J$. Therefore for $1\le J \le \tilde{J}$ we have \begin{align*}
             \PP\left(\|\Upsilon_J - \mu\| > \tfrac{d}{2^{J-1}}\right) &\le \left[\cMloc\left(\tfrac{2(C+1)\eta_J}{\sqrt{C_3(\kappa)}}, 2c\right)\right]^2\sum_{j=2}^J \exp\left(-\tfrac{N \eta_j^2}{\sigma^2} \right) \\
             &\le \mathbbm{1}(J>1)\left[\cMloc\left(\tfrac{2(C+1)\eta_J}{\sqrt{C_3(\kappa)}}, 2c\right)\right]^2 \frac{a_J}{1-a_J},
        \end{align*} where we set $a_J = \exp(-\tfrac{N\eta_J^2}{\sigma^2})$.

        Now suppose that $\tfrac{N\eta_J^2}{\sigma^2} > 2\log \left[\cMloc\left(\tfrac{2(C+1)\eta_J}{\sqrt{C_3(\kappa)}}, 2c\right)\right]^2\vee \log 2$, i.e., \eqref{eq:robust:theorem:condition} holds. Observe that $\tfrac{N\eta_J^2}{\sigma^2}>\log 2$ implies $a_J < 1/2$. Then we conclude \begin{align*}
             \PP\left(\|\Upsilon_J - \mu\| >\tfrac{d}{2^{J-1}}\right)
             &\le \mathbbm{1}(J>1) \exp(\tfrac{N\eta_J^2}{2\sigma^2})\cdot a_J \cdot\underbrace{\tfrac{1}{1-a_J}}_{\le 2} \\
             &\le  2\cdot\mathbbm{1}(J>1) \exp(\tfrac{N\eta_J^2}{2\sigma^2})\cdot\underbrace{\exp(-\tfrac{N\eta_J^2}{\sigma^2})}_{=a_J} \\
             &= 2\cdot\mathbbm{1}(J>1)\exp(-\tfrac{N\eta_J^2}{2\sigma^2}),
        \end{align*} giving us an upper bound on $\PP(A_J)$ so long as \eqref{eq:robust:theorem:condition} holds as well as our $\frac{d}{2^{J-1}(C+1)} \ge C_1(\kappa)\epsilon\sigma$ condition.
    \end{proof}

\begin{lemma} \label{lemma:for:theorem:gaussian:part2} Let $\eta_J$ be defined as in Theorem \ref{theorem:gaussian:version}. Suppose $\tilde J$ is such  \eqref{eq:robust:theorem:condition} holds and also $\frac{d}{2^{\tilde J-1}(C+1)} \ge C_1(\kappa)\epsilon\sigma$. Let $C_4(\kappa) = \tfrac{(19+16C)^2}{4C_3(\kappa)}$. Then if $\nu^{\ast\ast}$ denotes the output after at least $J^{\ast}$ iterations, we have \[\EE_X\|\mu-\nu^{\ast\ast}\|^2\le C_4(\kappa)\eta_{\tilde{J}}^2 + \mathbbm{1}(J^{\ast}>1)\cdot 4 C_4(\kappa)\cdot  \tfrac{\sigma^2}{N}\exp\Big(-\tfrac{N\eta_{\tilde{J}}^2}{2\sigma^2}\Big). \]
\end{lemma}
    \begin{proof}
         By Lemma \ref{lemma:for:theorem:gaussian}, we have for $1\le J \le\tilde J$ that \[ \PP\left(\|\Upsilon_J - \mu\| > \tfrac{d}{2^{J-1}}\right)
             \le 2\cdot\mathbbm{1}(J>1)\exp(-\tfrac{N\eta_J^2}{2\sigma^2}),\] i.e., upper bounding $\PP(A_j)$ where $A_j=\{\|\Upsilon_j-\mu\| > \tfrac{d}{2^{j-1}}\}$.

             Now recall our definition of $J^{\ast}$ and let $\nu^{\ast}=\Upsilon_{J^{\ast}}$ be the output of $J^{\ast} - 1$ steps. Define $B_J$ to be the event that $\|\mu-\nu^{\ast}\|> \omega \eta_J$ where $\omega = \tfrac{7+6C}{2\sqrt{C_3(\kappa)}}.$ We now upper bound $\PP(B_J)$ assuming \eqref{eq:robust:theorem:condition} holds. By Lemma \ref{lemma:cauchy:sequence}, for any $1\le J\le \tilde{J}$ (which implies $J\le J^{\ast}$) we have \[\|\Upsilon_J-\nu^{\ast}\|\le \tfrac{d(2+4c)}{c 2^{J}}= \left(\tfrac{1}{C+1}+4\right)\cdot \tfrac{d}{2^J}\] using  $c=2(C+1)$.  Note that $\PP(A_J^c)\le \PP(B_J^c)$, for if $\|\Upsilon_J-\mu\|\le \tfrac{d}{2^{J-1}}$, then by the triangle inequality and the definition of $\omega$ and $\eta_J$, we have \[\|\nu^{\ast}-\mu\|\le \|\nu^{\ast}-\Upsilon_J\| + \|\Upsilon_J-\mu\| \le \left(\tfrac{1}{C+1}+4\right)\cdot \tfrac{d}{2^J} +\tfrac{d}{2^{J-1}}= \omega \eta_J.\] Thus, for $1\le J\le\tilde{J}$,  \begin{equation} \label{eq:bounded:mu_minus_nu_ast}
                 \underbrace{\PP(\|\mu-\nu^{\ast}\|> \omega \eta_J)}_{=\PP(B_J)} \le \PP(A_J) \le 2\cdot\mathbbm{1}(J>1)\exp(-\tfrac{N\eta_J^2}{2\sigma^2}). 
             \end{equation}But this result in fact holds for all integers $J\le 0$. For such $J$ we have $\omega \eta_J > \frac{d}{2^{J-2}}\ge 4d$ so both sides of the inequality are 0. Moreover, note that $\mathbbm{1}(J>1)\le \mathbbm{1}(J^{\ast}>1)$ for all $J\le\tilde J$. 

        Now, since $\bigcup_{-\infty<J\le \tilde{J}}[\eta_J,\eta_{J-1})=[\eta_{\tilde{J}},\infty)$ and $\eta_J = \eta_{J-1}/2$, we observe that any $x\ge \eta_{\tilde{J}}$ belongs to some interval $[\eta_J,\eta_{J-1})$ for $J\le \tilde{J}$ and therefore satisfies $2\omega x \ge 2\omega \eta_J=\omega \eta_{J-1}$. Hence for $x\ge \eta_{\tilde{J}}$, \begin{align}
            \PP(\|\mu-\nu^{\ast}\|> 2\omega x) &\le \PP(\|\mu-\nu^{\ast}\|> \omega \eta_{J-1}) \notag \\
            &\le  2\cdot\mathbbm{1}(J^{\ast}>1)\exp(-\tfrac{N\eta_{J-1}^2}{2\sigma^2}) \notag \\
            &\le  2\cdot\mathbbm{1}(J^{\ast}>1)\exp(-\tfrac{Nx^2}{2\sigma^2}). \label{eq:theorem:mu:nu_ast:probability:bound}
        \end{align} Note the use of the monotonicity of $x\mapsto \exp(-\tfrac{Nx^2}{2\sigma^2})$ in the final inequality.

        We are interested in performing at least $J^{\ast}$ steps, so let ${\nu^{\ast\ast}}$ be the output of $J^{\ast\ast}\ge J^{\ast}$ steps. Then by Lemma \ref{lemma:cauchy:sequence},  for $x\ge \eta_{\tilde{J}}$ (noting $\tilde J\le J^{\ast}$ means $\eta_{\tilde J}\ge \eta_{J^{\ast}}$) we have  \begin{align*}
            \|\nu^{\ast} - \nu^{\ast\ast}\| &= \|\Upsilon_{J^{\ast}} - \Upsilon_{J^{\ast\ast}+1}\| \le \tfrac{d(2+4c)}{c2^{J^\ast}} = \tfrac{5+4C}{7+6C}\omega \eta_{J^{\ast}} \le\tfrac{5+4C}{7+6C}\omega \eta_{\tilde J} \\ &\le \tfrac{5+4C}{7+6C}\omega x.
        \end{align*} The triangle inequality implies \begin{align} \label{eq:theorem:triangle:inequality:for:J_ast_}
            \|\mu-\nu^{\ast\ast}\| \le \|\mu-\nu^{\ast}\| +\|\nu^{\ast} - \nu^{\ast\ast}\|  \le  \|\mu-\nu^{\ast}\|+\tfrac{5+4C}{7+6C}\omega x.
        \end{align} Set $\omega' = (2+\tfrac{5+4C}{7+6C})\omega$, which after substitution with $\omega=\tfrac{7+6C}{2\sqrt{C_3(\kappa)}}$ becomes $\omega' = \tfrac{19+16C}{2\sqrt{C_3(\kappa)}}$. Then for $x\ge \eta_{\tilde{J}}$, using \eqref{eq:theorem:mu:nu_ast:probability:bound} and \eqref{eq:theorem:triangle:inequality:for:J_ast_},  \begin{align}
            \PP(\|\mu-\nu^{\ast\ast}\|> \omega' x) &\le \PP(\|\mu-\nu^{\ast}\|+\tfrac{5+4C}{7+6C}\omega x> \omega' x) \notag \\
            &= \PP(\|\mu-\nu^{\ast}\|> 2\omega x) \notag \\
            &\le  2\cdot\mathbbm{1}(J^{\ast}>1)\exp(-\tfrac{Nx^2}{2\sigma^2}). \label{eq:theorem:mu:nu_ast_ast:probability:bound}
        \end{align}   

               Therefore, \begin{align*}
            \EE_X \|\mu-\nu^{\ast\ast}\|^2 &= \int_0^{\infty} \PP(\|\mu-\nu^{\ast\ast}\|^2 > x)\mathrm{d}x \\ &= 2{\omega'^2}\int_0^{\infty} u \cdot \PP(\|\mu-\nu^{\ast\ast}\| > \omega' u) \mathrm{d}u \notag\\
            &\leq 2{\omega'^2}\int_0^{\eta_{\tilde{J}}} u\mathrm{d}u +2{\omega'^2}\int_{\eta_{\tilde{J}}}^{\infty}  u\cdot \PP(\|\mu-\nu^{\ast\ast}\| > \omega'u) \mathrm{d}u \notag\\
            &\le  {\omega'}^2\eta_{\tilde{J}}^2+  \mathbbm{1}(J^{\ast}>1)\cdot 4{\omega'}^2
            \int_{\eta_{\tilde{J}}}^{\infty} u\exp(-\tfrac{Nu^2}{2\sigma^2})\mathrm{d}u \notag\\
            &=  {\omega'}^2\eta_{\tilde{J}}^2+  \mathbbm{1}(J^{\ast}>1)\cdot 4{\omega'}^2\cdot \tfrac{\sigma^2}{N}\exp\left(-\tfrac{N\eta_{\tilde{J}}^2}{2\sigma^2}\right) \notag\\
            &= \tfrac{(19+16C)^2}{4C_3(\kappa)}\eta_{\tilde{J}}^2 + \mathbbm{1}(J^{\ast}>1)\cdot \tfrac{(19+16C)^2}{C_3(\kappa)}\cdot  \tfrac{\sigma^2}{N}\exp\left(-\tfrac{N\eta_{\tilde{J}}^2}{2\sigma^2}\right). 
        \end{align*} 
    \end{proof}

\begin{proof}[\hypertarget{proof:theorem:gaussian:version}{Proof of Theorem \ref{theorem:gaussian:version}}] 
        If $J^{\ast}=1$, then $\eta_{J^{\ast}}\asymp d$. Then observe that $\max(d^2,\epsilon^2\sigma^2)\wedge d^2= d^2$ by considering the cases whether $d\le \epsilon \sigma$ or $d>\epsilon\sigma$. Since $\nu^{\ast},\mu\in K$, clearly $\EE_{X} \|\nu^{\ast}-\mu\|^2\le d^2$. Thus, we assume $J^{\ast}>1$.
        
        Recall we set our initial input $\Upsilon_1$, and call $\Upsilon_{k+1}$ the output of $k$ iterations of the algorithm. Note that $J^{\ast}<\infty$ since the left-hand side of \eqref{eq:robust:theorem:condition} is an increasing function of $\eta_J$ (thus decreasing with $J$) while the right-hand side is a non-increasing function of $\eta_J$ (thus non-decreasing with $J$).

        We now consider two cases, based on whether the $\frac{d}{2^{J-1}(C+1)} \ge C_1(\kappa)\epsilon\sigma$ condition fails prior to $J^{\ast}$. If it does fail, we further consider two sub-cases, depending on whether $\epsilon\gtreqqless \tfrac{1}{\sqrt{N}}$.
        
        \textsc{Case 1:} Assume $\frac{d}{2^{J^{\ast}-1}(C+1)} \ge C_1(\kappa)\epsilon\sigma$. By Lemma \ref{lemma:for:theorem:gaussian:part2} with $\tilde{J}=J^{\ast}$, 
        \begin{equation} \label{eq:theorem:penultimate:bound}
            \EE_X\|\mu-\nu^{\ast\ast}\|^2\le C_4(\kappa)\eta_{J^{\ast}}^2 + \mathbbm{1}(J^{\ast}>1)\cdot 4C_4(\kappa)\cdot  \tfrac{\sigma^2}{N}\exp\left(-\tfrac{N\eta_{J^{\ast}}^2}{2\sigma^2}\right). 
        \end{equation}

        Note that we have $N\eta_{J^{\ast}}^2/\sigma^2 > \log 2$ by definition of $J^{\ast}$, which implies  $\tfrac{\sigma^2}{N} <\eta_{J^{\ast}}^2/\log 2$. The exponential term is clearly bounded by $1$. Thus, we bound \eqref{eq:theorem:penultimate:bound} with   \[ C_4(\kappa)\eta_{J^{\ast}}^2+ 4C_4(\kappa)\tfrac{\sigma^2}{N} \le C_4(\kappa)\left(1+ \tfrac{4}{\log 2}\right)\eta_{J^{\ast}}^2.\] But clearly $\eta_{J^{\ast}}^2\lesssim \max(\eta_{J^{\ast}}^2, \epsilon^2\sigma^2)$ and we know the $d^2$ is always an upper bound, proving our claimed bound.

         \textsc{Case 2:} Assume for some $J'\in\{1,2,\dots, J^{\ast}\}$ we have \begin{equation} \label{eq:J:prime:condition}
             \tfrac{d}{2^{J'-1}(C+1)}<C_1(\kappa)\epsilon\sigma.
         \end{equation} If $J'=1$, then this means $d\lesssim \epsilon\sigma$, so that our claimed upper bound reduces to just $d^2$ and trivially holds. Suppose $J'>1$ and assuming $J'$ is chosen minimally, we know for all $1\le J\le J'-1$ that $\tfrac{d}{2^{J-1}(C+1)}\ge C_1(\kappa)\epsilon\sigma$. Then by Lemma \ref{lemma:for:theorem:gaussian:part2} with $\tilde{J}=J'-1$, we obtain   \begin{equation}
             \label{eq:theorem:last:bound}
             \EE_X\|\mu-\nu^{\ast\ast}\|^2\le C_4(\kappa)\eta_{J'-1}^2 + \mathbbm{1}(J^{\ast}>1)\cdot 4C_4(\kappa)\cdot  \tfrac{\sigma^2}{N}\exp\left(-\tfrac{N\eta_{J'-1}^2}{2\sigma^2}\right).
         \end{equation}

        \textsc{Case 2(a):} Suppose $\epsilon\le \tfrac{C_5(\kappa)}{\sqrt{N}}$ where $0<C_5(\kappa)<\tfrac{\sqrt{\log 2}}{\sqrt{C_3(\kappa)}C_1(\kappa)}$. Then $\epsilon\sigma\le C_5(\kappa)\tfrac{\sigma}{\sqrt{N}}$. By definition of $J'$ in \eqref{eq:J:prime:condition}, we obtain \begin{align*}
            \frac{d}{2^{J'-1}} < C_1(\kappa)(C+1)\epsilon\sigma \le C_5(\kappa)C_1(\kappa)(C+1)\cdot \tfrac{\sigma}{\sqrt{N}}.
        \end{align*} Rearranging and recalling the definition of $\eta_{J}$, we have \[\tfrac{N\eta_{J'}^2}{\sigma^2} \le \left[\sqrt{C_3(\kappa)}C_5(\kappa)C_1(\kappa)\right]^2 <\log 2.\] Thus, the condition in \eqref{eq:robust:theorem:condition} does not hold, and by maximality of $J^{\ast}$, we must have $J'> J^{\ast}$, which means $J'=J^{\ast}$. But if \eqref{eq:robust:theorem:condition} does not hold for $J^{\ast}$, we must have $J^{\ast}=1$, in which case the theorem trivially holds. 

        \textsc{Case 2(b):} Suppose $\epsilon\ge \tfrac{C_5(\kappa)}{\sqrt{N}}$. This implies $\epsilon^2\sigma^2\ge C_5(\kappa)^2\tfrac{\sigma^2}{N}$, and the exponential is always $\le 1$. Thus, the entire second term in \eqref{eq:theorem:last:bound} in fact $\lesssim \epsilon^2\sigma^2$. Moreover, note by definition of $J'$ in \eqref{eq:J:prime:condition} that \[\eta_{J'-1}=2\cdot \eta_{J'}  = \tfrac{2\sqrt{C_3(\kappa)}}{C+1}\cdot \tfrac{d}{2^{J'-1}} <2\sqrt{C_3(\kappa)}C_1(\kappa)\epsilon\sigma.\] So our bound in \eqref{eq:theorem:last:bound} is of the form $\eta_{J'-1}^2 + \epsilon^2\sigma^2 \lesssim \epsilon^2\sigma^2$, which is certainly less than $\max(\epsilon^2\sigma^2,\eta_{J^{\ast}}^2)$. Since the rate of any estimator outputting points in $K$ is always bounded by $d^2$, the proof is complete.        
    \end{proof}

   \begin{proof}[\hypertarget{proof:theorem:robust:minimax:rate:attained:gaussian}{Proof of Theorem \ref{theorem:robust:minimax:rate:attained:gaussian}}]  We split the proof in multiple cases depending on whether $\epsilon \ge \tfrac{1}{\sqrt{N}}$ and  $\tfrac{N{\eta^{\ast}}^2}{\sigma^2} > 8\log 2$.

   Before we proceed with our cases we begin with the edge case when $\eta^* = 0$. Note that this implies $\cMloc(\eta, c) = 1$ for any $\eta$ sufficiently small, which implies that the set $K$ consists of a single point (hence $d = 0$) and our algorithm will trivially output that point achieving the minimax rate of $0$. Thus we may assume  $\eta^* > 0$. Note that if $\eta^* > 0$ this implies that $d > 0$, which means we can make $\cMloc(\eta, c)$ bigger than any fixed constant by taking $c$ sufficiently large. We now proceed with case work.
    
    \textsc{Case 1:} $\epsilon \ge \tfrac{1}{\sqrt{N}}$ and $\tfrac{N{\eta^{\ast}}^2}{\sigma^2} > 8\log 2$. 
    
  First, let us derive a lower bound on the minimax rate. Observe (when $\eta^* > 0$) that \begin{align}
            \log \cMloc(\eta^{\ast}/4,c) &\ge \lim_{\gamma \rightarrow 0} \log \cMloc(\eta^{\ast} - \gamma,c) \ge \lim_{\gamma \rightarrow 0}\tfrac{N{(\eta^{\ast} - \gamma)}^2}{\sigma^2} \notag \\ &=  \tfrac{N{\eta^{\ast}}^2}{2\sigma^2} + \tfrac{N{\eta^{\ast}}^2}{2\sigma^2} \notag \ge \tfrac{N{\eta^{\ast}}^2}{2\sigma^2}+ 4\log 2 \notag \\ &\ge 4\cdot \tfrac{N{(\eta^{\ast}/4)}^2}{2\sigma^2}+ 4\log 2  \ge 4\left(\tfrac{N(\eta^{\ast}/4)^2}{2\sigma^2}\vee \log 2\right).\label{M:loc:inequality:eta:star}
        \end{align}
        Observe that the first inequality above holds since  $\eta^* \neq 0$. Thus by \eqref{M:loc:inequality:eta:star}, $\eta^{\ast}/4$ satisfies the condition in Lemma \ref{lemma:lower:bound:first:version}, so the minimax rate is lower bounded by ${\eta^{\ast}}^2$ up to constants. The minimax rate is also lower bounded by $\epsilon^2\sigma^2 \wedge d^2$ and thus $\max({\eta^{\ast}}^2, \epsilon^2\sigma^2 \wedge d^2)$ since Lemma \ref{corruptions:lower:bound} applies in this case. Since $\max({\eta^{\ast}}^2\wedge d^2, \epsilon^2\sigma^2\wedge d^2)\le \max({\eta^{\ast}}^2, \epsilon^2\sigma^2 \wedge d^2)$, we obtain the claimed  rate as a lower bound.

      Now we obtain the upper bound.  By Theorem \ref{theorem:gaussian:version}, $\max\left(\eta_{J^{\ast}}^2\wedge d^2, \epsilon^2\sigma^2\wedge d^2\right)$ is an upper bound on the minimax rate. If we can find a $\tilde\eta\asymp \eta^{\ast}$  such that ${\tilde\eta}^2\gtrsim \eta_{J^{\ast}}^2 $, which implies \[\max\left(\eta_{J^{\ast}}^2\wedge d^2, \epsilon^2\sigma^2\wedge d^2\right)\lesssim \max({\tilde\eta}^2\wedge d^2, \epsilon^2\sigma^2\wedge d^2),\] then $\max({\eta^{\ast}}^2\wedge d^2, \epsilon^2\sigma^2\wedge d^2)$ upper bounds the minimax rate.
        
        We start by obtaining a $\tilde\eta>0$ that satisfies \eqref{eq:robust:condition:rescaled} (in lieu of $\eta_J$).  Set $\beta = \min(\tfrac{1}{\sqrt{2}},\tfrac{c}{2}\sqrt{\tfrac{2}{C_3(\kappa)}})\in(0,1/\sqrt{2}]$, and pick a constant $D>1$ such that $D\beta>1$. We take $\tilde\eta = \sqrt{2}D\eta^{\ast}$. Then using the definition of $\eta^{\ast}$ as a supremum 
        and the non-increasing property of $\cMloc(\cdot, c)$, we have  \begin{align}
            N \tilde\eta^2 / \sigma^2 &= \beta^{-2}\cdot 2N(D\beta {\eta^{\ast}})^2/\sigma^2 \ge  4N(D\beta {\eta^{\ast}})^2/\sigma^2 \label{eq:minimax:attained:beta:scaling} \\ &>  4\log \cMloc(D\beta {\eta^{\ast}},c) \ge 4\log \cMloc\bigg(\underbrace{\frac{D c\sqrt{2}\eta^{\ast}}{2\sqrt{ C_3(\kappa)}}}_{ \geq D\beta\eta^{\ast}} , c\bigg) \notag\\ &=   4\log \cMloc\left(\frac{c\tilde\eta}{2\sqrt{ C_3(\kappa)}} , c\right) = 2\log \left[\cMloc\left(\frac{c\tilde\eta}{2\sqrt{ C_3(\kappa)}} , c\right)\right]^2. \notag     \end{align} We also have by the assumption in Case 1 that \[ N \tilde\eta^2 / \sigma^2 = 2N D^2{\eta^{\ast}}^2/\sigma^2 >16D^2\log 2>\log 2.\] Thus $\tilde\eta$ satisfies \eqref{eq:robust:condition:rescaled}. 
            
             Define the non-decreasing map $\phi\colon (0,\infty)\to\RR$ by  \[\phi(x)=Nx^2/\sigma^2 - 2\log \left[\cMloc\left(\frac{cx}{2\sqrt{ C_3(\kappa)}} , c\right)\right]^2 \vee \log 2. \] We have $\phi(\tilde\eta)>0$ since $\tilde\eta$ satisfies \eqref{eq:robust:condition:rescaled}. 
            
            First, suppose some $\eta_J$ for $J\ge 1$ satisfies \eqref{eq:robust:condition:rescaled}. Then we know by maximality of $J^{\ast}$ that $\eta_{J^{\ast}+1} = \eta_{J^{\ast}}/2$ will satisfy $\phi(\eta_{J^{\ast}+1})< 0<\phi(\tilde\eta)$. Thus $\tilde\eta \ge \eta_{J^{\ast}+1}=\eta_{J^{\ast}}/2$ since $\phi$ is non-decreasing. If no such $\eta_J$ exists, we must have $J^{\ast}=1$, and $\eta_{J^{\ast}}$ satisfies $\phi(\eta_{J^{\ast}}/2)\le \phi(\eta_{J^{\ast}})<0< \phi(\tilde\eta)$. Again $\tilde\eta \ge \eta_{J^{\ast}}/2$. We have therefore found a $\tilde\eta\asymp \eta^{\ast}$ such that ${\tilde\eta}^2 \gtrsim \eta_{J^{\ast}}^2$ and by our previous remarks, this completes the upper bound and thus Case 1.

        \textsc{Case 2:} $\epsilon \le \tfrac{1}{\sqrt{N}}$ and $\tfrac{N{\eta^{\ast}}^2}{\sigma^2} > 8\log 2$. 

        The upper bound argument in Case 1 is completely unchanged since we did not use our stated assumption on $\epsilon$. Thus $\max({\eta^{\ast}}^2\wedge d^2, \epsilon^2\sigma^2\wedge d^2)$ is still an upper bound on the minimax rate. 
        
        Proceeding to the lower bound, we can repeat an identical argument as in Case 1 and conclude that ${\eta^{\ast}}^2$ and therefore ${\eta^{\ast}}^2\wedge d^2$ is a lower bound on the minimax rate. But we cannot invoke Lemma \ref{corruptions:lower:bound}, and must directly show $\epsilon^2\sigma^2$ (and consequently $\epsilon^2\sigma^2\wedge d^2$) is a lower bound on the minimax rate, in which case $\max({\eta^{\ast}}^2\wedge d^2, \epsilon^2\sigma^2\wedge d^2)$ will also be a lower bound

        To see this, note that if we take $\tilde\eta = \epsilon\sigma$, then $\tfrac{N{\tilde
        \eta}^2}{\sigma^2} =N\epsilon^2 \le 1$ since $\epsilon\le\tfrac{1}{\sqrt{N}}$ while on the other hand $\tfrac{N {\eta^{\ast}}^2}{\sigma^2}>8\log 2$. Thus, $\epsilon\sigma=\tilde\eta < \eta^{\ast}$, and since ${\eta^{\ast}}^2$ is a lower bound on the minimax rate, so is $\epsilon^2\sigma^2$. This completes Case 2. 

        \textsc{Case 3:} $\epsilon \ge \tfrac{1}{\sqrt{N}}$ and $\tfrac{N{\eta^{\ast}}^2}{\sigma^2} \le 8\log 2$. 

        We first show that $d^2$ and hence ${\eta^{\ast}}^2\wedge d^2$ is a lower bound. By definition of $\eta^{\ast}$ as a supremum, we have $\log \cMloc(2\eta^{\ast},c) < N(2\eta^{\ast})^2/\sigma^2 \le 32\log 2$. Now we claim that it is impossible to fit a line segment of length $4\eta^{\ast}$ inside $K$. Suppose not, i.e., there is some line segment $l$ of this length contained in $K$. Let $B$ be the ball formed by taking $l$ as one of its diameters. Partition $l$ into $2c$ sub-intervals of length $2\eta^{\ast}/c$. By choosing $c$ large enough, we can produce a $2\eta^{\ast}/c$-packing set of $K\cap B$ of cardinality exceeding $\exp(32\log 2)$, violating our claim that  $\log \cMloc(2\eta^{\ast},c)\le 32\log 2$. Thus, no such line segment exists. However, Lemma \ref{lemma:star:shaped:has:line:segment} implies a line segment of length $d/3$ exists. So we conclude $d/3\le 4\eta^{\ast}\le 8\sigma\sqrt{2\log 2/N}$.

        Now take $\tilde\eta = d/6$. Then $\tfrac{2N{\tilde\eta}^2}{\sigma^2} =  \tfrac{2N (d/6)^2}{\sigma^2}\le 64\log 2$ using our bound on $d$. On the other hand, we can show that for sufficiently large $c$, $\log \cMloc(\tilde\eta,c)> 64 \log 2$. To see this, take one of the line segments in $K$ of length $d/3$ and again partition this diameter into sub-intervals of length $d/(6c)$. For sufficiently large $c$, we can pick $\exp(64\log 2)$ points of distance $d/(6c)$-apart in the set $K$ intersected with a ball of radius $d/6$ centered at a point in $K$. Hence $\log \cMloc(\tilde\eta,c) >\tfrac{2N{\tilde\eta}^2}{\sigma^2} \vee 4\log 2$. This means the condition in Lemma \ref{lemma:lower:bound:first:version} with $\tilde\eta=d/3$ holds, and the minimax rate is lower bounded by $d^2$ (up to constants) and in turn ${\eta^{\ast}}^2\wedge d^2$.

        Then since $\epsilon\ge \tfrac{1}{\sqrt{N}}$, we have by Lemma \ref{corruptions:lower:bound} that $\epsilon^2\sigma^2$ and hence $\epsilon^2\sigma^2\wedge d^2$ is a lower bound, and this proves the minimax lower bound of $\max({\eta^{\ast}}^2\wedge d^2, \epsilon^2\sigma^2\wedge d^2)$.

        Proceeding to the upper bound, we know $d^2$ is always an upper bound, and $d^2\le 12^2\eta^{\ast2}$. Hence ${\eta^{\ast}}^2\wedge d^2$ is an upper bound up to constants, which implies $\max({\eta^{\ast}}^2\wedge d^2, \epsilon^2\sigma^2\wedge d^2)$ is an upper bound as desired.

        \textsc{Case 4:} $\epsilon \le \tfrac{1}{\sqrt{N}}$ and $\tfrac{N{\eta^{\ast}}^2}{\sigma^2} \le 8\log 2$. 

        The proof that ${\eta^{\ast}}^2\wedge d^2$ is a lower bound as well as that $\max({\eta^{\ast}}^2\wedge d^2, \epsilon^2\sigma^2\wedge d^2)$ is an upper bound is identical to Case 3, as neither of the arguments used our condition on $\epsilon$. The only remaining claim to prove is that $\epsilon^2\sigma^2\wedge d^2$ is a lower bound.

        It suffices to show that $\tilde\eta=\tfrac{1}{3\sqrt{2}}\cdot (\epsilon\sigma \wedge d)$ satisfies the condition in Lemma \ref{lemma:lower:bound:first:version} so that $\epsilon^2\sigma^2 \wedge d^2$ forms a lower bound up to constants. Well, our assumption on $\epsilon$ implies $\tfrac{2N{\tilde\eta}^2}{\sigma^2} \leq N\epsilon^2/9 \le 1/9$. Similar to before, take a connected subset of a diameter of length $d/(3\sqrt{2})$ and repeat the argument from the previous case to produce a $d/(3c\sqrt{2})$-packing set of cardinality at least $\exp(4\log 2)$ for a sufficiently large choice of $c$. Then we have $\log \cMloc(\tfrac{1}{3\sqrt{2}}(\epsilon\sigma \wedge d),c) \ge \log \cMloc(d/(3\sqrt{2}),c)\ge 4\log 2$ and also $\log \cMloc(\tfrac{1}{3\sqrt{2}}(\epsilon\sigma \wedge d),c)\ge 1/9\ge 2N{\tilde\eta}^2/\sigma^2$. Thus, the condition in Lemma \ref{lemma:lower:bound:first:version} applies, demonstrating that $\epsilon^2\sigma^2 \wedge d^2$ is a lower bound for sufficiently large $c$.
    \end{proof}

 \section{Proofs for Section \ref{section:robust_gsm:subgaussian:upper:bound}}
 
\subsection{Proofs: Symmetric sub-Gaussian Noise}

First, we develop a concentration bound using a local central limit theorem from  \citet[Chapter VII, Theorem 10]{petrov1976sums}. After another lemma specifying some appropriate constants, we then prove our main Type I error bound. Recall that $R_i$ refers to the Gaussian noise we add to our data, and $\widetilde{X}_i$ is our notation for the uncorrupted data.

\begin{lemma} \label{lemma:local:clt} There exists universal constants $D_1, D_2>0$ that are independent of the sub-Gaussian distribution such that for any integer $k>2\pi(1+D_1)^2$, if $\sqrt{k}\delta/\sigma < \frac{2 D_2}{(C-2)}\cdot \sqrt{\log\left(\frac{k}{2\pi(1+D_1)^2}\right)}$, then
     \[\PP\left(0 \leq \sqrt{k} (\overline{ \tilde X_i} + \overline{ R_i} - \mu)^Tv \le \frac{\sqrt{k}\delta(C-2)}{2}\right)  \ge \frac{\delta(C-2)}{2\sigma D_2}.\]
\end{lemma}
   \begin{proof}[\hypertarget{proof:lemma:local:clt}{Proof of Lemma \ref{lemma:local:clt}}] 
        Set $Y_i=\sigma^{-1}(\tilde X_i-\mu+R_i)^T v$ which is sub-Gaussian with a constant parameter. Define $B_k=\sum_{i\in G_j}\EE[Y_i^2]$, noting this quantity is the same for any group $G_j$. Using moment properties of sub-Gaussian random variables \citep[Proposition 2.5.2]{vershynin2018high}, we have $B_k\lesssim k$. But also $B_k\gtrsim k$ since $\EE|Y_i|^2 \ge \EE|\sigma^{-1}R_i^T v|^2 = 1$. Hence $B_k\asymp k$ up to absolute constants that do not depend on the particular choice of sub-Gaussian distribution. Moreover, the moment properties imply $\sum_{i\in G_j}\EE|Y_i|^3\lesssim k=O(B_k)$.   As a consequence, we also have that $B_k\to\infty$ as $k\to\infty$. 

The characteristic function of $Y_i$ satisfies  \begin{align*}
    |\EE \exp(itY_i)| &= |\EE \exp(it\sigma^{-1} (\tilde X_i  - \mu)\T v)|\cdot|\EE\exp(it \sigma^{-1} R_i\T v)| \leq \exp(-t^2/2),
\end{align*} where in the last step we used that $|\EE\exp(i\theta)|\le 1$ for any $\theta\in\RR$ and that  $\sigma^{-1}R_i^Tv\sim \cN(0,1)$. Therefore, \citet[assumption (2.6)]{petrov1976sums} is implied if \[\int_{|t|\ge \gamma} \exp(-kt^2/2) \mathrm{d}t  =O(1/k)\] for every $\gamma>0$.

 Now, with some rearranging and an application of Chebyshev's inequality,
\begin{align*}
 \int_{|t| \geq \gamma} \exp(-kt^2/2) \mathrm{d}t &= \sqrt{2\pi k^{-1}} \int_{|t| \geq \gamma} \frac{\exp(-t^2/(2k^{-1}))}{\sqrt{2\pi k^{-1}}} \mathrm{d}t \\
    & = \sqrt{2\pi k^{-1}}\cdot \PP_{Z \sim \cN(0,k^{-1})}(|Z| > \gamma) \\ &\leq \frac{(\sqrt{2\pi k^{-1}}) k^{-1}}{\gamma^2} = \frac{\sqrt{2\pi}}{\gamma^2k\sqrt{k}} = O(1/k).
\end{align*} 

Before we apply \citet[Chapter VII, Theorem 10]{petrov1976sums}, observe that the theorem is stated for sufficiently large $k$. We want $k$ to be a true absolute constant, i.e., independent of the choice of sub-Gaussian random variable. In the theorem's proof, this requirement on $k$ is only applied when using the implied condition that $L_k\sqrt{B_k}\lesssim 1$, where $L_k=B_k^{-3/2}\sum_{i\in G_j} \EE|Y_i|^3$. From our earlier remarks, we know $B_k=\sum_{i\in G_j} \EE|Y_i|^2\gtrsim k$ and $\sum_{i\in G_j} \EE|Y_i|^3\lesssim k$, both up to absolute constants independent of the distribution. Indeed, we have
\begin{align*}
    L_k\sqrt{B_k} = \frac{\sum_{i\in G_j} \EE|Y_i|^3}{\sum_{i\in G_j} \EE|Y_i|^2} \lesssim 1.
\end{align*} Note that $B_k\gtrsim k$ critically relied on our addition of Gaussian noise. 

Having verified the assumptions of the theorem, we have that for our choice of fixed $k$ the density $p_k$ of $\frac{1}{\sqrt B_k}\sum_{i\in G_j} Y_i$ exists and satisfies
\begin{align*}
    \sup_{x} \left|p_k(x) - \frac{1}{\sqrt{2\pi}} \exp(-x^2/2)\right| = O\left(\frac{1}{\sqrt{k}}\right).
\end{align*} For some $D_1>0$, we have $p_k(x) \ge -\frac{D_1}{\sqrt{k}} + \frac{1}{\sqrt{2\pi}}\exp(-x^2/2)$. Thus, we set $D'=\sqrt{\log\left(\frac{k}{2\pi(1+D_1)^2}\right)}$ and assume $x\in[0,D']$ to obtain  \begin{align*}
    p_k(x) &> -\frac{D_1}{\sqrt{k}} + \frac{1}{\sqrt{2\pi}}\exp(-{D'}^2/2)\ge\frac{1}{\sqrt{k}}. 
\end{align*} We will assume $k> 2\pi(1+D_1)^2$ so that the expression inside the logarithm in $D'$ is larger than $1$.

Observe that there exists a distribution independent absolute constant $D_2>0$ such that $\sqrt{B_k} \le D_2\sqrt{k}$, which implies $D_2^{-1}\le \frac{\sqrt{k}}{\sqrt{B_k}} $. Moreover, note that $p_k$ is equivalently the density of \[\frac{\sqrt{k}}{\sqrt{B_k}}\cdot\frac{1}{\sqrt{k}}\sum_{i\in G_j}Y_i.\] Therefore, returning to the lemma's claim, we have 
 \begin{align*}
   \MoveEqLeft \PP\left(0 \leq \sqrt{k} (\overline{ \tilde X_i} + \overline{ R_i} - \mu)^Tv \le \frac{\sqrt{k}\delta(C-2)}{2}\right) \\ &=  \PP\left(0 \leq \frac{1}{\sqrt{k}}\sum_{i\in G_j} Y_i \le \frac{\sqrt{k}\delta(C-2)}{2\sigma}\right) \\
    &= \PP\left(0 \leq \frac{\sqrt{k}}{\sqrt{B_k}}\cdot \frac{1}{\sqrt{k}}\sum_{i\in G_j} Y_i \le \frac{\sqrt{k}\delta(C-2)}{2\sigma}\cdot \frac{\sqrt{k}}{\sqrt{B_k}}\right) \\
    &\ge \PP\left(0 \leq \frac{\sqrt{k}}{\sqrt{B_k}}\cdot \frac{1}{\sqrt{k}}\sum_{i\in G_j} Y_i \le \frac{\sqrt{k}\delta(C-2)}{2\sigma D_2}\right) \\
    &= \int_0^{\sqrt{k}\delta(C-2)/(2\sigma D_2)} p_k(x) \mathrm{d}x \\
    &\ge \frac{\delta(C-2)}{2\sigma D_2}\asymp \delta/\sigma,
\end{align*} where the third line used $D_2^{-1}\le \frac{\sqrt{k}}{\sqrt{B_k}}$ and the last line assumed that $\frac{\sqrt{k}(C-2)\delta}{2\sigma D_2} < D'$ (so that our $p_k(x)>1/\sqrt{k}$ bound holds and we pull $p_k$ out of the integrand), which is the case when $\delta/\sigma < \frac{2D_2}{(C-2)\sqrt{k}}\cdot \sqrt{\log\left(\frac{k}{2\pi(1+D_1)^2}\right)}.$ 
    \end{proof}

\begin{lemma} \label{lemma:technical:subgaussian} Let $C>2,\sigma>0$ be given. Let the function $g(t)$ be defined as in Lemma \ref{lemma:function:g:technical:details}. Then there exist absolute constants $k\in\NN,\alpha\in(0,1/2),\gamma\in(0,1)$, and $L>0$ that do not depend on the distribution with the following properties. Suppose $\epsilon \in[0,\gamma/k]$. Fix $\delta >0$. Set $\beta=1-\frac{16 \log 2}{L^2(C-2)^2}$ and denote $\varrho = \exp\left(-\frac{\beta k \delta^2(C-2)^2}{16\sigma^2}\right)$. Define the following conditions:
\begin{multicols}{2}
\begin{enumerate}[(i)]
    \item $\beta\in (0,1)$
    \item $k\epsilon < \alpha(1-\varrho)$
    \item $(1/2-\alpha)\log(1/\varrho)\ge -2g(\alpha)$
\end{enumerate}
    \end{multicols} Then (i) always hold, and if $\sqrt{k}\delta/\sigma>L$, (ii) and (iii) also hold.
\end{lemma}
    \begin{proof}
        Let $D_1, D_2>0$ be the distribution independent absolute constants from Lemma \ref{lemma:local:clt}. Take $k=\lceil 2\pi(1+D_1)^2\exp(\tfrac{4\log 8}{D_2^2}) \vee 2\pi(1+D_1)^2 \rceil + 1$. Set $L = \frac{2 D_2}{(C-2)}\cdot \sqrt{\log\left(\frac{k}{2\pi(1+D_1)^2}\right)}$, noting $k\ge 2\pi(1+D_1)^2$ by assumption. For convenience, we use this opportunity to define $C_1=\tfrac{4k D_2}{C-2}$ and $C_2=\frac{C-2}{2D_2}$, which we point out satisfy $C_1C_2=2k$ for later use. 
        
        Observe that $\beta\in(0,1)$ so long as $L^2(C-2)^2>16\log 2$, or equivalently, that $k>2\pi(1+D_1)^2\exp\left(\tfrac{4\log 2}{D_2^2}\right)$. Our choice of $k$ (noting the $4\log 8$ in place of $4\log 2$ inside the exponential term) satisfies this, yielding (ii).

        Next, observe that $1-\exp\left(\tfrac{2g(\alpha)}{1/2-\alpha}\right)\in(0,1)$ for $\alpha\in(0,1/2)$ using \eqref{enum:h:over:t} from Lemma \ref{lemma:function:g:technical:details}, and moreover, $1-\exp\left(-\tfrac{\beta L^2(C-2)^2}{16}\right)\in(0,1)$ since $\beta>0$ from (i). Now, \begin{equation} \label{eq:technical:alpha:compared:to:beta}
            1-\exp\left(\tfrac{2g(\alpha)}{1/2-\alpha}\right) < 1-\exp\left(-\tfrac{\beta L^2(C-2)^2}{16}\right)
        \end{equation} holds provided $\tfrac{\beta L^2(C-2)^2}{16}> \tfrac{-2g(\alpha)}{1/2-\alpha}$, noting $\tfrac{-2g(\alpha)}{1/2-\alpha}>0$  for $\alpha\in(0,1/2)$ by \eqref{enum:g:negative} of  Lemma \ref{lemma:function:g:technical:details}. Then from \eqref{enum:g:limit:0} and \eqref{enum:g:limit:1/2} of the same lemma, we have $\lim_{x\to 0^+}\frac{-2g(\alpha)}{1/2-\alpha} = \log 4$ and  $\lim_{x\to {1/2}^-}\frac{-2g(\alpha)}{1/2-\alpha} = \infty$. Thus, if we require $\tfrac{\beta L^2(C-2)^2}{16}>\log 4$ (which occurs if $ k > 2\pi (1+D_1)^2\exp\left(\tfrac{4\log 8}{D_2^2}\right)$, as we assumed), we may pick $\alpha>0$ sufficiently small such that $\tfrac{\beta L^2(C-2)^2}{16}>\tfrac{-2g(\alpha)}{1/2-\alpha}>\log 4$. Thus, \eqref{eq:technical:alpha:compared:to:beta} holds, so we may pick a $\gamma>0$ such that
        \begin{equation} \label{eq:gamma:exists}
            \alpha\left(1-  \exp\left(\tfrac{2g(\alpha)}{1/2-\alpha}\right)\right) < \gamma<  \alpha \left(1- \exp\left(-\tfrac{\beta L^2(C-2)^2}{16}\right)\right).
        \end{equation} 
        
        Now suppose $\sqrt{k}\delta/\sigma>L$. Then we have by the definition of $\varrho$, \eqref{eq:gamma:exists}, and our assumption $\epsilon \in[0,\gamma/k]$ that \begin{align*}
            \alpha(1-\varrho) > \alpha\left(1- \exp\left(-\tfrac{\beta L^2(C-2)^2}{16}\right)\right) > \gamma\ge k\epsilon.
        \end{align*} This yields (ii).

        Lastly, observe (iii) is equivalent to $\alpha(1-\varrho)\ge \alpha\left(1- \exp\left(\tfrac{2g(\alpha)}{1/2-\alpha}\right)\right)$. Well using $\sqrt{k}\delta/\sigma>L$ and \eqref{eq:gamma:exists}, we have \[\alpha(1-\varrho)\ge \alpha\left(1-\exp\left(-\tfrac{\beta L^2(C-2)^2}{16}\right)\right) \ge \alpha\left(1-  \exp\left(\tfrac{2g(\alpha)}{1/2-\alpha}\right)\right).\] 
    \end{proof}
  
    \begin{proof}[\hypertarget{proof:theorem:subgaussian:main:testing:result}{Proof of Theorem \ref{theorem:subgaussian:main:testing:result}}] 
    We will upper bound $\sup_{\mu: \|\mu- \nu_1\| \leq \delta}\PP_{\mu}(\psi = 1)$ and by symmetry we will obtain the bound for $\sup_{\mu: \|\mu- \nu_2\| \leq \delta}\PP_{\mu}(\psi = 0)$ by the same argument we gave at the end of the proof of Theorem \ref{theorem:main:testing:result} but with sub-Gaussian tail bounds instead. In particular, we note that since we add Gaussian noise, the tail bounds we use in this proof are the same with strict or weak inequality.
    
    Using Lemma \ref{lemma:technical:subgaussian}, we pick constants $k\in\NN, \alpha\in(0,1/2), L>0,\beta\in(0,1)$ such that conditions (i) from the lemma always holds and (ii) and (iii) hold when $\sqrt{k}\delta/\sigma>L$. Write $\varrho = \exp\left(-\frac{\beta k \delta^2(C-2)^2}{16\sigma^2}\right)$. Recall we also defined $C_1$ and $C_2$ in the proof of Lemma \ref{lemma:technical:subgaussian}.

    Let $A_j$ be the event that $\big\|k^{-1}\sum_{i\in G_j}(\tilde{X}_i + R_i) - \nu_1\big\| \geq \big\|k^{-1}\sum_{i\in G_j}(\tilde{X}_i + R_i) - \nu_2\big\|$, and let $B_j$ be the event $\big\|k^{-1}\sum_{i\in G_j}(X_i + R_i) - \nu_1\big\| \geq \big\|k^{-1}\sum_{i\in G_j}(X_i + R_i) - \nu_2\big\|$. Then $\psi$ is an indicator variable that takes value $1$ if at least $N/(2k)$ of $B_1,\dots, B_{N/k}$ occur. When $\sqrt{k}\delta/\sigma \le L$, let $\tilde\psi$ to be the indicator variable that takes value $1$ if at least $\tfrac{N}{2k} - \tfrac{C_2\delta}{\sigma}\cdot \tfrac{N}{2k}$ of $A_1,\dots, A_{N/k}$ occur. When  $\sqrt{k}\delta/\sigma > L$, let $\tilde\psi$ instead be the indicator variable that takes value $1$ if at least $\tfrac{N}{2k} - \tfrac{N}{k}\cdot\alpha(1-\varrho)$ of $A_1,\dots, A_{N/k}$ occur.

    Upon squaring and repeating the algebra from \citet[Lemma II.5]{neykov2022minimax}, the expressions in the definition of $A_j$ satisfy
\begin{align}
        \MoveEqLeft \bigg\|k^{-1}\sum_{i \in G_j}(\tilde X_i + R_i) - \nu_1\bigg\|^2 - \bigg\| k^{-1}\sum_{i \in G_j}(\tilde X_i + R_i) - \nu_2\bigg\|^2 \notag \\
        &\le \|\nu_2-\nu_1\|\left[(-1+2/C)\|\nu_2-\nu_1\|+ \frac{2(\overline{ \tilde X_i} + \overline{R_i} - \mu)^T(\nu_2-\nu_1)}{\|\nu_2-\nu_1\|}\right]. \label{eq:subgaussian:decomp}
\end{align}
The probability that the first line above is bigger than $0$ (i.e., $A_j$ occurs) is smaller than the probability \eqref{eq:subgaussian:decomp} is bigger than 0, which is equal to
\begin{align}\label{symmetry:equation}
    \MoveEqLeft \PP_{\mu}\left(\sqrt{k} (\overline{ \tilde X_i} + \overline{R_i} - \mu)^Tv \ge \tfrac{\sqrt{k}\delta(C-2)}{2}\right) \\ &= \frac{1}{2} - \PP_{\mu}\left(0 \leq \sqrt{k} (\overline{ \tilde X_i} + \overline{ R_i} - \mu)^Tv \le \tfrac{\sqrt{k}\delta(C-2)}{2}\right), \notag
\end{align}
where we used $\|\nu_2-\nu_1\|\ge C\delta$ in the first claim and the symmetry of $\overline{ \tilde X_i} + \overline{R_i} - \mu$ in second one, and we set $v = (\nu_2-\nu_1)/\|\nu_2-\nu_1\|$ for brevity.

    \textsc{Case 1:} Suppose 
    $\sqrt{k}\delta/\sigma \le L$. 

    Let us show \begin{equation}
        \sup_{\mu: \|\mu - \nu_1\|\leq \delta}\PP_{\mu}\left(A_j\right) \leq \tfrac{1}{2} - \tfrac{C_2\delta}{\sigma}. \label{eq:subgaussian:symmetric:case1:typeI}
    \end{equation}  where  $C_2>0$ is the universal constant from the proof of Lemma \ref{lemma:technical:subgaussian}.  Well, by Lemma \ref{lemma:local:clt},  which is based on a local central limit theorem from  \citet[Chapter VII, Theorem 10]{petrov1976sums}, and the definition of $L, C_2$ in Lemma \ref{lemma:technical:subgaussian}, we conclude that since $\sqrt{k}\delta/\sigma < L,$ we have
     \[\PP\left(0 \leq \sqrt{k} (\overline{ \tilde X_i} + \overline{ R_i} - \mu)^Tv \le \sqrt{k}\delta\cdot \tfrac{C-2}{2}\right) \ge \tfrac{\delta(C-2)}{2\sigma D_2} = \tfrac{C_2\delta}{\sigma}.\] Combined with \eqref{symmetry:equation}, we obtain \eqref{eq:subgaussian:symmetric:case1:typeI}.

Let $\mu\in\RR^n$ be such that $\|\mu-\nu_1\|\le\delta$. Let us prove that $\PP_{\mu}(\psi=1)\le \PP_{\mu}(\tilde\psi=1)$ by showing $\tilde\psi=0$ implies $\psi=0$. Well, if $\tilde\psi=0$, then no more than $N/(2k) - (C_2\delta/\sigma)\cdot N/(2k)$ of $A_1,\dots, A_{N/k}$ occur. Observe that if none of the $\tilde X_i$ in a group $G_j$ was corrupted, then $A_j$ occurs if and only if $B_j$ does. The only way for $A_j$ to occur but not $B_j$ or vice-versa is if one of the $\tilde X_i$ in $G_j$ was corrupted. Since at most $N\epsilon$ of the $\tilde X_i$ was corrupted and this is less than the number of groups ($N\epsilon\le N\gamma/k < N/k$), this means at most $N\epsilon$ of the groups are corrupted (i.e., contain a corrupted datapoint). Thus, at most \begin{align*}
    N/(2k) - (C_2\delta/\sigma)\cdot N/(2k) + N\epsilon &<
    N/(2k) -  C_1C_2\epsilon N/(2k)+ N\epsilon \\ 
    &= N/(2k) +N\epsilon(1-C_1C_2/(2k)) \\
    &= N/(2k)
\end{align*} of $B_1,\dots, B_{N/k}$ occur, using $\delta/\sigma>C_1\epsilon$ in the first inequality and $C_1C_2=2k$ in the last line (from the proof of Lemma \ref{lemma:technical:subgaussian}). Thus, $\psi=0$, proving our claim that $\PP_{\mu}(\psi=1)\le \PP_{\mu}(\tilde\psi=1)$.

Now, $\tilde\psi=1$ means no more than $N/(2k) + (C_2\delta/\sigma)\cdot N/(2k)$ of $A_1^c,\dots, A_{N/k}^c$ occur. Set $\zeta = C_2\delta/(2\sigma)$ and $p = 1/2 + C_2\delta/\sigma$, so that $\PP_{\mu}(A_j^c)\ge p$ and $p-\zeta= 1/2 + C_2\delta/(2\sigma)$. Then Hoeffding's bound implies
\begin{align*}
   \PP_{\mu}(\tilde\psi=1)&\le \PP(\mathrm{Bin}(N/k,\PP_{\mu}(A_j^c)) \le  N/(2k) + (C_2\delta/\sigma)\cdot N/(2k))  \\ 
   &= \PP(\mathrm{Bin}(N/k,\PP_{\mu}(A_j^c)) \le (N/k)(p-\zeta))  \\ 
   &\le \PP(\mathrm{Bin}(N/k,\PP_{\mu}(A_j^c)) \leq (N/k) (\PP_{\mu}(A_j^c) - \zeta)) \\ &\leq \exp(-2N \zeta^2/k) \\
   &= \exp(-\tfrac{C_2^2 N\delta^2}{2k\sigma^2}).
\end{align*}
Taking the supremum over $\mu\in\RR^n$ such that $\|\mu-\nu_1\|\le \delta$ finishes the claim in this first case.

\textsc{Case 2:} Suppose 
 $\sqrt{k}\delta/\sigma > L$. 
 
 Well, $ \sqrt{k} (\overline{ \tilde X_i} + \overline{R_i} - \mu)^Tv$ is a mean $0$ sub-Gaussian random variable with parameter $\sqrt{2}\sigma$. Then using the upper deviation inequality in \citet[Section 2.1.2, page 23]{wainwright2019high}, we have \begin{align*}
     \PP\left(\sqrt{k} (\overline{ \tilde X_i} + \overline{R_i} - \mu)^Tv \ge \sqrt{k}\delta\cdot \tfrac{C-2}{2}\right) &\le \exp\left(-\frac{k\delta^2(C-2)^2}{16\sigma^2}\right).
 \end{align*}
 
 Now observe that given $x,t>0$, we have $\exp(-x)\le \tfrac{1}{2}\exp(-tx)$ if $t\le 1 - \tfrac{\log 2}{x}$. Set $x = \frac{k\delta^2(C-2)^2}{16\sigma^2}$ and set $t=\beta\in(0,1)$ from Lemma \ref{lemma:technical:subgaussian}. Then since $\sqrt{k}\delta/\sigma > L$, one can verify $\beta \le 1-\tfrac{\log 2}{x}$ will hold. Thus, recalling the definition of $A_j$, we have \[\PP_{\mu}(A_j)\le \exp\left(-\frac{k\delta^2(C-2)^2}{16\sigma^2}\right) \le \frac{1}{2}\exp\left(-\frac{\beta k\delta^2(C-2)^2}{16\sigma^2}\right)=\varrho/2,\] noting we set $\varrho = \exp\left(-\frac{\beta k \delta^2(C-2)^2}{16\sigma^2}\right)$.

 Again, fix $\mu$ such that $\|\mu-\nu_1\|\le \delta$ and we show $\PP_{\mu}(\psi=1)\le \PP_{\mu}(\tilde\psi=1)$. Suppose $\tilde \psi=0$. Then no more than $N/2k - N\alpha(1-\varrho)/k$ of $A_1,\dots, A_{N/k}$ occur. In the worst case, no more than $N\epsilon$ of the groups (noting that $N\epsilon <N\gamma/k<N/k$) contain a corrupted point, hence no more than \begin{align*}
     N/(2k) - N\alpha(1-\varrho)/k + N\epsilon &= N/(2k) + N(\epsilon -\alpha(1-\varrho)/k) \\ \le N/(2k)
 \end{align*} of $B_1,\dots, B_{N/k}$ occur, using $k\epsilon < \alpha(1-\varrho)$ from (ii) of Lemma \ref{lemma:technical:subgaussian}. Hence $\psi = 0$, proving  $\PP_{\mu}(\psi=1)\le \PP_{\mu}(\tilde\psi=1)$.

To bound $\PP_{\mu}(\tilde\psi=1)$, we closely follow the calculations from the Gaussian case but with $N/k$ instead. We set $p = 1-\varrho/2$, $\zeta = (1/2-\alpha)(1-\varrho)$ so that $p-\zeta = \tfrac{1}{2}+\alpha-\alpha\varrho$ and $\PP_{\mu}(A_j^c)\ge p$. Then \begin{align*}
      \PP_{\mu}(\tilde\psi=1) &\le \PP(\mathrm{Bin}(N/k, \PP_{\mu}(A_j^c)\le N/(2k) + (N/k)\alpha(1-\varrho)) \\
      &= \PP(\mathrm{Bin}(N/k, \PP_{\mu}(A_j^c)\le (N/k)(p-\zeta)) \\
      &\le \PP(\mathrm{Bin}(N/k, p)\le (N/k)(p-\zeta)) \\
      &\le \exp\left(-\tfrac{N}{k}\cdot D(p-\zeta\|p)\right),
 \end{align*} using the same stochastic dominance and Chernoff bound argument from the Gaussian case. Recall we also derived
 \begin{align*}
    D(p-\zeta \| p ) & \ge g(\alpha)+ (1/2-\alpha)\log(1/\varrho), 
\end{align*} where $g$ is defined in Lemma \ref{lemma:function:g:technical:details}. Thus, it suffices to require as before that $(1/2-\alpha)\log(1/\varrho)\ge -2g(\alpha)$, as we did in (iii) of Lemma \ref{lemma:technical:subgaussian}. This will imply  \begin{align*}
    D(p-\zeta \| p ) & \ge (1/2)(1/2-\alpha)\log(1/\varrho).
\end{align*} Then since $\log(1/\varrho) = \tfrac{\beta k\delta^2(C-2)^2}{16\sigma^2}$, we obtain  \begin{align*}
    \PP_{\mu}(\tilde\psi=1) &\le  \exp\left(-\tfrac{N}{k}\cdot (1/2)(1/2-\alpha) \cdot \tfrac{ \beta k\delta^2(C-2)^2}{16\sigma^2}\right)\\
    &= \exp\left(-\tfrac{(1/2-\alpha) \beta (C-2)^2}{32}\cdot \tfrac{ N\delta^2}{\sigma^2}\right).
\end{align*} 
    \end{proof}

\subsection{Proofs: General sub-Gaussian Noise}

First, we give a lemma stating the existence of some useful constants. Then we give two auxiliary results to establish our Type I error bound, where one handles the case where we used the trimmed mean estimator and the other our usual median-like estimator. Note that these two auxiliary results bound the first supremum term in the statement of Theorem \ref{theorem:asymmetric:testing:result}, but as explained at the end of the proof of Theorem \ref{theorem:main:testing:result} and at the start of the proof of Theorem \ref{theorem:subgaussian:main:testing:result}, we can apply a symmetry argument (with sub-Gaussian tail bounds instead) for the other supremum. This is enabled by our addition of the $R_i$ so that the tail bounds are the same with strict or weak inequality. 

\begin{lemma}\label{lemma:technical:constants:subgaussian:asymmetric:version2} There exists positive absolute constants $C, C_1, C_3, D_2,\alpha\in(0,1/2)$ with the following properties. Let  $\sigma>0$, and $\epsilon\in(0,1/32)$ be given. Let the function $g(t)$ be defined as in Lemma \ref{lemma:function:g:technical:details}. Define $D_1=4\sqrt{2(C_3+\log 4)}$, $D_3 = 8+3D_1^2/8$, $D_4 = D_2\sqrt{2 D_3}+D_1$, $D_6 = \sqrt{64 D_3\log 2}$. 
 Pick any $\delta >0$. Set $\beta=1-\frac{64 D_3\log 2}{(C-2)^2}$ and denote $\varrho = \exp\left(-\frac{\beta \delta^2(C-2)^2}{16\sigma^2}\right)$. Define the following properties:
\begin{enumerate}[(i)]
    \item $C-2>\max(D_1, D_4, D_6)$ \label{enum:asymmetric2:C}
    \item $\tfrac{D_2D_3}{C_1(C-2-D_1)}\in(0,1)$ \label{enum:asymmetric2:C_1_C_D}
    \item $\beta\in (0,1)$ \label{enum:asymmetric2:beta}
    \item $(1/2-\alpha)\log(1/\varrho)\ge -2g(\alpha)$ \label{enum:asymmetric2:alpha_and_varrho}
    \item $\epsilon < \alpha(1-\varrho)$ \label{enum:asymmetric2:epsilon_and_alpha_varrho}
\end{enumerate}  Then (i)-(iii) always hold, and if $\delta^2/\sigma^2\ge D_3^{-1}/4$, (iv) and (v) hold. 
\end{lemma}

    \begin{proof}[{Proof of Lemma \ref{lemma:technical:constants:subgaussian:asymmetric:version2}}] 
        We take $D_2$ to be an absolute constant that bounded $\cE$ from \citet[Theorem 1]{mendelson_robust_mean} using moment properties of sub-Gaussian random variables (see the remarks following their theorem). Pick $\alpha\in(0,1/2)$ such that 
        \begin{equation} \label{eq:asymmetric2:picking:alpha}
        \alpha\left(1-\exp\left(\tfrac{2g(\alpha)}{1/2-\alpha}\right)\right)> 1/32,
        \end{equation}  noting this is possible by \eqref{enum:h:surjection} of Lemma \ref{lemma:function:g:technical:details}. Pick $C_3>0$ arbitrarily. This lets us define $D_1, D_3,D_4,D_6$. Note that $D_3$ is the same quantity that appeared in Definition \ref{eq:psi:definition:asymmetric}.
        
        Observe that $\lim_{x\to 0^+}\frac{-2g(x)}{1/2-x} = \log 4$ and  $\lim_{x\to {1/2}^-}\frac{-2g(x)}{1/2-\alpha} = \infty$ (\eqref{enum:g:limit:0} and \eqref{enum:g:limit:1/2} of Lemma \ref{lemma:function:g:technical:details}).
        So given our choice of $\alpha$, pick $C$ large enough such that  \eqref{enum:asymmetric2:C} holds and \begin{equation} \label{eq:asymmetric2:c_vs_alpha}
            \tfrac{(C-2)^2-64 D_3\log 2}{64 D_3} \ge -\tfrac{2g(\alpha)}{1/2-\alpha}.
        \end{equation} After that, pick  $C_1>\tfrac{D_2D_3}{C-2-D_1}$, noting the denominator is positive by \eqref{enum:asymmetric2:C}.
        
        Then note that \eqref{enum:asymmetric2:C} implies the quantity appearing in \eqref{enum:asymmetric2:C_1_C_D} is positive, and our requirement $C_1>\tfrac{D_2D_3}{C-2-D_1}$ ensures it is $<1$, yielding the claim \eqref{enum:asymmetric2:C_1_C_D}. Moreover, \eqref{enum:asymmetric2:beta} follows from \eqref{enum:asymmetric2:C} since $C-2>D_6$ implies $\beta>0$ and clearly $\beta<1$.
        
        Now suppose $\delta^2/\sigma^2\ge D_3^{-1}/4$. Observe our assumption on $\delta^2/\sigma^2$ implies \begin{align*}
            \tfrac{\beta(C-2)^2\delta^2}{16\sigma^2} &= \left(1 - \tfrac{64 D_3\log 2}{(C-2)^2}\right)\tfrac{\delta^2}{\sigma^2}\cdot \tfrac{(C-2)^2}{16} = \left(\tfrac{(C-2)^2}{16}-\tfrac{64 D_3\log 2}{16}\right)\cdot \tfrac{\delta^2}{\sigma^2}\\ &\ge \left(\tfrac{(C-2)^2}{16}-\tfrac{64 D_3\log 2}{16}\right)\cdot \tfrac{1}{4D_3} = \tfrac{(C-2)^2-64D_3\log 2}{64 D_3}
        \end{align*} where the numerator is positive since $C-2>D_6$. Therefore, \begin{align} \label{eq:asymmetric:varrho:bound}
            1-\varrho &\ge 1-\exp\left(-\tfrac{(C-2)^2-64 D_3\log 2}{64 D_3}\right).
        \end{align} 

        Note that \eqref{enum:asymmetric2:alpha_and_varrho} is equivalent to $1-\varrho \ge 1-\exp\left(\tfrac{2g(\alpha)}{1/2-\alpha}\right)$. Due to \eqref{eq:asymmetric:varrho:bound}, it suffices to require \begin{align*}
            \exp\left(-\tfrac{(C-2)^2-64 D_3\log 2}{64 D_3}\right) \le \exp\left(\tfrac{2g(\alpha)}{1/2-\alpha}\right)
        \end{align*} which we have from \eqref{eq:asymmetric2:c_vs_alpha}. So \eqref{enum:asymmetric2:alpha_and_varrho} 
 holds, hence $\alpha(1-\varrho)\ge \alpha\left(1-\exp\left(\tfrac{2g(\alpha)}{1/2-\alpha}\right)\right)$. But recalling \eqref{eq:asymmetric2:picking:alpha} and $\epsilon<1/32$,  we obtain \eqref{enum:asymmetric2:epsilon_and_alpha_varrho}.
    \end{proof}

\begin{lemma} \label{lemma:trimmed:mean} There exists positive absolute constants $C>2,C_1>0, C_3>0$ with the following property. Let $\sigma>0$ and $\epsilon\in(0,1/32)$.  Suppose $\|\nu_1-\nu_2\|\ge C\delta$ for $\nu_1,\nu_2\in K$ and $\mu\in K$ is such that $\|\mu-\nu_1\|\le\delta$.  Suppose $\delta>0$ is such $\delta/\sigma\ge C_1\epsilon\sqrt{\log(1/\epsilon)}$ and additionally that $N^{-1}\le\delta^2/\sigma^2\le D_3^{-1}/4$ where $D_3 =8+12\log 4+12 C_3$. Then \begin{align*}
    \PP_{\mu}(\mathrm{TM}_{\delta_0}(\{V_i\}_{i=1}^{2N})>0) \le \exp\left(-\tfrac{C_3 N\delta^2}{\sigma^2}\right).
\end{align*}
\end{lemma}
    \begin{proof}
    Let $C$, $C_1$, $C_3$, and each of the $D_i$ be taken from Lemma \ref{lemma:technical:constants:subgaussian:asymmetric:version2}. Recalling \eqref{eq:V_i:decomposition}, we may write for an uncorrupted $V_i$ that \begin{align*}
    V_i &\le 2(\tilde X_i+ R_i-\mu)^Tv - (C-2)\delta
\end{align*} where $v = \tfrac{\nu_2-\nu_1}{\|\nu_2-\nu_1\|}$. Hence, using both \eqref{eq:V_i:decomposition:before:inequality} and \eqref{eq:V_i:decomposition}, an uncorrupted $V_i$ is sub-Gaussian with mean $m$ where $m\le -(C-2)\delta$. 

Using \citet[Theorem 1]{mendelson_robust_mean} (see also \citet[Proposition 1.18]{diakonikolas2023algorithmic} for a related result), we may bound the deviation between this uncorrupted mean $m$ and the trimmed mean estimator applied to the post-corruption data. The theorem states with probability $1-\delta_0=1-\exp\left(-\tfrac{C_3 N\delta^2}{\sigma^2}\right)$ that \begin{align} \label{eq:mendelson:theorem:1}
    |\mathrm{TM}_{\delta_0}(\{V_i\}_{i=1}^{2N}) - m| \le 3\cE(4\tilde\epsilon, V) + 2\sigma_X \sqrt{\tfrac{\log(4/\delta_0)}{N}},
\end{align} where $\cE$ is defined in Section 2 of their paper. As explained in the main text, the assumption $\delta^2/\sigma^2\le D_3^{-1}/4\le C_3^{-1}$ implies their $\delta_0\ge e^{-N}/4$ condition (in addition $\delta_0 \leq 1$ so that $\log(4/\delta_0)>0$). Moreover, $\tilde\epsilon\in(0,1/2)$ from our $\epsilon\le 1/32$ requirement as explained earlier, so $\tilde\epsilon$ is a valid quantile for use in the trimmed mean algorithm.

For the second term in \eqref{eq:mendelson:theorem:1}, expanding $\delta_0$ we have \begin{align*}
    2\sigma_X \sqrt{\tfrac{\log(4/\delta_0)}{N}} &= 2\sigma_X\sqrt{\tfrac{\log 4}{N} + \tfrac{C_3\delta^2}{\sigma^2}} \le 4\sqrt{2}\sigma\sqrt{\tfrac{\log 4}{N} + \tfrac{C_3\delta^2}{\sigma^2}} \\
    &= 4\sqrt{2}\sqrt{\tfrac{\sigma^2\log 4}{N} + C_3\delta^2} \le 4\sqrt{2}\sqrt{\delta^2\log 4 + C_3\delta^2} \\ &= \delta\cdot4\sqrt{2(C_3+\log 4)} = D_1\delta,
\end{align*} where we used $\sigma_X\le 2\sqrt{2}\sigma$ and $\delta^2/\sigma^2\ge 1/N$ and recall $D_1 = 4\sqrt{2(C_3+\log 4)}$.

For the other term in \eqref{eq:mendelson:theorem:1}, we use the authors' second remark after their Theorem 1 for the sub-Gaussian case with some convenient re-scaling. Namely, there exists an absolute constant $D_2>0$ (which we used in Lemma \ref{lemma:technical:constants:subgaussian:asymmetric:version2}) such that $3\cE(4\tilde\epsilon, V) \le D_2\sigma \tilde\epsilon\sqrt{\log(1/\tilde\epsilon)}$. This fact can be seen with a sub-Gaussian moment bound and H\"{o}lder's inequality. Also, $\tilde\epsilon>\epsilon$ so $\sqrt{\log(1/\tilde\epsilon)}\le \sqrt{\log(1/\epsilon)}$. Then expanding $\tilde\epsilon$ and using this fact along with $\delta^2/\sigma^2\ge 1/N$, \begin{align*}
    3\cE(4\tilde\epsilon, V) &\le D_2\sigma\tilde\epsilon\sqrt{\log(1/\tilde\epsilon)} \\ &\le D_2\sigma(8\epsilon + 12 \tfrac{\log(4/\delta_0)}{N})\cdot \sqrt{\log(1/\epsilon)} \\
    &=  D_2\sigma(8\epsilon + \tfrac{12\log 4}{N}+\tfrac{12C_3\delta^2}{\sigma^2})\cdot \sqrt{\log(1/\epsilon)} \\
    &\le D_2\sigma(8\epsilon + \tfrac{12\delta^2\log 4}{\sigma^2}+\tfrac{12C_3\delta^2}{\sigma^2})\cdot \sqrt{\log(1/\epsilon)} \\
    &=D_2\sigma(8\epsilon + \tfrac{3 D_1^2\delta^2}{8\sigma^2}) \cdot \sqrt{\log(1/\epsilon)}. 
\end{align*}

Consider two sub-cases. In the first assume $\epsilon > \delta^2/\sigma^2$. Then continuing from the previous line,  \begin{align*}
    3\cE(4\tilde\epsilon, V) &\le  D_2\sigma(8\epsilon + \tfrac{3\epsilon D_1^2}{8})\cdot \sqrt{\log(1/\epsilon)} = D_2D_3\sigma\epsilon\sqrt{\log(1/\epsilon)}, 
\end{align*} noting $D_3 =8+ 3D_1^2/8 = 8+12C_3+12\log 4$.

On the other hand, suppose $\epsilon \le \delta^2/\sigma^2$. Then using this along with $\delta^2/\sigma^2\ge N^{-1}$, we have \begin{align*}\tilde\epsilon &= 8\epsilon + \tfrac{12\log 4}{N}+\tfrac{12C_3\delta^2}{\sigma^2} \le \tfrac{8\delta^2}{\sigma^2}+ \tfrac{12\log 4\cdot \delta^2}{\sigma^2}+\tfrac{12 C_3\delta^2}{\sigma^2} = \tfrac{D_3 \delta^2}{\sigma^2}.
\end{align*} Now the map $x\mapsto x\sqrt{\log(1/x)}$ is real-valued, positive, upper-bounded by $1$ on $(0,1)$, and moreover, it is increasing on $(0,1/2]$. It also has limit $0$ as $x\to 0^+$. Observe that $D_3\delta^2/\sigma^2\le D_3\cdot(D_3^{-1}/4)=1/4$ by our assumption on $\delta/\sigma$ so we may apply monotonicity to the inputs $\tilde\epsilon\le D_3\delta^2/\sigma^2 \le 1/4$. This inequality also shows $D_3^{1/2}\delta/\sigma\in(0,1)$, so the upper bound $(D_3^{1/2}\delta/\sigma)\sqrt{\log(D_3^{-1/2}\sigma/\delta)}\le 1$ applies.  

Thus, using monotonicity on $(0,1/2]$, pulling out a factor of $2$ from the logarithm, and applying the upper bound on $(D_3^{1/2}\delta/\sigma)\sqrt{\log(D_3^{-1/2}\sigma/\delta)}$, we obtain \begin{align*}
3\cE(4\tilde\epsilon, V) &\le D_2\sigma\tilde\epsilon\sqrt{\log(1/\tilde\epsilon)} \\ &\le D_2\sigma\cdot(D_3\delta^2/\sigma^2)\cdot\sqrt{\log(D_3^{-1}\sigma^2/\delta^2)} \\
     &= D_2\sqrt{D_3}\delta\cdot(\sqrt{D_3}\delta/\sigma)\cdot\sqrt{2\log(D_3^{-1/2}\sigma/\delta)} \\
     &\le D_2\sqrt{2D_3}\delta,
\end{align*} noting all logarithm terms are positive. 
To summarize, if  \[\max(N^{-1}, C_1^2\epsilon^2\log(1/\epsilon)) \le  \delta^2/\sigma^2\le D_3^{-1}/4,\] then with probability at least $1-\exp\left(-\tfrac{C_3 N\delta^2}{\sigma^2}\right)$, we have  \begin{align}
    |\mathrm{TM}_{\delta_0}(V_1,\dots, V_{2N}) - m| \le \begin{cases}
        D_1\delta+D_2D_3\sigma \epsilon\sqrt{\log(1/\epsilon)}  & \text{ if }\epsilon > \delta^2/\sigma^2 \\
        D_4\delta& \text{ if }\epsilon \le \delta^2/\sigma^2 ,
    \end{cases}  \label{eq:TM:summary}
\end{align} where $D_1=4\sqrt{2(C_3+\log 4)}$, $D_2$ is some absolute constant from sub-Gaussian properties, $D_3=8+3D_1^2/8$, and $D_4=D_2\sqrt{2D_3}+D_1$. Recall from Lemma \ref{lemma:technical:constants:subgaussian:asymmetric:version2} that $C-2>\max(D_1,D_4)$.

Then, since $m\le -(C-2)\delta$ and $\delta^2/\sigma^2\le D_3^{-1}/4$ by assumption, \begin{align}
    \PP_{\mu}(\psi(\{X_i\}_{i=1}^{2N})=1) &= \PP_{\mu}(\mathrm{TM}_{\delta_0}(\{V_i\}_{i=1}^{2N})>0) \notag\\
    &= \PP_{\mu}(\mathrm{TM}_{\delta_0}(\{V_i\}_{i=1}^{2N})-m > -m) \notag\\
    &\le  \PP_{\mu}(\mathrm{TM}_{\delta_0}(\{V_i\}_{i=1}^{2N})-m > (C-2)\delta)\notag \\
    &\le \PP_{\mu}(|\mathrm{TM}_{\delta_0}(\{V_i\}_{i=1}^{2N})-m| > (C-2)\delta). \label{eq:TM_m_C_minus_2}
\end{align} For the $\epsilon\le \delta^2/\sigma^2$ case, since $C-2>D_4,$ we have  from \eqref{eq:TM:summary} that \begin{align*}
    \PP_{\mu}(|\mathrm{TM}_{\delta_0}(\{V_i\}_{i=1}^{2N})-m| > (C-2)\delta) &\le \PP_{\mu}(|\mathrm{TM}_{\delta_0}(\{V_i\}_{i=1}^{2N})-m| > D_4\delta ) \\
    &\le \exp\left(-\tfrac{C_3 N\delta^2}{\sigma^2}\right).
\end{align*} Consider the $\epsilon >\delta^2/\sigma^2$ case. Recall we assumed $\delta\ge C_1\sigma\epsilon\sqrt{\log(1/\epsilon)}$. Then the following can be seen by re-arranging: \begin{align}
    \delta - \tfrac{D_2D_3}{C-2-D_1}\sigma\epsilon \sqrt{\log(1/\epsilon)} \ge \left(1-\tfrac{D_2D_3}{C_1(C-2-D_1)}\right)\delta = D_5\delta \label{eq:D_5_bound}
\end{align} where we set  $D_5 = 1 -\frac{D_2D_3}{C_1(C-2-D_1)}.$ Note that $D_5\in(0,1)$ by (ii) of Lemma \ref{lemma:technical:constants:subgaussian:asymmetric:version2}. Thus, returning to \eqref{eq:TM_m_C_minus_2}, adding and subtracting some terms, using \eqref{eq:D_5_bound}, applying $C-2>D_1$ from (i) of Lemma \ref{lemma:technical:constants:subgaussian:asymmetric:version2}, and finally using \eqref{eq:TM:summary}, we have 
\begin{align*}
    \MoveEqLeft \PP_{\mu}(|\mathrm{TM}_{\delta_0}(\{V_i\}_{i=1}^{2N})-m| > (C-2)\delta)  \\ &= \PP_{\mu}(|\mathrm{TM}_{\delta_0}(\{V_i\}_{i=1}^{2N})-m| > D_1\delta  + D_2D_3\sigma\epsilon\sqrt{\log(1/\epsilon)} \\ &\quad\quad +(C-2-D_1)\delta- D_2D_3\sigma\epsilon\sqrt{\log(1/\epsilon)}) \\
    &=  \PP_{\mu}(|\mathrm{TM}_{\delta_0}(\{V_i\}_{i=1}^{2N})-m| > D_1\delta + D_2D_3\sigma\epsilon\sqrt{\log(1/\epsilon)} \\ &\quad\quad +(C-2-D_1)(\delta- \tfrac{D_2D_3}{C-2-D_1}\sigma\epsilon\sqrt{\log(1/\epsilon)})) \\
    &\le \PP_{\mu}( |\mathrm{TM}_{\delta_0}(\{V_i\}_{i=1}^{2N})-m| >D_1\delta + D_2D_3\sigma\epsilon\sqrt{\log(1/\epsilon)}  \\ &\quad\quad+(C-2-D_1) D_5\delta) \\
    &\le \PP_{\mu}( |\mathrm{TM}_{\delta_0}(\{V_i\}_{i=1}^{2N})-m| >D_1\delta + D_2D_3\sigma\epsilon\sqrt{\log(1/\epsilon)} )\\
    &\le \exp\left(-\tfrac{C_3 N\delta^2}{\sigma^2}\right).
\end{align*}
    \end{proof}

\begin{lemma} \label{lemma:subgaussian:asymmetric:varrho:case} There exist sufficiently large absolute constants $C>2,C_1>0,C_3>0, C_4>0$ and $\alpha\in(0,1/2)$ with the following properties. Let $\sigma>0$ and $\epsilon\in(0,1/32)$. Suppose $\|\nu_1-\nu_2\|\ge C\delta$ for $\nu_1,\nu_2\in K$ and $\mu\in K$ is such that $\|\mu-\nu_1\|\le\delta$. Suppose $\|\nu_1-\nu_2\|\ge C\delta$ for $\nu_1,\nu_2\in K$. Suppose $\delta>0$ is such $\delta/\sigma\ge C_1\epsilon\sqrt{\log(1/\epsilon)}$ and additionally that  $\delta^2/\sigma^2\ge D_3^{-1}/4$ where $D_3 =8+12\log 4+12 C_3$. Then \begin{align*}
    \PP_{\mu}(|\{i \in [2N]: \|X_i - \nu_1\| \geq \|X_i - \nu_2\|\}| \geq N) \le \exp\left(-\tfrac{C_4 N\delta^2}{\sigma^2}\right).
\end{align*}
\end{lemma}
    \begin{proof}
        Let $C$, $C_1$, $C_3$, $\alpha$, and the $D_i$ be given from Lemma \ref{lemma:technical:constants:subgaussian:asymmetric:version2}. We borrow arguments from the Gaussian case. We set $A_i$ to be the event $\|\tilde X_i-\nu_1\|\ge \|\tilde X_i-\nu_2\|$ and $B_i$ the event $\| X_i-\nu_1\|\ge \| X_i-\nu_2\|$, where recall $\tilde X_i$ is the uncorrupted data and $X_i$ is the corrupted data (and we added a Gaussian noise term $R_i\sim\cN(0,\sigma^2\II)$ in either case). Recalling the computations in \eqref{eq:V_i:decomposition}, we have \begin{align*}
            \MoveEqLeft \|\tilde X_i+R_i-\nu_1\|^2 - \|\tilde X_i+R_i-\nu_2\|^2 \\ &\leq 2(\tilde{X}_i+R_i-\mu)^T(\nu_2-\nu_1) + (-1+2/C)\|\nu_2-\nu_1\|^2,
        \end{align*} noting this is a sub-Gaussian random variable with mean $m = (-1+2/C)\|\nu_2-\nu_1\|^2 \le (2-C)\delta <0$ and parameter $\sqrt{8}\sigma\|\nu_2-\nu_1\|$  (noting the addition of a Gaussian noise term multiples the parameter by $\sqrt{2}$). Using $\|\nu_2-\nu_1\|\ge C\delta$ and the upper deviation inequality from \citet[Section 2.1.2]{wainwright2019high} for sub-Gaussian random variables, \begin{align*}
            \MoveEqLeft\PP_{\mu}\left(\|\tilde X_i+R_i-\nu_1\|^2 - \|\tilde X_i+R_i-\nu_2\|^2  >0\right) \\ &\leq \PP_{\mu}\left(2(\tilde{X}_i+R_i-\mu)^T(\nu_2-\nu_1)  > (1-2/C)\|\nu_2-\nu_1\|^2\right) \\
            &\le \PP_{\mu}\left(2(\tilde{X}_i+R_i-\mu)^T(\nu_2-\nu_1)  > (C-2)\delta\|\nu_2-\nu_1\|\right) \\
            &\le \exp\left(-\tfrac{(C-2)^2\delta^2\|\nu_2-\nu_1\|^2}{2\cdot 8\sigma^2\|\nu_2-\nu_1\|^2}\right) \\
            &= \exp\left(-\tfrac{(C-2)^2\delta^2}{16\sigma^2}\right). 
        \end{align*}

Recall once more that for any $x,t>0$, we have $\exp(-x)\le \tfrac{1}{2}\exp(-tx)$ provided $t\le 1 - \tfrac{\log 2}{x}$. Well setting $x = \tfrac{(C-2)^2\delta^2}{16\sigma^2}$ and $t = 1- \tfrac{64 D_3\log 2}{(C-2)^2}$ (which is positive so long as $C> 2+\sqrt{64 D_3\log 2}$, guaranteed from \eqref{enum:asymmetric2:C} of Lemma \ref{lemma:technical:constants:subgaussian:asymmetric:version2}),  one can verify $t\le 1-\tfrac{\log 2}{x}$ indeed holds since $\delta^2/\sigma^2 \ge D_3^{-1}/4$. Thus, \begin{align*}
            \PP_{\mu}\left(\|\tilde X_i+R_i-\nu_1\|^2 - \|\tilde X_i+R_i-\nu_2\|^2  >0\right) &\le \tfrac{1}{2}\exp\left(-\tfrac{\beta(C-2)^2\delta^2}{16\sigma^2}\right) = \varrho/2
        \end{align*} where we set $\beta = 1- \tfrac{64D_3\log 2}{(C-2)^2}$ and $\varrho=\exp\left(-\tfrac{\beta(C-2)^2\delta^2}{16\sigma^2}\right)$. We will have $\beta\in(0,1)$ from Lemma \ref{lemma:technical:constants:subgaussian:asymmetric:version2}.

        From here, the proof is completely identical to the second case of the Gaussian setting's testing result in Theorem \ref{theorem:main:testing:result} once we had established $\PP_{\mu}(A_i)\le \varrho/2$, just with a different $\varrho$ and we use $\tilde X_i+R_i$ and $X_i+R_i$ instead of just $\tilde X_i$ and $X_i$. We still require the same assumptions that $\epsilon<\alpha(1-\varrho)$ and that $(1/2-\alpha)\log(1/\varrho)\ge -2g(\alpha)$ where $g$ was defined in Lemma \ref{lemma:function:g:technical:details} and $\alpha\in(0,1/2)$ is some absolute constant. But these will hold by Lemma \ref{lemma:technical:constants:subgaussian:asymmetric:version2}. Following the last few computations in the proof of Theorem \ref{theorem:main:testing:result}, our ultimate bound on the Type I error will be \[\exp\left(-N\cdot \tfrac{1/2-\alpha}{2}\cdot\log(1/\varrho)\right) = \exp\left(-\tfrac{N(1/2-\alpha)\beta(C-2)^2\delta^2}{32\sigma^2}\right)=\exp(-\tfrac{C_4N\delta^2}{\sigma^2}),\] where $\alpha\in(0,1/2)$ is some absolute constant, $\beta = 1- \tfrac{64D_3\log 2}{(C-2)^2}\in(0,1)$, and we set $C_4 =\frac{(1/2-\alpha)\beta(C-2)^2}{32}.$ 
    \end{proof}

We split the proof of Theorem \ref{theorem:general:subgaussian:version} into lemmas as in the Gaussian case. We start with the analogues of Lemma \ref{lemma:for:theorem:gaussian} and Lemma \ref{lemma:for:theorem:gaussian:part2} with Lemma \ref{lemma:for:theorem:asymmetric} and Lemma \ref{lemma:for:theorem:asymmetric:part2}, respectively. We omit the proofs of both lemmas as they are near replicas of the Gaussian counterpart. The main distinction is that we swap out the hypothesis on $\delta/\sigma$. That is, we assume $\delta/\sigma \gtrsim \epsilon\sqrt{\log(1/\epsilon)}$ rather than $\delta/\sigma\gtrsim\epsilon$, and we additionally require $\tfrac{\sigma}{\sqrt{N}} \le \delta$, where in this context $\delta = \tfrac{d}{2^{J-1}(C+1)}$. The last distinction is that our bound is on $\EE_R\EE_X\|\mu-\nu^{\ast\ast}\|^2$ as explained in Remark \ref{remark:expectation:over:R}. Otherwise there are no changes to the underlying logic.  

\begin{lemma} \label{lemma:for:theorem:asymmetric} Let $\eta_J$ be defined as in Theorem \ref{theorem:general:subgaussian:version} and let $C_1, C_3, C_5$ be given from Theorem \ref{theorem:asymmetric:testing:result}. Suppose $\tilde J$ is such that \eqref{eq:asymmetric:robust:theorem:condition} holds, that $\frac{d}{2^{\tilde J-1}(C+1)} \ge (C_1\sigma\epsilon\sqrt{\log(1/\epsilon)}) \vee \tfrac{\sigma}{\sqrt{N}}$. Then for each  $1\le J\le\tilde J$ we have \[ \PP\left(\|\Upsilon_J - \mu\| \ge \tfrac{d}{2^{J-1}}\right)
             \le 2\cdot\mathbbm{1}(J>1)\exp(-\tfrac{N\eta_J^2}{2\sigma^2}).\]
\end{lemma}

\begin{lemma} \label{lemma:for:theorem:asymmetric:part2} 
 Let $\nu^{\ast\ast}$ denote the output after at least $J^{\ast}$ iterations. Set $C_2 = \tfrac{(19+16C)^2}{4C_5}$. Then under the same assumptions as in Lemma \ref{lemma:for:theorem:asymmetric}, we have \[\EE_R\EE_X\|\mu-\nu^{\ast\ast}\|^2\le C_2\eta_{\tilde{J}}^2 + \mathbbm{1}(J^{\ast}>1)\cdot 4C_2\cdot  \tfrac{\sigma^2}{N}\exp\left(-\tfrac{N\eta_{\tilde{J}}^2}{2\sigma^2}\right).\]
\end{lemma}

We may now proceed with proof of Theorem \ref{theorem:general:subgaussian:version} by handling the scenarios when these assumptions on $\delta/\sigma$ in the previous two lemmas fail. In all cases, we end up dealing with a bound of the form in Lemma \ref{lemma:for:theorem:asymmetric:part2} but at indices possibly prior to $J^{\ast}$, like in the Gaussian case. The analogous non-corrupted setting of \citet{neykov2022minimax} did not require this case-by-case analysis, by contrast.

    \begin{proof}[\hypertarget{proof:theorem:general:subgaussian:version}{Proof of Theorem \ref{theorem:general:subgaussian:version}}] 
    We re-use the same notation from the proof of Theorem \ref{theorem:gaussian:version}. If $J^{\ast}=1$, then the claimed upper bound reduces to $d^2$ and trivially holds. Suppose $J^{\ast}>1$.

    \textsc{Case 1:} Suppose  we have that $\tfrac{d}{2^{J^{\ast}-1}(C+1)} \ge C_1\sigma\epsilon\sqrt{\log(1/\epsilon)}$ and moreover that $\tfrac{\sigma}{\sqrt{N}}\le \tfrac{d}{2^{J^{\ast}-1}(C+1)}$. Here $C_1$ is chosen from Theorem \ref{theorem:asymmetric:testing:result}. Then by Lemma \ref{lemma:for:theorem:asymmetric:part2}, we will obtain  \[\EE_R\EE_X\|\mu-\nu^{\ast\ast}\|^2 \le C_2\eta_{J^{\ast}}^2 + \mathbbm{1}(J^{\ast}>1)\cdot 4C_2\cdot  \tfrac{\sigma^2}{N}\exp\left(-\tfrac{N\eta_{J^{\ast}}^2}{2\sigma^2}\right)\] where $\nu^{\ast\ast}$ is the output after at least $J^{\ast}$ steps. Note the use of $\EE_R$ in Lemma \ref{lemma:for:theorem:asymmetric:part2}, in line with the discussion from Remark \ref{remark:expectation:over:R}.  Since $N\eta_{J^{\ast}}^2/\sigma^2>\log 2$ from the definition of $J^{\ast}$, we have $\tfrac{\sigma^2}{N}<\frac{\eta_{J^{\ast}}^2}{\log 2}$ and the exponential is of constant order, so the bound on $\EE_X\EE_{R}\|\mu-\nu^{\ast\ast}\|^2$ is of order $\eta_{J^{\ast}}^2$, meaning it is bounded by $\max(\eta_{J^{\ast}}^2,\sigma^2\epsilon^2\log(1/\epsilon))$. The bound from $d^2$ is immediate, completing this case.

    Let us now handle the remaining scenarios that would prevent the above argument from working. There are one of two assumptions that could go wrong. We could either obtain a $J\in\{1,2,\dots, J^{\ast}\}$ such that $\tfrac{d}{2^{J-1}(C+1)}<C_1\sigma\epsilon\sqrt{\log(1/\epsilon)}$ or such that $\tfrac{d}{2^{J-1}(C+1)}<\tfrac{\sigma}{\sqrt{N}}$.

   \textsc{Case 2:} Suppose $\tfrac{d}{2^{J'-1}(C+1)}<\tfrac{\sigma}{\sqrt{N}}$ at some $1\le J'\le J^{\ast}$ (chosen minimally). 
    
   \textsc{Case 2(a)}: If $J'=1$, this means $d\lesssim \tfrac{\sigma}{\sqrt{N}}$. But this implies $\EE_R\EE_X\|\mu - \nu^{\ast\ast}\|^2 \lesssim d^2\asymp \tfrac{\sigma^2}{N}\wedge d^2$. Now we must have $\eta_{J^{\ast}} \ge \sigma\sqrt{\log 2}/\sqrt{N}$ by \eqref{eq:asymmetric:robust:theorem:condition} since $J^{\ast}>1$. Thus \begin{equation} \label{eq:case2a:edge:max_term_bound}
\tfrac{\sigma^2}{N}\wedge d^2\lesssim \eta_{J^{\ast}}^2 \wedge d^2 \le \max(\eta_{J^{\ast}}^2,\sigma^2\epsilon^2\log(1/\epsilon))\wedge d^2. 
   \end{equation}Indeed $\EE_R\EE_X\|\mu - \nu^{\ast\ast}\|^2 \lesssim \max(\eta_{J^{\ast}}^2,\sigma^2\epsilon^2\log(1/\epsilon))\wedge d^2$ as claimed.

    \textsc{Case 2(b)}: Suppose $J'\ge 2$. By minimality, this means for $1\le J\le J'-1$, we have $\tfrac{d}{2^{J-1}(C+1)} \ge \tfrac{\sigma}{\sqrt{N}}$. Now consider two sub-cases. 
    
    \textsc{Case 2(b)(i):} In the first, say for some $J''\le J'-1$ chosen minimally, we have $\tfrac{d}{2^{J''-1}(C+1)} < C_1\sigma\epsilon\sqrt{\log(1/\epsilon)}$. If $J''=1$, that means $d\lesssim \sigma\epsilon\sqrt{\log(1/\epsilon)}$ which means \[\EE_R\EE_X\|\mu - \nu^{\ast\ast}\|^2\lesssim d^2\lesssim\max(\eta_{J^{\ast}}^2,\sigma^2\epsilon^2\log(1/\epsilon))\wedge d^2.\] We proceed to the case $J''>1$. That means for $1\le J\le J''-1$, noting $J''\le J'-1$, we have both  $\tfrac{d}{2^{J-1}(C+1)} \ge \tfrac{\sigma}{\sqrt{N}}$ and $\tfrac{d}{2^{J-1}(C+1)} \ge C_1\sigma\epsilon\sqrt{\log(1/\epsilon)}$. We may therefore apply Lemma \ref{lemma:for:theorem:asymmetric:part2} with $\tilde{J}=J''-1$ to obtain  \begin{equation} \label{eq:asymmetric:case2a}
        \EE_R\EE_X\|\mu-\nu^{\ast\ast}\|^2 \le C_2\eta_{J''-1}^2 + \mathbbm{1}(J^{\ast}>1)\cdot 4C_2\cdot  \tfrac{\sigma^2}{N}\exp\left(-\tfrac{N\eta_{J''-1}^2}{2\sigma^2}\right).
    \end{equation}
    
    Consider two sub-cases. First, say $\epsilon\sqrt{\log(1/\epsilon)}\le \tfrac{C_6}{\sqrt{N}}$  where $0<C_6 \le \tfrac{\sqrt{\log 2}}{C_1\sqrt{C_5}}$. Then by definition of $J''$ we have \[\tfrac{d}{2^{J''-1}} < (C+1)C_1\sigma\epsilon\sqrt{\log(1/\epsilon)} \le (C+1)C_1C_6 \tfrac{\sigma}{\sqrt{N}}.\] Rearranging, \[\tfrac{N\eta_{J''}^2}{\sigma^2} \le [C_1\sqrt{C_5}C_6]^2 \le \log 2.\] Thus, \eqref{eq:asymmetric:robust:theorem:condition} does not hold, which means $J''\ge J^{\ast}$. But we assumed $J''\le J'-1<J^{\ast}$, so this is a contradiction. 

    Suppose instead $\epsilon\sqrt{\log(1/\epsilon)}\ge \tfrac{C_6}{\sqrt{N}}$. This implies $\tfrac{\sigma^2}{N} \lesssim \sigma^2\epsilon^2\log(1/\epsilon)$ and the exponential is always $\le 1$. So the entire second term of \eqref{eq:asymmetric:case2a} is $\lesssim \sigma^2\epsilon^2\log(1/\epsilon)$. Moreover, $\eta_{J''-1}=2\eta_{J''} \asymp \tfrac{d}{2^{J''-1}} \lesssim \sigma\epsilon\sqrt{\log(1/\epsilon)}$ by definition of $J''$. So \eqref{eq:asymmetric:case2a} is bounded by $\sigma^2\epsilon^2\log(1/\epsilon)$ which is clearly less than $\max(\eta_{J^{\ast}}^2, \sigma^2\epsilon^2\log(1/\epsilon))$. The upper bound of $d^2$ is trivial, so we obtain the desired bound on $\EE_R\EE_X\|\mu-\nu^{\ast\ast}\|^2$.

    \textsc{Case 2(b)(ii)}: In the other sub-case, we have  $\tfrac{d}{2^{J'-1}(C+1)}<\tfrac{\sigma}{\sqrt{N}}$, while for $1\le J\le J'-1$ both $\tfrac{d}{2^{J-1}(C+1)} \ge \tfrac{\sigma}{\sqrt{N}}$ and $\tfrac{d}{2^{J-1}(C+1)} \ge C_1\sigma\epsilon\sqrt{\log(1/\epsilon)}$. In other words, only the $\tfrac{\sigma}{\sqrt{N}}$ condition fails, while in Case 2(b)(i) both this and the $\sigma\epsilon\sqrt{\log(1/\epsilon)}$ condition failed at some point. 

    We may therefore apply Lemma \ref{lemma:for:theorem:asymmetric:part2}  with $\tilde{J}=J'-1$ to conclude \begin{equation} \label{eq:asymmetric:case2b}
        \EE_R\EE_X\|\mu-\nu^{\ast\ast}\|^2 \le C_2\eta_{J'-1}^2 + \mathbbm{1}(J^{\ast}>1)\cdot 4C_2\cdot  \tfrac{\sigma^2}{N}\exp\left(-\tfrac{N\eta_{J'-1}^2}{2\sigma^2}\right).
    \end{equation}
    Then observe that $\eta_{J'-1}=2\eta_{J'} < 2\sqrt{C_5}\tfrac{\sigma}{\sqrt{N}}$ by assumption of this Case 2(b)(ii). The bound in \eqref{eq:asymmetric:case2b} becomes $\lesssim \tfrac{\sigma^2}{N} + \tfrac{\sigma^2}{N}\exp\left(-\tfrac{N\eta_{J'-1}^2}{2\sigma^2}\right)$. The exponential term is clearly $\le 1$. Thus, \eqref{eq:asymmetric:case2b} is  $\lesssim \tfrac{\sigma^2}{N}$. Hence $ \EE_R\EE_X\|\mu-\nu^{\ast\ast}\|^2 \lesssim \tfrac{\sigma^2}{N} \wedge d^2$. But recall our argument that  $J^{\ast}>1$ and \eqref{eq:asymmetric:robust:theorem:condition} implies the bound \eqref{eq:case2a:edge:max_term_bound}. Hence \[\EE_R\EE_X\|\mu-\nu^{\ast\ast}\|^2 \lesssim \tfrac{\sigma^2}{N} \wedge d^2 \lesssim \max(\eta_{J^{\ast}}^2, \sigma^2\epsilon^2\log(1/\epsilon))\wedge d^2\] as required.

    \textsc{Case 3}: The final consideration is if $\tfrac{d}{2^{J'-1}(C+1)}<\tfrac{\sigma}{\sqrt{N}}$ never occurs for $1\le J'\le J^{\ast}$, but we do encounter $\tfrac{d}{2^{J'-1}(C+1)} < C_1\sigma\epsilon\sqrt{\log(1/\epsilon)}$ at some $1\le J'\le J^{\ast}$. Then we repeat the argument of Case 2(b)(i) with $J'$ in place of $J''$. To briefly summarize the argument, if $J'=1$ we bound $d\lesssim \sigma\epsilon\sqrt{\log(1/\epsilon)}$ and the claim immediately follows. If not, apply Lemma \ref{lemma:for:theorem:asymmetric:part2} with $\tilde{J}=J'-1$. Split into two subcases on whether $\epsilon\sqrt{\log(1/\epsilon)} \le \tfrac{C_6}{\sqrt{N}}$ or not. In the first sub-case, we end up concluding $J^{\ast}=1$ (theorem trivially holds) and in the second sub-case, we bound the expectation with $\sigma^2\epsilon^2\log(1/\epsilon)$ and the rest of the bound easily follows.    
    \end{proof}

    \begin{proof}[\hypertarget{proof:theorem:robust:minimax:rate:attained:subgaussian}{Proof of Theorem \ref{theorem:robust:minimax:rate:attained:subgaussian}}] 
    As we argued in the Gaussian case, the $\eta^{\ast}=0$ case implies $d=0$, and the minimax rate is trivially achieved. We henceforth assume $\eta^{\ast}>0$ and thus that $d>0$, so that $ \log \cMloc(\eta,c)$ can be made arbitrarily large by taking $c$ arbitrarily large.
    
    \textsc{Case 1:} $N{\eta^{\ast}}^2/\sigma^2> 8\log 2$.
    
    We start with the lower bound. Recall Lemma \ref{lemma:lower:bound:first:version} still applies in the sub-Gaussian setting. Repeating the argument from Case 1 of Theorem \ref{theorem:robust:minimax:rate:attained:gaussian}, one can verify \[\log \cMloc(\eta^{\ast}/4,c) \ge 4\left(\tfrac{N(\eta^{\ast}/4)^2}{2\sigma^2}\vee \log 2\right),\] so again applying Lemma \ref{lemma:lower:bound:first:version}, we conclude ${\eta^{\ast}}^2$ and therefore ${\eta^{\ast}}^2\wedge d^2$ is a lower bound up to constants. But also by Lemma \ref{lemma:subgaussian:lower:bound}, $\sigma^2\epsilon^2\log(1/\epsilon)\wedge d^2$ is a lower bound. Hence $\max({\eta^{\ast}}^2 \wedge d^2, \sigma^2\epsilon^2\log(1/\epsilon)\wedge d^2)$ is a lower bound as claimed. 

    For the upper bound, we again repeat the argument from Case 1 of Theorem \ref{theorem:robust:minimax:rate:attained:gaussian}. That is, observe from Theorem \ref{theorem:general:subgaussian:version} that $\max(\eta_{J^{\ast}}^2\wedge d^2,\sigma^2\epsilon^2\log(1/\epsilon)\wedge d^2)$ is an upper bound, so it suffices to obtain $\tilde\eta\asymp\eta^{\ast}$ such that ${\tilde\eta}^2\ge \eta_{J^{\ast}}^2$. So set $\beta' = \min(\tfrac{1}{\sqrt{2}}, \tfrac{c}{2}\sqrt{\tfrac{2}{C_5}})\in(0,1/\sqrt{2}]$, pick $D>1$ such that $D\beta'>1$, and take $\tilde\eta = D\sqrt{2}\eta^{\ast}$. One can verify as before that $\tilde\eta$ satisfies \eqref{eq:asymmetric:robust:theorem:condition} (with $c$, not $2c$ by the logic in Remark \ref{remark:changing:2c:to:c}) using our  $N{\eta^{\ast}}^2/\sigma^2> 8\log 2$ assumption. Then define a similar $\phi$ (with  $C_5$ instead of $C_3(\kappa)$) and deduce ${\tilde\eta}^2\ge \eta_{J^{\ast}}^2$. Thus, $\max({\eta^{\ast}}^2 \wedge d^2, \sigma^2\epsilon^2\log(1/\epsilon)\wedge d^2)$ is both a lower and upper bound up to constants.

    \textsc{Case 2:}  $N{\eta^{\ast}}^2/\sigma^2\le 8\log 2$. 
    
    Repeat the argument of Case 3 of Theorem \ref{theorem:robust:minimax:rate:attained:gaussian}. Namely, argue that a line segment of length $4\eta^{\ast}$ cannot fit inside $K$, hence $d/3\le 4\eta^{\ast}\le 8\sigma\sqrt{2(\log 2)/N}$. Then $\tilde\eta = d/6$ satisfies $2N{\tilde\eta}^2/\sigma^2 \le 64\log 2$ while for sufficiently large $c$ we can obtain $\log \cMloc(\tilde\eta, c)> 2N{\tilde\eta}^2/\sigma^2  \vee 4\log 2$. Thus, by Lemma \ref{lemma:lower:bound:first:version}, for sufficiently large $c$ the minimax rate is lower bounded by $d^2$ and hence ${\eta^{\ast}}^2\wedge d^2$. Combined with Lemma \ref{lemma:subgaussian:lower:bound}, we conclude $\max({\eta^{\ast}}^2\wedge d^2, \sigma^2\epsilon^2\log(1/\epsilon)\wedge d^2)$ is a lower bound. For the upper bound, note that $d^2$ is trivially an upper bound, and we argued earlier that $d^2 \le 12^2{\eta^{\ast}}^2$. Hence ${\eta^{\ast}}^2 \wedge d^2\le \max({\eta^{\ast}}^2\wedge d^2, \sigma^2\epsilon^2\log(1/\epsilon)\wedge d^2)$ is an upper bound, establishing the minimax rate.
    \end{proof}

\section{Proofs for Section \ref{section:robust_gsm:unbounded}}

    \begin{proof}[\hypertarget{proof:lemma:R:tail:bound}{Proof of Lemma \ref{lemma:R:tail:bound}}] 
    For brevity set $Y_i=\tilde X_i-\mu$, which is a Gaussian with mean 0 and variance $\sigma^2$ (or sub-Gaussian with mean 0 and parameter bounded by $\sigma$). Pick any $t\in(0,1)$. Construct a maximal $t$-packing of the unit sphere $S^{n-1}$ in $n$ dimensions which has cardinality bounded by $(1+\tfrac{2}{t})^n$ by \citet[Corollary 4.2.13]{vershynin2018high}. Now pick any unit vector $v$ in this packing. Then we have \begin{align*}
            \PP( |v^TY_i| >R) \le \exp(-\tfrac{R^2}{2\sigma^2}).
        \end{align*} Apply a union bound over the packing set, we have \begin{align*}
            \PP(\sup_v |v^TY_i|> R) \le (1+\tfrac{2}{t})^n\exp(-\tfrac{R^2}{2\sigma^2}).
        \end{align*} Now since we have a packing of $S^{n-1}$, for some $u\in S^{n-1}$ we have $\|u- Y_i/\|Y_i\|\|\le t$. Therefore \begin{align*}
            |u^T Y_i - \|Y_i\|| &= |u^TY_i - \|Y_i\|\cdot u^T u| = \|Y_i\|\cdot \big|u^T\big(\tfrac{Y_i}{\|Y_i\|} - u\big)\big| \\
            &\le \|Y_i\|\cdot \|u\|\cdot \underbrace{\big\|\tfrac{Y_i}{\|Y_i\|} - u\big\|}_{\le t} \le t \|Y_i\|.
        \end{align*} This means $(1-t)\|Y_i\|\le u^T Y_i$ for some $u$ in our packing set. So \begin{align*}
            \PP((1-t)\|Y_i\|> R) &\le \PP(\exists u: u^T Y_i > R) \le  \PP(\sup_v |v^TY_i|> R) \\ &\le (1+\tfrac{2}{t})^n\exp(-\tfrac{R^2}{2\sigma^2}).
        \end{align*} Take $t=1/2$, then we have $\PP(\|Y_i\|> 2R) \le 5^n\exp(-\tfrac{R^2}{2\sigma^2})$. Performing a change of variables by replacing $R$ with $R/2$, the bound becomes $\PP(\|Y_i\| > R)\le 5^n\exp(-\tfrac{R^2}{8\sigma^2})$.
    \end{proof}

    \begin{proof}[\hypertarget{proof:lemma:unbounded:mu:S}{Proof of Lemma \ref{lemma:unbounded:mu:S}}] 
       Recall $\tilde E_{\mu, i}$ is the event that $\|\tilde X_i-\mu\| > R$.  We know $\PP(\tilde E_{\mu, i})\le \exp(-\tfrac{\gamma R^2}{8\sigma^2})$ from \eqref{eq:unbounded:tilde:X_i:bound:post:R}. Observe that if $x>2\log 2$ and we take $\beta = 1-\tfrac{\log 2}{x} \in(1/2,1)$, we have $\exp(-x)\le \tfrac{1}{2}\exp(-\beta x)$. Thus, take $x=\tfrac{\gamma R^2}{8\sigma^2}$ and $\beta = 1- \tfrac{8\sigma^2\log 2}{\gamma R^2}$, noting that $\tfrac{\gamma R^2}{8\sigma^2}>2\log 2$ follows from the third bound in \eqref{eq:R:conditions}. Then $\beta\in(1/2,1)$ and $\PP(\tilde E_{\mu, i})\le \varrho/2$ where $\varrho = \exp(-\tfrac{\beta \gamma R^2}{8\sigma^2})$. 

      Let us now bound $\PP(\mu\not\in S)$. Well, if $\mu\not\in S$, then at least $N/2$ of the events $E_{\mu,1},\dots, E_{\mu,N}$ occur, i.e., $\|X_i-\mu\|> R$. That means at least $N/2-\epsilon N$ of $\tilde E_{\mu,i}$ (i.e., $\|\tilde X_i-\mu\|> R$) occur. To see this: if strictly fewer than $N/2 - \epsilon N$ of the $\tilde X_i$ satisfy $\|\tilde X_i-\mu\|>  R$, and we can corrupt at most fraction $\epsilon$ of the $\tilde X_i$, then strictly fewer than $N/2-\epsilon N + \epsilon N=N/2$ of the $X_i$ can satisfy $\|X_i-\mu\|> R$. This contradicts our assumption that at least $N/2$ of the $E_{\mu,i}$ occur. To summarize: \[\PP(\mu\not\in S)\le \PP(\mathrm{Bin}(N, \PP(E_{\mu, i})) \ge N/2) \le \PP(\mathrm{Bin}(N, \PP(\tilde E_{\mu, i})) \ge N/2-\epsilon N). \]

       Next, note that the $\tilde E_{\mu, i}$ are i.i.d. and $\PP(\tilde E_{\mu, i})\le \varrho/2$ (so $\PP(\tilde E_{\mu, i}^c)\ge 1-\varrho/2$). Set $p=1-\varrho/2\in(1/2,1)$. Set $\zeta = p-1/2-\epsilon$, so $p-\zeta = 1/2+\epsilon$. Then our probability is bounded by  \begin{align*}
            \PP(\mathrm{Bin}(N, \PP(\tilde E_{\mu, i})) \ge N/2- \epsilon N) &=  \PP(\mathrm{Bin}(N, \PP(\tilde E_{\mu, i}^c)) \le   N/2+\epsilon N) \\
            &= \PP(\mathrm{Bin}(N, \PP(\tilde E_{\mu, i}^c)) \le N(p-\zeta) ) \\
             &\le \PP(\mathrm{Bin}(N, p) \le N(p-\zeta) ) \\
            &\le \exp\left(-N\cdot D(p-\zeta \| p)\right),
        \end{align*}  where we repeat the reasoning leading up to \eqref{chernoff:bound} in the proof of Theorem \ref{theorem:main:testing:result} with stochastic dominance and a Chernoff bound.

        Recalling the definition of $D(q \| p)$ from  \eqref{chernoff:bound}, we compute \begin{align*}
            D((p-\zeta) \| p) &= (p-\zeta) \log\frac{p-\zeta}{p} + (1-(p-\zeta))\log\frac{1-(p-\zeta)}{1-p} \\
            &= (1/2+\epsilon)\log\frac{1/2+\epsilon}{1-\varrho/2} + (1/2-\epsilon)\log\frac{1/2-\epsilon}{\varrho/2} \\
            &\ge (1/2+\epsilon)\log(1/2+\epsilon) +(1/2-\epsilon)\log\frac{1/2-\epsilon}{\varrho/2}  \\
            &= g(\epsilon) + (1/2-\epsilon) \log(1/\varrho),
        \end{align*} where $g$ was defined in Lemma \ref{lemma:function:g:technical:details}. In the third line, we used that $x\mapsto \tfrac{1/2+\epsilon}{1-x/2}$ has derivative $\tfrac{2\epsilon+1}{(2-x)^2}$, so the map is increasing. Thus, we can lower bound the first term by setting $\varrho=0$. 

        If we can ensure $(1/2-\epsilon)\log(1/\varrho) \ge -2g(\epsilon)$ like in Theorem \ref{theorem:main:testing:result}, we will have $D((p-\zeta)\| p)\ge (1/2)(1/2-\epsilon)\log(1/\varrho)$. This means our exponential term above will be bounded by \begin{align*}
            \exp\left(-\frac{N(1/2)(1/2-\epsilon)\cdot \beta \gamma R^2}{8\sigma^2}\right)\le  \exp\left(-\frac{N(1/2-\epsilon)\gamma R^2}{32\sigma^2}\right),
        \end{align*} where we used that $\beta= 1-\tfrac{8\sigma^2\log 2}{\gamma R^2}\in(1/2,1)$ from our requirements on $R$, and $\gamma\in(0,1)$ is just some constant.     

        Well, we can have $(1/2-\epsilon)\log(1/\varrho) \ge -2g(\epsilon)$ if $\log(1/\varrho) \ge \frac{-2g(\epsilon)}{1/2-\epsilon}$. Using \eqref{enum:g:limit:0}, \eqref{enum:g:limit:1/2}, and \eqref{enum:g:increasing} of Lemma \ref{lemma:function:g:technical:details}, we know $\frac{-2g(\epsilon)}{1/2-\epsilon} \in(\log 4,\infty)$. On the other hand, $\log(1/\varrho) =\tfrac{\beta \gamma R^2}{8\sigma^2} > \tfrac{\gamma R^2}{16\sigma^2}$ using $\beta>1/2$. Thus, $\log(1/\varrho) \ge \frac{-2g(\epsilon)}{1/2-\epsilon}$ follows by assuming $R> \sqrt{\frac{-32 \sigma^2 g(\epsilon)}{\gamma(1/2-\epsilon)}}$, as we did in the second requirement of \eqref{eq:R:conditions}. 
    \end{proof}

    \begin{proof}[Proof of Lemma \ref{lemma:mpacking:2mcovering}]
    Take any rational point $p \in \QQ^n$ and $r = m/2$, and consider the closed Euclidean ball  $B(p, r)$ centered at $p$ with radius $r$. Clearly $\cup_{p \in \QQ^n} B(p, r) = \RR^n$ for any $r > 0$. Exclude any $p$ such that $B(p, r) \cap T = \varnothing$ by letting $P = \{p \in \QQ^n: B(p, r) \cap T \neq \varnothing\}$. Clearly $P$ is countable and in addition $T \subset \cup_{p \in P} B(p,r)$. 

    Now for each $p\in P$, we are going to select a point $\hat p=\hat p(p)$ from the set $B(p, r) \cap T$. Specifically, for each $p \in P$ we will select the smallest point in the lexicographic ordering that belongs to these sets. First fix $p \in P$. Solve the following minimization problem: $\hat p_1 := \min_{\nu \in B(p, r) \cap T} \nu_1$ where $\nu_i$ represents the $i$th coordinate of $\nu$. The minimum is attained since the set $B(p, r) \cap T$ is compact and $\nu \mapsto \nu_1$ is a continuous function. Next, solve the minimization: $\hat p_2 := \min_{\nu \in B(p, r) \cap T, \nu_1 = \hat p_1} \nu_2$. Observe that we have a minimization since the set $B(p, r) \cap T \cap \{ \nu:  \nu_1 = \hat p_1\}$ since subspaces are closed and the intersection of closed sets is closed. For the third coordinate, solve $\hat p_3 := \min_{\nu \in B(p, r) \cap T, \nu_1 = \hat p_1, \nu_2 = \hat p_2} \nu_3$, and repeat this up to the $n$th coordinate. Then it is clear that the point $\hat p = (\hat p_1, \hat p_2, \ldots, \hat p_n) \in B(p,r) \cap T$ (in fact, this is the smallest point lexicographically). Now consider $\cup_{p \in P} B(\hat p, 2r)$. This is a countable collection which covers $T$, and each point $\hat p \in T$. The covering claim follows since if $x\in T$, then $x\in B(p,r)$ for some $p\in  P$. But $B(p,r) \subset B(\hat p, 2r)$ since for any $x \in B(p,r)$, the triangle inequality implies $\|x-\hat p\| \leq \|x- p\| + \|p- \hat p\| \leq 2r$. Hence, we have obtained a $2r = m$ cover of $T$ with points that are in $T$.

    Now we prune this cover by the following procedure. Let $f$ be a bijection from the set $\NN$ onto $F = \{\hat p: p \in P\}$. Set $L = \varnothing$. Our procedure is as follows: for each $i \in \NN$, if $\inf_{l \in L} \|l - f(i)\| > m$  (where infimum over the empty set is assumed to be $\infty$), then update $L = L \cup \{f(i)\}$.

    Clearly this procedure produces a set $L$ which is countable and an $m$-packing set. Now we verify that the set $L$ satisfies the property that for any $x \in T$: $\inf_{l \in L} \|l - x\| \leq 2m$. Note that for any $x \in T$ there exists a $\hat p$ constructed from some $p \in P$ such that $\|\hat p - x\| \leq m$. Now, if $\hat p \in L$ there is nothing to prove. If $\hat p \not \in L$ this means that $\hat p$ was pruned at some step, which implies for any $\delta > 0$ the existence of a point $q \in L$ such that $\|\hat p -q \| \leq m + \delta$ so that $\|q - x\| \leq \|\hat p - q\| + \|\hat p - x\| \leq 2m + \delta$ which completes the proof upon taking $\delta \rightarrow 0$.

    Finally, we argue that the set $L$ satisfies the definition of covering as defined in the statement. Above we only showed that for any $x \in T$ that $\inf_{l \in L} \|l - x\| \leq 2m$. Consider the set $B(x,2m + \gamma) \cap T$ for some small $\gamma$. Observe that there can only be finitely many points that belong to both this set and $L$ since the points in $L$ form an $m$-packing. On the other hand, by $\inf_{l \in L} \|l - x\| \leq 2m$, we know that there exists a point in $L$ for any value of $\gamma$ in the set $B(x,2m + \gamma) \cap T$. Since the set $B(x,2m) \cap T$ is closed, it follows that there exists some $l \in L$ which belongs to that set, proving that $\min_{l \in L}\|l-x\|\leq 2m$. This completes the proof.
\end{proof}

    \begin{proof}[\hypertarget{proof:lemma:cauchy:sequence:unbounded}{Proof of Lemma \ref{lemma:cauchy:sequence:unbounded}}] 
    Recall $\Upsilon_1$ is our chosen root node from the countable packing $S_m$, and $\Upsilon_2$ is picked from a $d_m/c$-packing of $B(\Upsilon_1, d_m)\cap K$. Then for each $k\ge 3$, $\Upsilon_{k}$ is picked by first taking a $\tfrac{d_m}{2^{k-1}c}$-packing of $B(\Upsilon_{k-1}, \tfrac{d_m}{2^{k-2}})\cap K\cap B(\Upsilon_1,d_m)$ and then possibly pruning the tree. If $\Upsilon_k$ was pruned, that means for some $u$ in this $\tfrac{d_m}{2^{k-1}c}$-packing of $B(\Upsilon_{k-1}, \tfrac{d_m}{2^{k-2}})\cap K\cap B(\Upsilon_1,d_m)$, we have $\|\Upsilon_k-u\|\le \tfrac{d_m}{2^{k-1}c}$. By the triangle inequality, \begin{align*}
        \|\Upsilon_k-\Upsilon_{k-1}\|\le \|\Upsilon_k-u\|+\|u-\Upsilon_{k-1}\| \le \tfrac{d_m}{2^{k-1}c} +\tfrac{d_m}{2^{k-2}},
    \end{align*} noting our bound on $\|u-\Upsilon_{k-1}\|$ followed from $u\in B(\Upsilon_{k-1}, \tfrac{d_m}{2^{k-2}})$. If $\Upsilon_k$ was not pruned, $\|\Upsilon_k-\Upsilon_{k-1}\|\le \tfrac{d_m}{2^{k-2}}$. Thus, for all $k\ge 3$, we have $\|\Upsilon_k-\Upsilon_{k-1}\|\le \tfrac{d_m}{2^{k-1}c} +\tfrac{d_m}{2^{k-2}}$, and clearly $\|\Upsilon_2-\Upsilon_1\|\le d_m$.

    If $J\ge J'\ge 2$, then \begin{align*}
        \|\Upsilon_J-\Upsilon_{J'}\| &\le \sum_{k=J'+1}^J \|\Upsilon_k-\Upsilon_{k-1}\| \le \sum_{k=J'+1}^J \left[\tfrac{d_m}{2^{k-1}c} +\tfrac{d_m}{2^{k-2}}\right] \\ &=  \tfrac{2(2c+1)d_m}{c}\left[\tfrac{1}{2^{J'}}-\tfrac{1}{2^J}\right] \le \tfrac{d_m(2+4c)}{c 2^{J'}}.
    \end{align*}
    
    If $J'=1$ and $J\ge 1$, then clearly $\|\Upsilon_J-\Upsilon_{1}\| \le d_m$ since both points belong to $K\cap B(\Upsilon_1, d_m)$, but $d_m\le \frac{d_m(2+4c)}{c\cdot 2}$ clearly holds. Therefore, for all $J\ge J'\ge 1$, we have $\|\Upsilon_J-\Upsilon_{J'}\|\le  \frac{d_m(2+4c)}{c 2^{J'}}$.
    \end{proof}

 The following lemma is essentially the same as Lemma \ref{lemma:for:theorem:gaussian} except we must condition everywhere on $S\ne\varnothing$ and add an exponential to control $\PP(\|\Upsilon_1-\mu\|> d_m|S\ne\varnothing)$, which we just showed is upper bounded by $\exp\left(-\Omega(N R^2/(8\sigma^2))\right)$, omitting some constants. The proof is also more complicated in that we bound $\PP(A_j\cap A_{j-1}^c\cap A_1^c|S\ne\varnothing)$ rather than $\PP(A_j\cap A_{j-1}^c|S\ne\varnothing)$ where $A_j$ is the event $\|\Upsilon_j-\mu\|> d_m$. This is because we wish to take a union bound over all possible points in $S_m$ (which is how we pick $\Upsilon_1$), while in the bounded setting, we know in advance which particular set to apply the algorithm to.

\begin{lemma} \label{lemma:for:theorem:gaussian:unbounded} Consider the Gaussian noise setting. Let $\eta_J$ be defined as in Theorem \ref{theorem:unbounded:version}. Suppose $\tilde J$ is such  \eqref{eq:robust:unbounded:theorem:condition} holds and also $\frac{d_m}{2^{\tilde J-1}(C+1)} \ge C_1(\kappa)\epsilon\sigma$. Let $C_6(\kappa) = \frac{\kappa \gamma C^2}{256 C_3(\kappa)}$, an absolute constant depending only on $C$ and $\kappa$. Then for each such $1\le J \le \tilde{J}$, we have \[ \PP\left(\|\Upsilon_{J} - \mu\| > \tfrac{d_m}{2^{J-1}}|S\ne\varnothing\right)
             \le \exp\left(-\tfrac{NC_6(\kappa)\eta_J^2}{2\sigma^2}\right) + 4\cdot\mathbbm{1}(J>1)\exp\left(-\tfrac{N\eta_J^2}{2\sigma^2}\right).\]
\end{lemma}
    \begin{proof}[\hypertarget{proof:lemma:for:theorem:gaussian:unbounded}{Proof of Lemma \ref{lemma:for:theorem:gaussian:unbounded}}] 
    If $\tilde{J}$ satisfies the stated conditions, so will all $1\le J\le \tilde{J}$. For $j\ge 1$ set $A_j=\{\|\Upsilon_j-\mu\| > \tfrac{d_m}{2^{j-1}}\}$. If a tree is rooted at a point $\Upsilon_1=s$, denote $\cL^s(j)$ to be the $j$th level of the directed tree rooted at that point. 

      Recall that our third requirement in \eqref{eq:R:conditions} implies $R>\sqrt{\frac{32\sigma^2\log 2}{N(1/2-\epsilon)\gamma}}$ which is equivalent to $\exp\left(-\frac{N(1/2-\epsilon)\gamma R^2}{32\sigma^2}\right)\le 1/2$. Thus, using Lemma \ref{lemma:unbounded:mu:S}, \[\PP(S\ne\varnothing)\geq \PP(\mu \in S) \geq 1 - \exp\left(-\tfrac{N(1/2-\epsilon)\gamma R^2}{32\sigma^2}\right)\ge 1/2.\]  
        
    Now, our goal is to control $\PP(A_J|S\ne\varnothing)$ so we will apply \eqref{eq:set:intersection:general} of Lemma \ref{lemma:set:complement:induction}.  Let us start with $\PP(A_j\cap A_{j-1}^c\cap A_1^c|S\ne\varnothing)$ for $j\ge 3$. Note that if $\|\Upsilon_1-\mu\|\le d_m$, then $\Upsilon_1$ belongs to a ball $B(\mu, d_m)\cap S_m$, where recall $S_m$ is our countable $m$-packing. Moreover, for the directed tree at $\Upsilon_1=s$, if $\Upsilon_{j-1}$ at level $j-1$ satisfies $ \|\Upsilon_{j-1} - \mu\| \leq \tfrac{d_m}{2^{j-2}}$, then $\Upsilon_{j-1}\in \cL^s(j-1)\cap B(\mu, \tfrac{d_m}{2^{j-2}})$. Thus, we can apply two union bounds and substitute the update rule for $\Upsilon_j$ with $\delta = \tfrac{d_m}{2^{j-1}(C+1)}$.  Along the way, we will be conditioning on $S\ne\varnothing$, applying Bayes' rule, then dropping it from the probability and using $\PP(S\ne\varnothing)\ge 1/2$ as established above.
         \begin{align*}
            \MoveEqLeft \PP(A_j\cap A_{j-1}^c\cap A_1^c| S\ne\varnothing) \\ &=  \frac{1}{\PP(S\ne\varnothing)}\PP(\|\Upsilon_j - \mu\| > \tfrac{d_m}{2^{j-1}}, \|\Upsilon_{j-1} - \mu\| \leq \tfrac{d_m}{2^{j-2}}, \\&\qquad\qquad\qquad\quad\|\Upsilon_1 - \mu\| \leq d_m, S\ne\varnothing)\\
            &\leq 2\!\!\!\!\sum_{\substack{s\in S_m\\ \|s-\mu\|\le d_m}} \sum_{\substack{u \in \cL^s(j-1) \cap \\ B(\mu,  \tfrac{d_m}{2^{j-2}})}} \!\!\PP(\|\Upsilon_j - \mu\| > \tfrac{d_m}{2^{j-1}}, \Upsilon_{j-1} = u, \Upsilon_1=s, S\ne \varnothing) \\
            & =2\!\!\!\!\sum_{\substack{s\in S_m\\ \|s-\mu\|\le d_m}} \sum_{\substack{u \in \cL^s(j-1) \cap \\ B(\mu,  \tfrac{d_m}{2^{j-2}})}}\!\!\PP\bigg(\Big\|\argmin_{\nu \in \cO(u)} T(\delta, \nu, \cO(u)) - \mu\Big\| > (C+1)\delta, \\&\qquad\qquad\qquad\qquad \Upsilon_{j-1} = u, \Upsilon_1=s, S\ne \varnothing\bigg)\\
            & \leq 2\!\!\!\!\sum_{\substack{s\in S_m\\ \|s-\mu\|\le d_m}} \sum_{\substack{u \in \cL^s(j-1) \cap \\ B(\mu,  \tfrac{d_m}{2^{j-2}})}} \PP\left(\Big\|\argmin_{\nu \in \cO(u)} T(\delta, \nu, \cO(u)) - \mu\Big\| > (C+1)\delta\right).
        \end{align*} 
        
        By Lemma \ref{lemma:bound:level:J:intersect:ball:mu:unbounded}, $\left|\cL^s(j-1)\cap B(\mu, \tfrac{d_m}{2^{j-2}})\right|$ is bounded by $\cMKloc(\tfrac{d_m}{2^{j-2}}, 2c)$. Moreover, since $S_m$ is an $m$-packing, the number of choices for $\Upsilon_1$ in $B(\mu,d_m)$ is bounded by $\cMKloc(d_m, c')\le \cMKloc(\tfrac{d_m}{2^{j-2}}, c')$ where $c' = \tfrac{d_m}{m}$. Recall we took $m = \tfrac{R}{c-1}$, so 
        \[c'=\frac{2R+2m}{m} = \frac{2R}{m}+2 = \frac{2R(c-1)}{R}+2= (2c-2)+2=2c.\] 
        For a fixed choice of $\Upsilon_1=s$ and $u\in \cL^s(j-1) \cap B(\mu,  \tfrac{d_m}{2^{j-2}})$, set $K' = B(u, \tfrac{d_m}{2^{j-2}})\cap K \cap B(s, d_m)$. Observe by Lemma \ref{lemma:pruned:tree:properties:unbounded} that $\cO(u)$ is a $\tfrac{d_m}{2^{j-2}c}=\tfrac{d_m}{2^{j-1}(C+1)}=\delta$-covering of the set $B(u, \tfrac{d_m}{2^{j-2}}) \cap K\cap B(s, d_m)$ with cardinality bounded by $\cMKloc(\tfrac{d_m}{2^{j-2}}, 2c)$. Thus, by our tournament bound in Lemma \ref{lemma:tournament} followed by our union bounds from earlier and bounds on the number of summands, we have for $j\ge 3$ that \begin{align*}
             \MoveEqLeft \PP(A_j\cap A_{j-1}^c\cap A_1^c| S\ne\varnothing) \\ &\le 2\!\!\!\!\sum_{\substack{s\in S_m\\ \|s-\mu\|\le d_m}} \sum_{\substack{u \in \cL^s(j-1) \cap \\ B(\mu,  \tfrac{d_m}{2^{j-2}})}} \cMKloc(\tfrac{d_m}{2^{j-2}}, 2c)\exp(-\tfrac{C_3(\kappa) N\delta^2}{\sigma^2}) \\
             &\le 2\left[\cMKloc(\tfrac{d_m}{2^{j-2}}, 2c)\right]^3 \exp(-\tfrac{C_3(\kappa) N\delta^2}{\sigma^2}) \\
             &=  2\left[\cMKloc(\tfrac{d_m}{2^{j-2}}, 2c)\right]^3 \exp(-\tfrac{N\eta_j^2}{\sigma^2}). 
        \end{align*}

        Next, note that $A_1$ is the event $\|\Upsilon_1-\mu\|> d_m$, which implies $\mu\not\in B(\Upsilon_1,d_m)\cap K$. By \eqref{eq:unbounded:mu:K_m}, we have \[\PP(A_1|S\ne\varnothing)\le\exp\left(-\tfrac{N(1/2-\epsilon) \gamma R^2}{32\sigma^2}\right).\] 

        For $j=2$, we again apply a union bound over all possible $\Upsilon_1=s$, and note that $\Upsilon_2$ is chosen from $\cO(s)$. Set $\delta = \tfrac{d_m}{c}=\tfrac{d_m}{2(C+1)}$. Then we have  \begin{align*}
            \MoveEqLeft \PP(A_2\cap A_1^c|S\ne\varnothing)\\ &= \frac{1}{\PP(S\ne\varnothing)}\PP(\|\Upsilon_2-\mu\| > d_m/2 , \|\Upsilon_1-\mu\| \le d_m|S\ne\varnothing)  \\ &\le 2\!\!\!\!\sum_{\substack{s\in S_m\\ \|s-\mu\|\le d_m}}  \PP(\|\Upsilon_2-\mu\|\ge d_m/2 , \Upsilon_1=s,S\ne\varnothing) \\
            &= 2\!\!\!\!\sum_{\substack{s\in S_m\\ \|s-\mu\|\le d_m}}  \PP\left(\Big\|\argmin_{\nu\in\cO(s)} T(\delta,\nu,\cO(s))-\mu\Big\|\ge \delta(C+1) , \Upsilon_1=s,S\ne\varnothing\right) \\
            &\le  2\!\!\!\!\sum_{\substack{s\in S_m\\ \|s-\mu\|\le d_m}}  \PP\left(\Big\|\argmin_{\nu\in\cO(s)} T(\delta,\nu,\cO(s))-\mu\Big\|\ge \delta(C+1)\right).
        \end{align*} We again apply Lemma \ref{lemma:tournament} with $K'=B(s, d_m)\cap K$ which is covered by a $\delta=d_m/c$ maximal packing set $\cO(s)$ (with no pruning). Noting that $|\cO(s)|\le \cMKloc(d_m, 2c)$ (see Lemma \ref{lemma:pruned:tree:properties:unbounded}), and repeating our argument about $s\in S_m$, we have \begin{align*}
             \PP(A_2\cap A_1^c|S\ne\varnothing)   &\le 2\!\!\!\!\sum_{\substack{s\in S_m\\ \|s-\mu\|\le d_m}}\cMKloc(d_m, 2c)\exp\left(-\tfrac{C_3(\kappa)N\delta^2}{\sigma^2}\right) \\
             &\le 2\left[\cMKloc(d_m, 2c)\right]^2\exp\left(-\tfrac{N\eta_2^2}{\sigma^2}\right) \\
             &\le  2\left[\cMKloc(d_m, 2c)\right]^3\exp\left(-\tfrac{N\eta_2^2}{\sigma^2}\right),
        \end{align*} since $\cMKloc(d_m, 2c)\ge 1$.
        
        Therefore, for any $1\le J\le\tilde J$, we apply  \eqref{eq:set:intersection:general} (but conditional on $S\ne\varnothing$) and repeat the logic of Lemma \ref{lemma:for:theorem:gaussian} but with $\PP(A_1|S\ne\varnothing)$ now non-zero. We then have \begin{align*}
            \PP(\|\Upsilon_J-\mu\|> \tfrac{d_m}{2^{J-1}}|S\ne\varnothing) &\le \exp\left(-\tfrac{N(1/2-\epsilon)\gamma R^2}{32\sigma^2}\right) \\ &+ 2\cdot \mathbbm{1}(J>1)\cdot \left[\cMloc\left(\tfrac{c\eta_J}{\sqrt{C_3(\kappa)}}, 2c\right)\right]^3 \frac{a_J}{1-a_J}
        \end{align*} where $a_J = \exp(-\tfrac{N\eta_J^2}{\sigma^2})$. But then as before, assuming \eqref{eq:robust:unbounded:theorem:condition} holds, we can bound the second term and conclude \begin{align*}
            \MoveEqLeft\PP(\|\Upsilon_J-\mu\| > \tfrac{d_m}{2^{J-1}}|S\ne\varnothing) \\ &\le \exp\left(-\tfrac{N (1/2-\epsilon) \gamma R^2}{32\sigma^2}\right)+ 4\cdot\mathbbm{1}(J>1)\exp\left(-\tfrac{N\eta_J^2}{2\sigma^2}\right).
        \end{align*} Next, we can compute from $m=\tfrac{R}{c-1}$ that $d_m = 2R+2m = \tfrac{2c}{c-1}\cdot R$. Thus \[ \eta_J = \tfrac{d_m\sqrt{C_3(\kappa)}}{2^{J-1}(C+1)} =\tfrac{2c\sqrt{C_3(\kappa)}}{2^{J-1}(C+1)(c-1)}\cdot R = \tfrac{4\sqrt{C_3(\kappa)}}{2^{J-1}(2C+1)}\cdot R. \] Consequently, $R = \tfrac{2^{J-1}(2C+1)}{4\sqrt{C_3(\kappa)}}\eta_J \ge \tfrac{C}{4\sqrt{C_3(\kappa)}}\eta_J$,  and \[\exp\left(-\tfrac{N(1/2-\epsilon)\gamma R^2}{32\sigma^2}\right)\le \exp\left(-\tfrac{N(1/2-\epsilon) \gamma C^2\eta_J^2}{512\sigma^2C_3(\kappa)}\right) \le \exp\left(-\tfrac{N\kappa\gamma C^2\eta_J^2}{512\sigma^2C_3(\kappa)}\right). \] The last step used that $1/2-\epsilon > \kappa$ in our Gaussian setting. Finally, we set $C_6(\kappa) = \frac{\kappa\gamma C^2}{256 C_3(\kappa)}$, we bound our probability by $\exp\left(-\frac{NC_6(\kappa)\eta_J^2}{2\sigma^2}\right)$. This yields our claimed bound.
    \end{proof}

In the bounded case, we used the equivalent of  Lemma \ref{lemma:for:theorem:gaussian:unbounded}  to construct a bound on $\PP(\|\mu-\nu^{\ast\ast}\|\gtrsim x)$ for all $x\ge\eta_{\tilde J}$. That proof relied on arguing $\PP(\|\mu-\nu^{\ast\ast}\|\gtrsim \eta_J)=0$ for all $J\le 0$ since $\eta_J\gtrsim d_m$. But since our set is unbounded now, we no longer have this fact, so our proof becomes considerably more complex. Moreover, we consider cases for whether $S=\varnothing$ and apply the law of total expectation for our final bound. We will show the contribution from the ``trivial'' part of the estimator (when $S=\varnothing$) is $\tfrac{\sigma^2}{N}$, and in the subsequent theorem, demonstrate this won't affect the minimax rate.

\begin{lemma} \label{lemma:for:theorem:gaussian:unbounded:part2}  Consider the Gaussian noise setting. Let $\eta_J$ be defined as in Theorem \ref{theorem:unbounded:version}. Suppose $\tilde J$ is such \eqref{eq:robust:unbounded:theorem:condition} holds and also $\frac{d_m}{2^{\tilde J-1}(C+1)} \ge C_1(\kappa)\epsilon\sigma$. Then if $\nu^{\ast\ast}$ denotes the output after at least $J^{\ast}$ iterations of Algorithm \ref{algorithm:robust} in the $S(R)\ne\varnothing$ case or the smallest lexicographic point in $S(\hat R)$ in the $S(R)=\varnothing$ case, we have for some constant $C_9(\kappa)$ depending on $C$ and $\kappa$ that \begin{align*}
    \EE_X\|\mu-\nu^{\ast\ast}\|^2 
    \lesssim \eta_{\tilde J}^2 +\tfrac{\sigma^2}{N}\exp\left(-\tfrac{NC_9(\kappa) \eta_{\tilde J}^2}{\sigma^2}\right) +\tfrac{\sigma^2}{N}.
\end{align*}
\end{lemma}
\begin{proof} 

The proof is split up into six parts. In Parts I-III, we bound $\PP(\|\mu-\nu^{\ast\ast}\|\gtrsim y|S(R)\ne\varnothing)$ for all $y\ge \eta_{\tilde J}$, which we do by separately considering cases where $y\gtrsim R$ and $\eta_{\tilde J}\le y\lesssim R$ and stitching together the results. In Part IV, we deduce a bound on $\EE_X[\|\mu-\nu^{\ast\ast}\|^2|S\ne\varnothing]$. Part V considers the $S(R)=\varnothing$ scenario, in which case we set the estimator to be $\hat p\in S(\hat R)$ (defined in the main text). For any suitable $k\in\NN$ we bound $\EE[\|\hat p-\mu\|^2|2^{k-1}R<\hat R\le 2^kR]$ as well as $\PP(2^{k-1}R<\hat R\le 2^kR)$. Applying the law of total expectation in Part VI will let us bound the unconditional risk $\EE_X\|\mu-\nu^{\ast\ast}\|^2$ using our conditional risk bounds.  

\noindent\textsc{Part I: Bounding $\PP(\|\mu-\nu^{\ast\ast}\|\gtrsim t|S\ne\varnothing)$ for $t\gtrsim R$}.

 Let $C_6(\kappa)$ be defined as in Lemma \ref{lemma:for:theorem:gaussian:unbounded}. Lemma \ref{lemma:for:theorem:gaussian:unbounded} proved for $1\le J\le \tilde J$ that
\[ \PP\left(\|\Upsilon_{J} - \mu\| > \tfrac{d_m}{2^{J-1}}|S\ne\varnothing\right)
             \le \exp\left(-\tfrac{NC_6(\kappa)\eta_J^2}{2\sigma^2}\right) + 4\cdot\mathbbm{1}(J>1)\exp\left(-\tfrac{N\eta_J^2}{2\sigma^2}\right).\] Suppose $S\ne\varnothing$. Recall that we had a $2m$-covering of $K$ (hence $S$) and $\Upsilon_1$ is the closest point in that covering to $S$ (using lexicographic ordering to break ties). Hence picking any $p\in S(R)$, we have \begin{equation} \label{eq:unbounded_upsilon_1_minus_mu}
    \|\Upsilon_1 - \mu\| \leq \|\Upsilon_1 - p\| + \|p - \mu\| \leq (2m+2R) + \|p - \mu\|, 
\end{equation} using $\mathrm{diam}(S(R))\le 2R$ and the fact $\Upsilon_1$ is from a $2m$-covering to argue $\|\Upsilon_1-p\|\le 2m+2R$. Then for $t \geq R$, \begin{align} \label{eq:unbounded_p_minus_mu}
    \PP(\|p - \mu\| >2t | S \neq \varnothing) &\leq \frac{\PP(\mu \not \in S(t), S \neq \varnothing)}{\PP(S \neq \varnothing)} \leq \frac{\PP(\mu \not \in S(t))}{\PP(S \neq \varnothing)} \\ &\le 2\exp\left(-\tfrac{N (1/2-\epsilon)\gamma t^2}{32\sigma^2}\right). \notag
\end{align} The first inequality follows since $p\in S(R)$ implies $p\in S(t)$ using Lemma \ref{lemma:nesting:sets:S}, but both $p$ and $\mu$ cannot belong to $S(t)$ if $\|p-\mu\|>2t$; otherwise, it would violate $\mathrm{diam}(S(t))\le 2t$ (Lemma \ref{lemma:diam:S:2R}). We also used expanded the conditional probability. The second inequality dropped the intersection and the third applied Lemma \ref{lemma:unbounded:mu:S}. We also used that $\PP(S\ne\varnothing)>1/2$, as we proved in Lemma \ref{lemma:for:theorem:gaussian:unbounded} using our third requirement on $R$ in \eqref{eq:R:conditions}. Hence for $t\ge R$, \eqref{eq:unbounded_upsilon_1_minus_mu} and \eqref{eq:unbounded_p_minus_mu} imply  \begin{align*}
    \PP(\|\Upsilon_1-\mu\|> 2t + 2m+2R |S\ne\varnothing)  \le 2\exp\left(-\tfrac{N (1/2-\epsilon)\gamma t^2}{32\sigma^2}\right).
\end{align*}
Recall $\|\nu^{\ast\ast}-\Upsilon_1\|\le \frac{(2m+2R)(1+2c)}{c}$ by our Cauchy sequence result (applying Lemma \ref{lemma:cauchy:sequence:unbounded}). By the triangle inequality,  for $t\ge R$ \begin{align*}
    \MoveEqLeft \PP(\|\mu-\nu^{\ast\ast}\|> 2t + 2m+2R  + \tfrac{(2m+2R)(1+2c)}{c}|S\ne\varnothing) \\ &\le 2\exp\left(-\tfrac{N(1/2-\epsilon)\gamma t^2}{32\sigma^2}\right).
\end{align*} Recall we took $m=\tfrac{R}{c-1}$, so $2m+2R=\tfrac{2c}{c-1}\cdot R \le \frac{c}{c-1}\cdot 2t$. Hence \[2t + 2m+2R  + \tfrac{(2m+2R)(1+2c)}{c} \le 2t\left(1+\tfrac{c}{c-1}\left(1+\tfrac{1+2c}{c}\right)\right) = t\cdot \tfrac{8c}{c-1}.\] Set $C_8 = \tfrac{8c}{c-1}= \tfrac{16(C+1)}{2C+1}$. Hence for $t\ge R$, \begin{align}
    \PP(\|\mu-\nu^{\ast\ast}\|>C_8 t |S\ne\varnothing) \le 2\exp\left(-\tfrac{N(1/2-\epsilon)\gamma t^2}{32\sigma^2}\right) \le 2\exp\left(-\tfrac{N\kappa\gamma t^2}{32\sigma^2}\right).  \label{eq:unbounded:aux:1}
\end{align} We used $1/2-\epsilon>\kappa$ in the last inequality. 

\noindent\textsc{Part II: Bounding $\PP(\|\mu-\nu^{\ast\ast}\|\gtrsim t|S\ne\varnothing)$ for $\eta_{\tilde J} \le t \lesssim R$}.

Now we prove a similar bound for $t\lesssim R$ by re-using logic from the bounded case. That is, repeat the proof of Lemma \ref{lemma:for:theorem:gaussian:part2} through \eqref{eq:bounded:mu_minus_nu_ast} with our modified bounds and conclude for $1\le J\le \tilde J$ that \[\PP(\|\mu-\nu^{\ast}\|>\omega \eta_J|S\ne\varnothing)\le \exp\left(-\tfrac{NC_6(\kappa)\eta_J^2}{2\sigma^2}\right) + 4\cdot\mathbbm{1}(J>1)\exp\left(-\tfrac{N\eta_J^2}{2\sigma^2}\right),\] where $\omega=\tfrac{7+6C}{2\sqrt{C_3(\kappa)}}$ and $\nu^{\ast}$ is $\Upsilon_{J^{\ast}}$ (the output of $J^{\ast}-1$ steps).

 Observe that  $\bigcup_{1<J\le \tilde{J}}[\eta_J,\eta_{J-1})=[\eta_{\tilde{J}},\eta_1)$ and $\eta_J = \eta_{J-1}/2$. Moreover, any $\eta_{\tilde{J}}\le x <\eta_1$ belongs to some interval $[\eta_J,\eta_{J-1})$ for $1<J\le \tilde{J}$ and therefore satisfies $2\omega x \ge 2\omega \eta_J=\omega\eta_{J-1}$. Hence for $x\in[\eta_{\tilde{J}},\eta_1)$, \begin{align*}
            \PP(\|\mu-\nu^{\ast}\|> 2\omega x|S\ne\varnothing) &\le \PP(\|\mu-\nu^{\ast}\|> \omega \eta_{J-1}|S\ne\varnothing) \\
            &\le  \exp\left(-\tfrac{NC_6(\kappa)\eta_J^2}{2\sigma^2}\right) + 4\cdot\mathbbm{1}(J>1)\exp\left(-\tfrac{N\eta_{J-1}^2}{2\sigma^2}\right) \\
            &\le  \exp\left(-\tfrac{NC_6(\kappa)x^2}{2\sigma^2}\right) + 4\cdot\mathbbm{1}(J>1)\exp\left(-\tfrac{Nx^2}{2\sigma^2}\right). 
        \end{align*} 
Let $\nu^{\ast\ast}=\Upsilon_{J^{\ast}+1}$ be the output of $J^{\ast\ast}\ge J^{\ast}$ steps. In the lines leading up to \eqref{eq:theorem:triangle:inequality:for:J_ast_}, we showed $\|\nu^{\ast}-\nu^{\ast\ast}\|\le \tfrac{5+4C}{7+6C}\omega x$ since $x\ge \eta_{\tilde J}$. This implies by the triangle inequality $\|\mu-\nu^{\ast\ast}\| \le \|\mu-\nu^{\ast}\|+ \tfrac{5+4C}{7+6C}\omega x$ for such $x$. Setting $\omega' = \left(2+\tfrac{5+4C}{7+6C}\right)\omega=\tfrac{19+16C}{2\sqrt{C_3(\kappa)}}$, we obtain for $x\in[\eta_{\tilde{J}},\eta_1)$ that \begin{align}
    \PP(\|\mu-\nu^{\ast\ast}\|>\omega'x|S\ne\varnothing) &\le \PP(\|\mu-\nu^{\ast}\|+\tfrac{5+4C}{7+6C}\omega x> \omega' x|S\ne\varnothing) \notag \\
            &= \PP(\|\mu-\nu^{\ast}\|> 2\omega x|S\ne\varnothing) \notag \\
            &\le   \exp\left(-\tfrac{NC_6(\kappa)x^2}{2\sigma^2}\right) + 4\cdot\mathbbm{1}(J>1)\exp\left(-\tfrac{Nx^2}{2\sigma^2}\right). \label{eq:unbounded:aux:2}
\end{align} 

\noindent\textsc{Part III: Bounding $\PP(\|\mu-\nu^{\ast\ast}\|\gtrsim t|S\ne\varnothing)$ for all $t \ge \eta_{\tilde J}$}.

We wish to combine our two bounds in \eqref{eq:unbounded:aux:1} and \eqref{eq:unbounded:aux:2}, so let us compare $C_8R$ and $\omega'\eta_1$, which is the smallest and largest respective deviation for $\|\mu-\nu^{\ast\ast}\|$ controlled by either bound. Well, recalling $d_m=2m+2R = \tfrac{2c}{c-1}\cdot R = \tfrac{4(C+1)}{2C+1}\cdot R$, observe \[\omega'\eta_1 = \tfrac{19+16C}{2\sqrt{C_3(\kappa)}}\cdot \tfrac{d_m\sqrt{C_3(\kappa)}}{(C+1)}= \tfrac{19+16C}{2(C+1)}\cdot (2m+2R) =\tfrac{2(19+16C)}{2C+1}\cdot R,\] while $C_8 R = \tfrac{16(C+1)}{2C+1}\cdot R.$ Thus, one can see that $\omega'\eta_1>C_8R$. 

Now, if we set $y= C_8t/\omega'$ in \eqref{eq:unbounded:aux:1}, we obtain an equivalent bound for all $y\ge C_8 R/\omega'$ that \begin{align}
    \PP(\|\mu-\nu^{\ast\ast}\|>\omega'y|S\ne\varnothing) \le 2\exp\left(-\tfrac{N\kappa{\omega'}^2\gamma y^2}{32C_8^2\sigma^2}\right).  \label{eq:unbounded:aux:1:version:2}
\end{align} But this bound clearly holds for $y\ge \eta_1$ since $\eta_1\ge C_8R/\omega'$ as we just demonstrated. Hence \eqref{eq:unbounded:aux:1:version:2} controls all deviations larger than $\omega'\eta_1$ while \eqref{eq:unbounded:aux:2} controls all deviations larger than $\omega'\eta_{\tilde J}$. Therefore, for all $y\ge \eta_{\tilde J}$, we have \begin{align*}
     \MoveEqLeft\PP(\|\mu-\nu^{\ast\ast}\|>\omega' y|S\ne\varnothing)  \\ &\le  2\exp\left(-\tfrac{N\kappa{\omega'}^2\gamma y^2}{32C_8^2\sigma^2}\right) \vee  \\ &\quad\quad \left(\exp\left(-\tfrac{NC_6(\kappa)y^2}{2\sigma^2}\right) + 4\cdot\mathbbm{1}(J>1)\exp\left(-\tfrac{Ny^2}{2\sigma^2}\right)\right) \\
     &\le 7\exp\left(-\tfrac{N C_9(\kappa)y^2}{\sigma^2}\right),
\end{align*} where $C_9(\kappa)$ is a constant that depends on $\kappa, C_8,\omega',\gamma$, i.e., only $\kappa$ and $C$. 

\noindent\textsc{Part IV: Bounding $ \EE_X\left[ \|\mu-\nu^{\ast\ast}\|^2 | S\ne\varnothing\right]$}.

Integrate as before: \begin{align}
            \MoveEqLeft \EE_X\left[ \|\mu-\nu^{\ast\ast}\|^2 | S\ne\varnothing\right] \notag \\ &= \int_0^{\infty} \PP(\|\mu-\nu^{\ast\ast}\|^2 > x|S\ne \varnothing)\mathrm{d}x \notag \\ 
            &= 2{\omega'^2}\int_0^{\infty} u \cdot \PP(\|\mu-\nu^{\ast\ast}\| > \omega' u|S\ne \varnothing) \mathrm{d}u \notag\\
            &\leq 2{\omega'^2}\int_0^{\eta_{\tilde{J}}} u\cdot \mathrm{d}u +2{\omega'^2}\int_{\eta_{\tilde{J}}}^{\infty}  u\cdot \PP(\|\mu-\nu^{\ast\ast}\| > \omega'u) \mathrm{d}u \notag\\
            &\le  {\omega'}^2\eta_{\tilde{J}}^2+ 14{\omega'}^2
            \int_{\eta_{\tilde{J}}}^{\infty} u\cdot\exp(-\tfrac{N C_9(\kappa)u^2}{\sigma^2})\mathrm{d}u \notag\\
            &=  {\omega'}^2\eta_{\tilde{J}}^2+  
            14{\omega'}^2\cdot \tfrac{\sigma^2}{2N C_9(\kappa)}\exp\left(-\tfrac{N C_9(\kappa)\eta_{\tilde{J}}^2}{\sigma^2}\right) \notag\\
            &= \tfrac{(19+16C)^2}{4C_3(\kappa)}\eta_{\tilde{J}}^2 + \tfrac{7(19+16C)^2}{2C_3(\kappa)}\cdot  \tfrac{\sigma^2}{2N C_9(\kappa)}\exp\left(-\tfrac{NC_9(\kappa) \eta_{\tilde{J}}^2}{\sigma^2}\right). \label{eq:unbounded:nontrivial:conditional:risk}
        \end{align} 

\noindent\textsc{Part V: Bounding $\PP(2^{k-1}R < \hat R\leq 2^{k} R])$ and $\EE [\|\hat p - \mu\|^2 | 2^{k-1}R < \hat R\leq 2^{k} R]$}.

  On the other hand, let us consider the scenario where $S(R)=\varnothing$. In this case, recall we set the estimator to the smallest point lexicographically in $\hat p \in S(\hat R)$ (see main text). Now we perform a so-called peeling argument on $\hat R$. For any $k\in \NN$, define $p_k:=\PP(2^{k-1}R < \hat R \le 2^k R)$. Let $p_0$ be the probability $\PP(S\ne\varnothing) = \PP(\hat R \leq R)$.  If $p_0=1$, then we immediately obtain $\EE_X[\|\mu-\nu^{\ast\ast}\|^2|S\ne\varnothing]=\EE_X\|\mu-\nu^{\ast\ast}\|^2$ and we can apply our previous bound from \eqref{eq:unbounded:nontrivial:conditional:risk}. Otherwise, we assume $p_0<1$ so that for some $k\in\NN$, $p_k>0$. Pick one such $k$. Then observe \begin{align*}
    p_k &= \PP(\hat R\leq 2^{k} R) - \PP(\hat R\leq 2^{k-1} R)\leq 1 - \PP(\mu \in S(2^{k-1} R)) \\ &\leq \exp\left(-\tfrac{N (1/2-\epsilon)\gamma (2^{k-1}R)^2}{32\sigma^2}\right) \le \exp\left(-\tfrac{N\kappa\gamma (2^{k-1}R)^2}{32\sigma^2}\right).
\end{align*} The first inequality holds since if $\mu\in S(2^{k-1}R)$, then $S(2^{k-1}R)\ne\varnothing$ so $\hat R\le 2^{k-1}R$ by definition as a minimum; the second inequality uses Lemma \ref{lemma:unbounded:mu:S} and the third $1/2-\epsilon>\kappa$. Then our estimator $\hat p$ satisfies  \begin{align}
        \PP(\|\hat p - \mu\| > 2t | 2^{k-1}R < \hat R \leq 2^k R) &= \frac{\PP(\|\hat p - \mu\| > 2t, 2^{k-1}R < \hat R \leq 2^k R)}{p_k} \notag \\
        &\leq \frac{\PP(\|\hat p - \mu\| > 2t, \hat p \in S(2^k R))}{p_k} \notag\\
        &\leq \frac{\PP(\mu \not\in S(t))}{p_k}
        \notag\\ &\leq \frac{ \exp\left(-\tfrac{N\kappa\gamma t^2}{32\sigma^2}\right)}{p_k} \label{eq:peeling:tail:bound}
    \end{align} for all $t \geq 2^k R$. The second line follows since $\hat R\le 2^k R$ implies $\hat p\in S(\hat R)\subseteq S(2^k R)$ by Lemma \ref{lemma:nesting:sets:S}, and the third since $S(2^k R)\subseteq S(t)$ with $\mathrm{diam}(S(t))\le 2t$ so both $\hat p,\mu$ cannot belong to $S(t)$ if $\|\hat p-\mu\|> 2t$. The fourth again applied Lemma \ref{lemma:unbounded:mu:S} and that $1/2-\epsilon>\kappa$.
    
    Now observe that for values of $t \ge C_{10}(\kappa)(\sigma /\sqrt{N}) \sqrt{\log 1/p_k}$ where $C_{10}(\kappa)=\sqrt{32/(\kappa\gamma)}$, the above bound is meaningful (i.e., smaller than $1$).  We integrate our tail bound in \eqref{eq:peeling:tail:bound} and obtain
    \begin{align}
       \MoveEqLeft \EE [\|\hat p - \mu\|^2 | 2^{k-1}R < \hat R\leq 2^{k} R] \notag \\ &= \int_0^{\infty}  \PP(\|\hat p - \mu\|^2 > u | 2^{k-1}R < \hat R \leq 2^k R)  \mathrm{d}u\notag\\
        &= 4\int_0^{\infty}  2t\cdot \PP(\|\hat p - \mu\| > 2t | 2^{k-1}R < \hat R \leq 2^k R)  \mathrm{d}t \notag\\ &\leq 8\int_{0}^{C_{10}(\kappa)(\sigma /\sqrt{N}) \sqrt{\log 1/p_k} \vee (2^{k} R)} t\, \mathrm{d}t  \notag\\ &\quad\quad+ 8\int_{C_{10}(\kappa)(\sigma /\sqrt{N}) \sqrt{\log 1/p_k}\vee (2^{k} R)}^{\infty} t\cdot  \tfrac{ \exp\left(-\tfrac{N\kappa\gamma t^2}{32\sigma^2}\right)}{p_k} \mathrm{d}t \notag \\
        & \lesssim (2^kR)^2\vee \tfrac{\sigma^2}{N} \log\tfrac{1}{p_k} +\tfrac{\sigma^2}{N}. \label{eq:peeling:integrated}
    \end{align} To elaborate, we first split up the integral into two portions, where in the first region, we simply bound the probability by $1$, and in the second, we apply our tail bound. The first integral is of a linear function and integrates easily to get the $(2^kR)^2\vee \tfrac{\sigma^2}{N} \log\tfrac{1}{p_k}$ expression. The second is more elaborate. Ignoring the constants, we perform a substitution $u = \sqrt{N} t /\sigma$, and then rescale the integrand and bounds of integration. We obtain an integral of the form $\int_a^{\infty} u\exp(-u^2)du\asymp \exp(-a^2)$ where $a = \sqrt{\log 1/p_k} \vee (2^k R\sqrt{N}/\sigma)$. This gives us a bound of \begin{align} \label{eq:complex:p_k:expression}
        \frac{\sigma^2}{N p_k} \exp\left(-\left[\log(1/p_k) \vee (2^k R \sqrt{N}/\sigma)^2\right]\right).
    \end{align} But notice \[\log(1/p_k) \vee (2^k R \sqrt{N}/\sigma)^2 \ge \log(1/p_k),\] hence \[-\left[\log(1/p_k) \vee (2^k R \sqrt{N}/\sigma)^2\right] \le -\log(1/p_k).\] Exponentiating both sides, we conclude \[\exp\left(-\left[\log(1/p_k) \vee (2^k R \sqrt{N}/\sigma)^2\right]\right) \le \exp(-\log(1/p_k))=p_k.\] Thus, the bound in \eqref{eq:complex:p_k:expression} reduces to $\sigma^2/N$ as claimed.

\noindent\textsc{Part VI: Bounding $\EE \|\mu - \nu^{\ast\ast}\|^2$}.

    By the law of total expectation, the risk of our estimator is exactly
    \begin{equation} \label{eq:unbounded:law:total:expectation}
        \EE [\|\nu^{**} - \mu\|^2  | S \neq \varnothing]\cdot p_0 + \sum_{\substack{k\in\NN\\ p_k>0}} p_k\cdot  \EE [\|\hat p - \mu\|^2 | 2^{k-1}R < \hat R\leq 2^{k} R]. 
    \end{equation}  Note that we have already bounded  $\EE [\|\nu^{**} - \mu\|^2  | S \neq \varnothing]$ in \eqref{eq:unbounded:nontrivial:conditional:risk}. Using \eqref{eq:peeling:integrated}, the second term in \eqref{eq:unbounded:law:total:expectation} is bounded by
    \begin{align*}
        \MoveEqLeft \sum_{\substack{k\in\NN\\ p_k>0}} p_k \left[(2^kR)^2\vee \tfrac{\sigma^2}{N} \log\tfrac{1}{p_k} +\tfrac{\sigma^2}{N}\right]  \\
        & \leq \sum_{\substack{k\in\NN\\ p_k>0}}  \exp\left(-\tfrac{N(1/2-\epsilon)\gamma (2^{k-1}R)^2}{32\sigma^2}\right)\cdot(2^kR)^2+  \sum_{\substack{k\in\NN\\ p_k>0}} p_k\cdot\tfrac{\sigma^2}{N} \log \tfrac{1}{p_k} + \tfrac{\sigma^2}{N}  \\
        & \lesssim   \tfrac{\sigma^2}{N}\cdot\sum_{\substack{k\in\NN\\ p_k>0}} \exp\left(-\tfrac{N(1/2-\epsilon)\gamma (2^{k-1}R)^2}{32\sigma^2}\right)(2^kR)^2\tfrac{N}{\sigma^2}  \\
        &\quad\quad+  \tfrac{\sigma^2}{N}\cdot\sum_{\substack{k\in\NN\\ p_k>0}} \left[\exp\left(-\tfrac{N(1/2-\epsilon)\gamma (2^{k-1}R)^2}{32\sigma^2}\right) \tfrac{N(2^{k-1} R)^2}{\sigma^2} \right]   + \tfrac{\sigma^2}{N} \\ &\lesssim \tfrac{\sigma^2}{N}. 
    \end{align*} 
    
   The first inequality above follows by $a\vee b \leq a + b$ and $\sum p_k \leq 1$.  The second inequality comes from bounding the second term as follows: we observe that $x \mapsto x \log (1/x)$ is increasing on $(0,1/e)$ and also that $p_k \leq\exp\left(-\tfrac{N (1/2-\epsilon)\gamma (2^{k-1}R)^2}{32\sigma^2}\right) \le 1/e$ (using the third requirement on $R$ in \eqref{eq:R:conditions}).  Thus, we can replace the $p_k$ in $p_k \log\tfrac{1}{p_k}$ with $\exp\left(-\tfrac{N (1/2-\epsilon)\gamma (2^{k-1}R)^2}{32\sigma^2}\right)$, and discard the $1/2-\epsilon<1/2$ (and other constants) in the $\log(1/p_k)$ term. The last inequality follows by first noting we have can bound each of the first two terms with summations of the form $\tfrac{\sigma^2}{N}\sum_{k = 1}^\infty \exp(-4^kc) (2^k  c)^2$ where $c= \tfrac{N(1/2-\epsilon)\gamma R^2}{4\cdot 32\sigma^2}$.  Using the third requirement of $R$, we can show $c>\tfrac{16 N \log 2}{4}\ge 2$. Then using $\exp(-x)\le x^{-3}$, which holds in particular for all $x\ge 5$ (and thus $4^kc$), we have  \[\sum_{k\in\NN} \exp(-4^kc) (2^k  c)^2 \le \sum_{k\in\NN}(4^kc)^{-3}\cdot (4^k c^2) = \sum_{k\in\NN}\tfrac{c^{-1}}{16^k}\lesssim 1.\]  Combining this result with \eqref{eq:unbounded:law:total:expectation} and our previous bound from \eqref{eq:unbounded:nontrivial:conditional:risk}, we conclude \[ \EE_X\|\mu-\nu^{\ast\ast}\|^2 
    \lesssim \eta_{\tilde J}^2 +\tfrac{\sigma^2}{N}\exp\left(-\tfrac{NC_9(\kappa) \eta_{\tilde J}^2}{\sigma^2}\right) +\tfrac{\sigma^2}{N}.\]
\end{proof}

\begin{proof}[\hypertarget{proof:theorem:unbounded:version}{Proof of Theorem \ref{theorem:unbounded:version}}] 

We claim that \eqref{eq:robust:unbounded:theorem:condition} must hold for $J=1$, for suppose not. Recall $d_m\ge R$ and we required in \eqref{eq:R:conditions} that $R\ge \tfrac{\sigma\sqrt{n}}{\sqrt{1-\gamma}}$. Thus $d_m^2\ge\tfrac{\sigma^2 n}{1-\gamma}$ so by definition of $\eta_1$ we have \[\tfrac{N\eta_1^2}{\sigma^2} = \tfrac{Nd_m^2 C_3(\kappa)}{(C+1)^2\sigma^2}\ge  \tfrac{C_3(\kappa) Nn}{(C+1)^2(1-\gamma)}.\]  Now if \eqref{eq:robust:unbounded:theorem:condition} does not hold, either $\tfrac{N\eta_1^2}{\sigma^2}<\log 2$ or $\tfrac{N\eta_1^2}{\sigma^2}<6 \log \cMloc(\tfrac{c\eta_1}{\sqrt{C_3(\kappa)}}, 2c).$ In the first case, we have $\tfrac{C_3(\kappa) Nn}{(C+1)^2(1-\gamma)} \le \log 2$, which means \[\tfrac{C_3(\kappa)}{(C+1)^2\log 2}\le \tfrac{C_3(\kappa)Nn}{(C+1)^2\log 2}\le 1-\gamma \quad\Rightarrow \quad \gamma \le 1- \tfrac{C_3(\kappa)}{(C+1)^2\log 2}.\] But recall (from the text leading up to \eqref{eq:R:conditions}) we assumed $\gamma$ exceeds $1- \tfrac{C_3(\kappa)}{6(C+1)^2\log 2}$, so this cannot occur. In the other case, we have  $\tfrac{N\eta_1^2}{\sigma^2}<6 \log \cMloc(\tfrac{c\eta_1}{\sqrt{C_3(\kappa)}}, 2c).$  Well from \citet[Example 5.8]{wainwright2019high} \[\log \cMloc(\tfrac{c\eta_1}{\sqrt{C_3(\kappa)}}, 2c)\le n\cdot\log\left(1+\tfrac{2\sqrt{C_3(\kappa)}}{\eta_1}\right) =n\cdot\log\left(1+\tfrac{2(C+1)}{d_m}\right).  \] Note $d_m>R$ and we assumed $R\ge 2(C+1)$, so that  $\log \cMloc(\tfrac{c\eta_1}{\sqrt{C_3(\kappa)}}, 2c) \le n\log 2.$ Thus \[\tfrac{C_3(\kappa)n }{(C+1)^2(1-\gamma)}\le \tfrac{C_3(\kappa) Nn}{(C+1)^2(1-\gamma)} \le \tfrac{N\eta_1^2}{\sigma^2} < 6\log \cMloc(\tfrac{c\eta_1}{\sqrt{C_3(\kappa)}}, 2c) \le 6n\log 2,\] i.e., $1-\gamma > \tfrac{C_3(\kappa)}{6(C+1)^2\log 2}$ which implies $\gamma < 1- \tfrac{C_3(\kappa)}{6(C+1)^2\log 2}$. But we required $\gamma > 1 -\tfrac{C_3(\kappa)}{6(C+1)^2\log 2}$, so this cannot occur. Hence, \eqref{eq:robust:unbounded:theorem:condition} holds for $J=1$.

We may now proceed as in the bounded proof by considering two cases. We check whether the $\frac{d_m}{2^{J-1}(C+1)} \ge C_1(\kappa)\epsilon\sigma$ condition fails for $J\le J^{\ast}$ or not. If it does fail, we additionally consider two sub-cases where $\epsilon \gtreqless \tfrac{1}{\sqrt{N}}$.

\textsc{Case 1:} Assume $\frac{d_m}{2^{J^{\ast}-1}(C+1)} \ge C_1(\kappa)\epsilon\sigma$. Then by Lemma \ref{lemma:for:theorem:gaussian:unbounded:part2} with $\tilde{J}=J^{\ast}$, \begin{align}
    \EE_X\|\mu-\nu^{\ast\ast}\|^2 &\lesssim \eta_{J^{\ast}}^2 +\tfrac{\sigma^2}{N}\exp\left(-\tfrac{NC_9(\kappa) \eta_{J^{\ast}}^2}{\sigma^2}\right) +\tfrac{\sigma^2}{N}. \label{eq:theorem:unbounded:penultimate:bound}
\end{align} The exponential term is $\le 1$. Moreover, \eqref{eq:robust:unbounded:theorem:condition} implies $\tfrac{\sigma^2}{N} <\tfrac{\eta_{J^{\ast}}^2}{\log 2}$. Thus, we bound \eqref{eq:theorem:unbounded:penultimate:bound} up to constants depending on $C$ and $\kappa$ with  \[ \eta_{J^{\ast}}^2+ \tfrac{\sigma^2}{N}\lesssim \eta_{J^{\ast}}^2\lesssim \max(\eta_{J^{\ast}}^2, \epsilon^2\sigma^2),\] completing the first case. 

\textsc{Case 2}: Suppose for some $J'\in\{1,2,\dots, J^{\ast}\}$ \begin{equation} \label{eq:J:prime:condition:unbounded}
             \tfrac{d_m}{2^{J'-1}(C+1)}<C_1(\kappa)\epsilon\sigma.
         \end{equation} If $J'=1$, then $d_m\le (C+1)C_1(\kappa) \epsilon\sigma$. But we also know  $d_m> R\ge\tfrac{\sigma\sqrt{n}}{\sqrt{1-\gamma}}$ and $\epsilon<1/2$, which implies \[\tfrac{1}{\sqrt{1-\gamma}}\le \tfrac{\sqrt{n}}{\sqrt{1-\gamma}} \le (C+1)C_1(\kappa) \epsilon \le \tfrac{(C+1)C_1(\kappa)}{2}.\] After some rearranging, this is equivalent to $\gamma \le 1-\tfrac{4}{(C+1)^2C_1^2(\kappa)}$. But we already assumed $\gamma > 1-\tfrac{4}{(C+1)^2C_1^2(\kappa)}$, yielding a contradiction. Thus, we assume $J'>1$ and choose it minimally. Then for all $1\le J\le J'-1$, we have $\tfrac{d_m}{2^{J-1}(C+1)}\ge C_1(\kappa)\epsilon\sigma$. Applying Lemma \ref{lemma:for:theorem:gaussian:unbounded:part2} with $\tilde{J}=J'-1$,  \begin{align}
    \EE_X\|\mu-\nu^{\ast\ast}\|^2 &\lesssim \eta_{J'-1}^2 +\tfrac{\sigma^2}{N}\exp\left(-\tfrac{NC_9(\kappa) \eta_{J'-1}^2}{\sigma^2}\right) +\tfrac{\sigma^2}{N}. \label{eq:theorem:last:bound:unbounded}
\end{align}

Now pick $C_5(\kappa)$ to be some positive absolute constant less than $\tfrac{\sqrt{\log 2}}{\sqrt{C_3(\kappa)}C_1(\kappa)}$. Let us compare $\epsilon$ to $\tfrac{C_5(\kappa)}{\sqrt{N}}$.

\textsc{Case 2(a)}: Suppose $\epsilon\le \tfrac{C_5(\kappa)}{\sqrt{N}}$. Then $\epsilon\sigma\le \tfrac{C_5(\kappa)\sigma}{\sqrt{N}}$. By definition of $J'$ in \eqref{eq:J:prime:condition:unbounded}, \begin{align*}
            \tfrac{d_m}{2^{J'-1}} < C_1(\kappa)(C+1)\epsilon\sigma \le C_5(\kappa)C_1(\kappa)(C+1)\cdot \tfrac{\sigma}{\sqrt{N}}.
        \end{align*} Applying the definition of $\eta_{J}$,  \[\tfrac{N\eta_{J'-1}^2}{\sigma^2} \le \left[\sqrt{C_3(\kappa)}C_5(\kappa)C_1(\kappa)\right]^2 <\log 2.\] Thus, \eqref{eq:robust:unbounded:theorem:condition} does not hold, and again by maximality of $J^{\ast}$, we have $J'> J^{\ast}$. This means $J'=J^{\ast}$. But  \eqref{eq:robust:unbounded:theorem:condition} not holding for $J^{\ast}$ implies $J^{\ast}=1$ and hence $J'=1$, which we already handled. 

\textsc{Case 2(b)}: Assume $\epsilon\ge \tfrac{C_5(\kappa)}{\sqrt{N}}$. Then  $\epsilon^2\sigma^2\ge \tfrac{C_5(\kappa)^2\sigma^2}{N}$. Also, the exponential term is $\le 1$. Also, by definition of $J'$ in \eqref{eq:J:prime:condition:unbounded}, \[\eta_{J'-1}=2\cdot \eta_{J'}  = \tfrac{2\sqrt{C_3(\kappa)}}{C+1}\cdot \tfrac{d}{2^{J'-1}} <2\sqrt{C_3(\kappa)}C_1(\kappa)\epsilon\sigma.\] So our expression in \eqref{eq:theorem:last:bound:unbounded} is $\lesssim\eta_{J'-1}^2 + \epsilon^2\sigma^2 \lesssim \epsilon^2\sigma^2\le\max(\epsilon^2\sigma^2,\eta_{J^{\ast}}^2)$. 
\end{proof}

    \begin{proof}[\hypertarget{proof:theorem:robust:minimax:rate:attained:unbounded:gaussian}{Proof of Theorem \ref{theorem:robust:minimax:rate:attained:unbounded:gaussian}}]

    Note that $\eta^{\ast}>0$, since if $\eta^{\ast}=0$, then for all sufficiently small $\eta$ we have $\cMloc(\eta,c)=1$, which is impossible since $K$ is an unbounded star-shaped set.
    
    \textsc{Case 1:} $\epsilon\ge \tfrac{1}{\sqrt{N}}$.
    Observe \begin{align*}
        \log \cMloc(\eta^{\ast}/4,c) &\ge \lim_{\gamma \rightarrow 0} \log \cMloc(\eta^{\ast} - \gamma,c) \ge \lim_{\gamma \rightarrow 0}\tfrac{N{(\eta^{\ast} - \gamma)}^2}{\sigma^2} \ge 2\cdot \tfrac{N(\eta^{\ast}/4)^2}{\sigma^2}.
    \end{align*} Thus, by the remarks in Section \ref{subsection:lower:bound:unbounded} about Lemma \ref{lemma:lower:bound:first:version}, the minimax rate is lower bounded by ${\eta^{\ast}}^2$ up to constants. Moreover, since $\epsilon\ge \tfrac{1}{\sqrt{N}}$, we have $\sigma^2\epsilon^2$ as a lower bound. Hence $\max({\eta^{\ast}}^2,\sigma^2\epsilon^2)$ is a lower bound.

    On the other hand, we have $\max(\eta_{J^{\ast}}^2,\epsilon^2\sigma^2)$ as an upper bound from Theorem \ref{theorem:unbounded:version}. Note by Remark \ref{remark:changing:2c:to:c}, we can swap out the $2c$ in \eqref{eq:robust:unbounded:theorem:condition} with $c$. Now, observe that if $c$ is sufficiently large, then the $\log 2$ term in \eqref{eq:robust:unbounded:theorem:condition} can be dropped (since the local entropy term can be made  large by taking points along line segments of arbitrary long lengths in $K$ by the unbounded star-shaped property). Then repeat the argument from Case 1 of the proof of Theorem \ref{theorem:robust:minimax:rate:attained:gaussian} with minor changes. We take $\beta' = \min(\tfrac{1}{\sqrt{3}}, \tfrac{c}{2}\sqrt{\tfrac{2}{C_3(\kappa)}})\in(0,1/\sqrt{3}]$ so that we get a factor of 6 instead of 4 in \eqref{eq:minimax:attained:beta:scaling} (which allows us to get the desired $2\log \left[\cMloc\left(\frac{c\tilde\eta}{2\sqrt{ C_3(\kappa)}} , c\right)\right]^3$ lower bound). Taking the same $\tilde\eta\asymp \eta^{\ast}$ such that ${\tilde\eta}^2\ge {\eta_{J^{\ast}}}^2$, disregarding the $\log 2$ requirement, and defining the analogous non-decreasing map $\phi$ without the $\log 2$ term and with a cubic power on the metric entropy. Ultimately, we conclude $\max(\eta_{J^{\ast}}^2,\epsilon^2\sigma^2) \lesssim \max(\eta_{\ast}^2,\epsilon^2\sigma^2)$ is an upper bound, matching our lower bound.

    \textsc{Case 2:} $\epsilon\le \tfrac{1}{\sqrt{N}}$. The upper bound argument in Case 1 still holds, so it remains to prove the lower bound. We also have the lower bound of ${\eta^{\ast}}^2$  by the same argument in Case 1. It suffices to show that $\epsilon^2\sigma^2$ is a lower bound. Well, if we take $\eta\asymp \epsilon\sigma$, then $\tfrac{2N\eta^2}{2\sigma^2}\lesssim 1$, and we know for sufficiently large $c$ that the local entropy can be made arbitrarily large. So again by the remarks in Section \ref{subsection:lower:bound:unbounded} about the extension of Lemma \ref{lemma:lower:bound:first:version}, the required hypothesis holds and $\epsilon^2\sigma^2$ is indeed a lower bound.
    \end{proof}

We now proceed to the unbounded sub-Gaussian setting. Due to similarities with the unbounded Gaussian setting, we skip most of the details in Lemma \ref{lemma:for:theorem:subgaussian:unbounded}. Lemma \ref{lemma:for:theorem:subgaussian:unbounded:part2} has an identical proof to Lemma \ref{lemma:for:theorem:gaussian:unbounded:part2} with some minor differences of constants which we explain how to handle.

\begin{lemma} \label{lemma:for:theorem:subgaussian:unbounded} Consider the sub-Gaussian noise setting. Let $\eta_J$ be defined as in Theorem \ref{theorem:general:subgaussian:version:unbounded} and let $C_1, C_3, C_5$ be given from Theorem \ref{theorem:asymmetric:testing:result}. Set $C_7 = \frac{15\gamma C^2}{256\cdot 32C_5}$. Suppose $\tilde J$ is such that \eqref{eq:asymmetric:robust:theorem:condition:subgaussian:unbounded} holds, that $\frac{d_m}{2^{\tilde J-1}(C+1)} \ge C_1\sigma\epsilon\sqrt{\log(1/\epsilon)}$, and additionally that $\tfrac{\sigma}{\sqrt{N}} \le \tfrac{d_m}{2^{\tilde J-1}(C+1)}$. Then for each  $1\le J\le\tilde J$ we have \[ \PP\left(\|\Upsilon_J - \mu\| \ge  \tfrac{d_m}{2^{J-1}}\right)
             \le \exp\left(-\tfrac{NC_7\eta_J^2}{2\sigma^2}\right) + 4\cdot\mathbbm{1}(J>1)\exp(-\tfrac{N\eta_J^2}{2\sigma^2}).\]
\end{lemma}
    \begin{proof}
        As with the bounded sub-Gaussian case compared to the bounded Gaussian cases, the details are identical except we substitute a different assumption about $\tfrac{d_m}{2^{J-1}}$ and obtain different absolute constants. We will also use the tournament bound in Lemma \ref{lemma:tournament:asymmetric} instead.

        Repeating the logic, we will again obtain $\PP(A_1|S\ne\varnothing)\le\exp\left(-\frac{N(1/2-\epsilon)\gamma R^2}{32\sigma^2}\right)$, that \[\PP(A_2\cap A_1^c|S\ne\varnothing)\le 2\left[\cMloc(d_m, 2c)\right]^3\exp(-\tfrac{N\eta_2^2}{\sigma^2}),\] and for $j\ge 3$, \[\PP(A_j\cap A_{j-1}^c\cap A_1^c|S\ne\varnothing)\le 2\left[\cMloc(\tfrac{d_m}{2^{j-2}}, 2c)\right]^3\exp(-\tfrac{N\eta_j^2}{\sigma^2}). \] Hence carrying through the usual calculations along with Lemma \ref{lemma:set:complement:induction}, 
        \begin{align*}
            \MoveEqLeft\PP(\|\Upsilon_J-\mu\| > \tfrac{d_m}{2^{J-1}} | S\ne\varnothing) \\ &\le \exp\left(-\tfrac{N(1/2-\epsilon)\cdot \gamma R^2}{32\sigma^2}\right)+ 4\cdot\mathbbm{1}(J>1)\exp\left(-\tfrac{N\eta_J^2}{2\sigma^2}\right).
        \end{align*} As we argued in the Gaussian case with slightly different constants, using the definition of $\eta_J$ we have $R =\tfrac{2^{J-1} (2C+1) \eta_J}{4\sqrt{C_5}} \ge \tfrac{C \eta_J}{4\sqrt{C_5}}$, hence \[\exp\left(-\tfrac{N(1/2-\epsilon)\cdot \gamma R^2}{32\sigma^2}\right) \le \exp\left(-\tfrac{N(1/2-\epsilon)\cdot \gamma C^2\eta_J^2}{512C_5\sigma^2}\right)\le \exp\left(-\tfrac{15N \gamma C^2\eta_J^2}{512\cdot 32C_5 \sigma^2}\right).\] The last step used that $\epsilon<1/32$ in the sub-Gaussian case, so $1/2-\epsilon > 15/32$. Then set $C_7 = \tfrac{15\gamma C^2}{256\cdot 32C_5}$, and the bound becomes $\exp\left(-\tfrac{NC_7\eta_J^2}{2\sigma^2}\right)$.
    \end{proof}

\begin{lemma} \label{lemma:for:theorem:subgaussian:unbounded:part2}  Consider the sub-Gaussian noise setting. Let $\eta_J$ be defined as in Theorem \ref{theorem:general:subgaussian:version:unbounded}. Suppose $\tilde J$ is such that \eqref{eq:asymmetric:robust:theorem:condition:subgaussian:unbounded} holds, that $\frac{d_m}{2^{\tilde J-1}(C+1)} \ge C_1\sigma\epsilon\sqrt{\log(1/\epsilon)}$, and additionally $ \tfrac{d_m}{2^{\tilde J-1}(C+1)}\ge \tfrac{\sigma}{\sqrt{N}}$. Then if $\nu^{\ast\ast}$ denotes the output after at least $J^{\ast}$ iterations of Algorithm \ref{algorithm:robust} in the $S(R)\ne\varnothing$ case or the smallest lexicographic point in $S(\hat R)$ in the $S(R)=\varnothing$ case, for some constant $C_9$ we have \begin{align*}
    \EE_R\EE_X\|\mu-\nu^{\ast\ast}\|^2 &\lesssim \eta_{\tilde J}^2 +\tfrac{\sigma^2}{N}\exp\left(-\tfrac{NC_9 \eta_{\tilde J}^2}{\sigma^2}\right) +\tfrac{\sigma^2}{N}.
\end{align*}
\end{lemma}

The details of the omitted proof of Lemma \ref{lemma:for:theorem:subgaussian:unbounded:part2} are extremely similar to the Gaussian unbounded case in Lemma \ref{lemma:for:theorem:gaussian:unbounded:part2}, just swapping out constants in most places. However, we explain one possible gap. In the Gaussian unbounded risk bound (Lemma \ref{lemma:for:theorem:gaussian:unbounded:part2}),  Part I and Part II yielded two bounds: $\PP(\|\cdot\|>t) \lesssim \exp(-\Omega(Nt^2/\sigma^2))$ for $t\ge C_8 R$, and $\PP(\|\cdot\|>t)\lesssim \exp(-\Omega(Nt^2/\sigma^2))$ for  $t\in[\eta_{\tilde J}\omega', \eta_{1}\omega']$ (omitting the inputs, some constants in the exponential, and the conditioning on $S\ne\varnothing$). Because it happened that $\eta_{1}\omega' > C_8 R$, in Part III we could stitch together a bound for all $t\ge \eta_{\tilde J}\omega'$. In the sub-Gaussian case, it may happen that $\eta_1\omega' < C_8 R$ (due to different analogous constants defining each of these terms). Let us explain how the argument can be resolved. Well, for $t\in [\eta_1\omega', C_8 R]$, observe that $\PP(\|\cdot\|> t) \le \PP(\|\cdot\| > \eta_1\omega')\lesssim \exp(-\Omega(NR^2/\sigma^2))\asymp \exp(-Nt^2/\sigma^2)$ using our bound for $t\in[\eta_{\tilde J}\omega', \eta_{1}\omega']$ and the fact that $t\asymp R$ when $t\in [\eta_1\omega', C_8 R]$. Thus, this potential for a gap is a non-issue for our bound even in the sub-Gaussian case. 

Another minor remark is that we used $p_k<1/e$ in Part VI of the Gaussian case, but since Lemma \ref{lemma:unbounded:mu:S} applies for both the Gaussian and sub-Gaussian setting, we have the same bound of $p_k\le\exp(-\tfrac{N\gamma(1/2-\epsilon) (2^{k-1}R)^2}{32\sigma^2})$, and our third requirement on $R$ in \eqref{eq:R:conditions} still yields $p_k<1/e$.

\begin{proof}[\hypertarget{proof:theorem:general:subgaussian:version:unbounded}{Proof of Theorem \ref{theorem:general:subgaussian:version:unbounded}}] 

        By the same reasoning as in the proof of  Theorem \ref{theorem:unbounded:version} (but instead requiring $\gamma > 1 -\tfrac{C_5}{6(C+1)^2\log 2}$), the condition \eqref{eq:asymmetric:robust:theorem:condition:subgaussian:unbounded} must hold for $J=1$, so for all $1\le J\le J^{\ast}$ (including if $J^{\ast}=1$), we have \eqref{eq:asymmetric:robust:theorem:condition:subgaussian:unbounded}.

        Now we consider 3 cases as in the bounded, sub-Gaussian setting. 
        
       \textsc{Case 1:} Suppose  $\tfrac{d_m}{2^{J^{\ast}-1}(C+1)} \ge C_1\sigma\epsilon\sqrt{\log(1/\epsilon)}$, and also that $\tfrac{d_m}{2^{J^{\ast}-1}(C+1)}\ge \tfrac{\sigma}{\sqrt{N}} $. Then \begin{align*}
            \EE_R\EE_X\|\mu-\nu^{\ast\ast}\|^2 &\lesssim \eta_{J^{\ast}}^2 +\tfrac{\sigma^2}{N}\exp\left(-\tfrac{NC_9 \eta_{J^{\ast}}^2}{\sigma^2}\right) +\tfrac{\sigma^2}{N}. 
        \end{align*} by Lemma \ref{lemma:for:theorem:subgaussian:unbounded:part2}. Then since $\tfrac{N\eta_{J^{\ast}}^2}{\sigma^2}>\log 2$, we have $\tfrac{\sigma^2}{N}\lesssim\eta_{J^{\ast}}^2$. The exponential is always $\le 1$. Hence $ \EE_R\EE_X\|\mu-\nu^{\ast\ast}\|^2\lesssim \eta_{J^{\ast}}^2\le \max(\eta_{J^{\ast}}^2, \sigma^2\epsilon^2\log(1/\epsilon))$.

        \textsc{Case 2:} Suppose for some minimally chosen index $1\le J'\le J^{\ast}$ we have $\frac{d_m}{2^{J'-1}(C+1)} <\tfrac{\sigma}{\sqrt{N}}$.

        \textsc{Case 2(a):} Suppose $J'=1$. Then $d_m\le \tfrac{(C+1)\sigma}{\sqrt{N}}$. But we also know $d_m>R \ge \tfrac{\sigma\sqrt{n}}{\sqrt{1-\gamma}}$, so $\tfrac{\sqrt{n}}{\sqrt{1-\gamma}} \le \tfrac{C+1}{\sqrt{N}}$. This means $\sqrt{1-\gamma}\ge \tfrac{\sqrt{Nn}}{C+1} \ge \tfrac{1}{C+1}$, which rearranges to $\gamma\le 1-\tfrac{1}{(C+1)^2}$. But recall we required $\gamma> 1-\tfrac{1}{(C+1)^2}$, a contradiction. Thus $J'\ge 2$.
        
        \textsc{Case 2(b):} Suppose $J'\ge 2$. Then by minimality, each $1\le J\le J'-1$ satisfies $\frac{d_m}{2^{J-1}(C+1)} \ge \tfrac{\sigma}{\sqrt{N}}$. 

        \textsc{Case 2(b)(i):} Assume there exists some minimally chosen index $1\le J''\le J'-1$  such that $\tfrac{d_m}{2^{J''-1}(C+1)}< C_1\sigma\epsilon\sqrt{\log(1/\epsilon)}$. If $J''=1$, then $d_m \le C_1(C+1)\sigma\epsilon\sqrt{\log(1/\epsilon)}$. But as we showed in the proof of Theorem \ref{theorem:unbounded:version}, we have $d_m>R\ge \tfrac{\sigma\sqrt{n}}{\sqrt{1-\gamma}}$ so \[\tfrac{1}{\sqrt{1-\gamma}}\le \tfrac{\sqrt{n}}{\sqrt{1-\gamma}} \le C_1(C+1) \epsilon\sqrt{\log(1/\epsilon)} \le C_1(C+1)\] since $\epsilon\sqrt{\log(1/\epsilon)}\le 1.$ This re-arranges to $\gamma \le 1- \tfrac{1}{C_1^2(C+1)^2}$, in violation of our requirements on $\gamma$.
        
        So suppose $J''>1$. Then repeat the argument of Case 2(b)(i) in the proof of Theorem \ref{theorem:general:subgaussian:version} in the $J''>1$ cases, where we replace $d$ with $d_m$ and apply Lemma \ref{lemma:for:theorem:subgaussian:unbounded:part2} instead to obtain \begin{align*}
            \EE_R\EE_X\|\mu-\nu^{\ast\ast}\|^2 &\lesssim\eta_{J''-1}^2 +\tfrac{\sigma^2}{N}\exp\left(-\tfrac{NC_9 \eta_{J''-1}^2}{\sigma^2}\right) +\tfrac{\sigma^2}{N}. 
        \end{align*} In one sub-case ($\epsilon\sqrt{\log(1/\epsilon)}\le C_6/\sqrt{N}$ where $C_6$ is defined in the bounded proof), we would conclude \eqref{eq:asymmetric:robust:theorem:condition:subgaussian:unbounded} does not hold for $J''$, which means $J''>J^{\ast}$, a contradiction since $J''<J^{\ast}$. In the other sub case, ($\epsilon\sqrt{\log(1/\epsilon)}\ge C_6/\sqrt{N}$), note that $\tfrac{\sigma^2}{N} \lesssim \sigma^2\epsilon^2\log(1/\epsilon)$ and that the exponential term is $\le 1$. Moreover, $\eta_{J''-1}=2\eta_{J''}\asymp \frac{d_m}{2^{J''-1}}\lesssim \sigma \epsilon\sqrt{\log(1/\epsilon)}$. So the entire bound on $\EE_R\EE_X\|\mu-\nu^{\ast\ast}\|^2$ is $\lesssim \sigma^2\epsilon^2\log(1/\epsilon)$ and thus $\lesssim \max(\eta_{J^{\ast}}^2,\sigma^2\epsilon^2\log(1/\epsilon))$.

          \textsc{Case 2(b)(ii):} In this subcase, we have $\tfrac{d_m}{2^{J'-1}(C+1)}<\tfrac{\sigma}{\sqrt{N}}$ but for $1\le J\le J'-1$, we have $\tfrac{d_m}{2^{J-1}(C+1)}\ge\tfrac{\sigma}{\sqrt{N}}$ and $\tfrac{d_m}{2^{J-1}(C+1)}\ge C_1\sigma\epsilon \sqrt{\log(1/\epsilon)}$. Lemma \ref{lemma:for:theorem:subgaussian:unbounded:part2} implies \begin{align*}
            \EE_R\EE_X\|\mu-\nu^{\ast\ast}\|^2 &\lesssim\eta_{J'-1}^2 +\tfrac{\sigma^2}{N}\exp\left(-\tfrac{NC_9 \eta_{J'-1}^2}{\sigma^2}\right) +\tfrac{\sigma^2}{N}, 
        \end{align*} and then we repeat the proof of Case 2(b)(ii) in the bounded sub-Gaussian version to argue the bound in $\EE_R\EE_X\|\mu-\nu^{\ast\ast}\|^2$ above reduces to $\lesssim\tfrac{\sigma^2}{N}$. But we also know $\eta_{J^{\ast}} \ge \tfrac{\sigma\log 2}{\sqrt{N}}$ by definition of $J^{\ast}$, so $\EE_R\EE_X\|\mu-\nu^{\ast\ast}\|^2\lesssim \eta_{J^{\ast}}^2 \lesssim \max(\eta_{J^{\ast}}^2,\sigma^2\epsilon^2\log(1/\epsilon))$.

        \textsc{Case 3:} Suppose $\tfrac{d_m}{2^{J'-1}(C+1)}<\tfrac{\sigma}{\sqrt{N}}$ never occurs for $J'\le J^{\ast}$, but at some point $1\le J'\le J^{\ast}$ we have $\tfrac{d_m}{2^{J'-1}(C+1)}<C_1\sigma\epsilon \sqrt{\log(1/\epsilon)}$. Then repeat the proof of Case 2(b)(i) using index $J'$ instead of $J''$, i.e., let $J'$ be the minimal  index such that $\tfrac{d_m}{2^{J'-1}(C+1)}< C_1\sigma\epsilon\sqrt{\log(1/\epsilon)}$ (since  we may apply Lemma \ref{lemma:for:theorem:subgaussian:unbounded:part2} with $J'$).
    \end{proof}

    \begin{proof}[\hypertarget{proof:theorem:robust:minimax:rate:attained:unbounded:subgaussian}{Proof of Theorem \ref{theorem:robust:minimax:rate:attained:unbounded:subgaussian}}] 

    The $\eta^{\ast}=0$ edge-case cannot happen as argued in the proof of Theorem \ref{theorem:robust:minimax:rate:attained:unbounded:gaussian}.

   Next, without modifications, the argument from Case 1 in the proof of Theorem \ref{theorem:robust:minimax:rate:attained:unbounded:gaussian} that ${\eta^{\ast}}^2$ is a lower bound carries through (this portion did not use an $\epsilon\ge \tfrac{1}{\sqrt{N}}$ assumption and still works for sub-Gaussian noise). By the remarks in Section \ref{subsection:lower:bound:unbounded} about extending Lemma \ref{lemma:subgaussian:lower:bound}, we also have $\sigma^2\epsilon^2\log(1/\epsilon)$ as a lower bound. Hence, $\max({\eta^{\ast}}^2, \sigma^2\epsilon^2\log(1/\epsilon))$ is a lower bound.

    For the upper bound, recall from Remark \ref{remark:changing:2c:to:c} that we can swap out the $2c$ in \eqref{eq:asymmetric:robust:theorem:condition:subgaussian:unbounded} with $c$ in the following proof.  Moreover, as in the proof of Theorem \ref{theorem:robust:minimax:rate:attained:unbounded:gaussian}, we can also drop the $\log 2$ term for sufficiently large $c$.  Theorem \eqref{theorem:general:subgaussian:version:unbounded} implies $\max(\eta_{J^{\ast}}^2,\sigma^2\epsilon^2\log(1/\epsilon))$ is an upper bound. Now repeat the argument in Case 1 of the proof of Theorem \ref{theorem:robust:minimax:rate:attained:subgaussian} with $\beta' = \min(\tfrac{1}{\sqrt{3}}, \tfrac{c}{2}\sqrt{\tfrac{2}{C_5}})\in(0,1/\sqrt{3}]$ (for the same reason as in Theorem \ref{theorem:robust:minimax:rate:attained:unbounded:gaussian}), with the same $\tilde\eta$ and $\phi$ defined without a $\log 2$ term and a cubic power instead. The argument implies $\tilde\eta^2\ge \eta_{J^{\ast}}^2$ where $\tilde\eta\asymp\eta^{\ast}$, and hence $\max(\eta_{\ast}^2,\sigma^2\epsilon^2\log(1/\epsilon))$ is an upper bound.
    \end{proof}

    \begin{proof}[\hypertarget{proof:lemma:sparse:varshamov}{Proof of Lemma \ref{lemma:sparse:varshamov}}] 
        We apply the sparse Varshamov-Gilbert Lemma given in \citet[Lemma 4.14]{rigollet2023highdimensional}. This yields $\omega_1,\dots,\omega_M\in\{0,1\}^n$ where each $\omega_i$ has $s$ non-zero coordinates, $\omega_i$ and $\omega_j$ disagree on at least $s/2$ coordinates for $i\ne j$, and $\log M\ge \tfrac{s}{8}\log(1+\tfrac{n}{2s})$. Define $\omega_i' = \omega_i\delta/\sqrt{s}$, and then observe by $s$-sparsity and $\omega_i$ being binary that $\|\omega_i'\|^2=\delta^2$. Thus, $\omega_i'\in B(0,\delta)\cap K$, using a closed ball. Moreover, for $i\ne j$ we have \begin{align*}
            \|\omega_i'-\omega_j'\|^2=\tfrac{\delta^2}{s}\|\omega_i-\omega_j\|^2 \ge \tfrac{\delta^2}{s}\cdot \tfrac{s}{2}= \tfrac{\delta^2}{2} > \tfrac{\delta^2}{c^2}, 
        \end{align*} for $c>\sqrt{2}$. Thus, we have a $\delta/c$-packing of $B(0,\delta)\cap K$ of log cardinality $\ge \tfrac{s}{8}\log(1+\tfrac{n}{2s})$, proving the lower bound.

      Now define the local Gaussian width at a point $\nu$ by\begin{equation*}
            w_{\nu}(\delta)=\EE_{Z\sim \cN(0,\II)}\bigg[\sup_{y\in B(\nu, \delta)\cap K}Z\T y\bigg].
        \end{equation*} Then by a Sudakov minoration argument \citep[Theorem 5.30]{wainwright2019high}, we have  $\log \cMKloc(\delta,c)\lesssim \sup_{\nu\in K}\tfrac{[w_{\nu}(\delta)]^2}{\delta^2}.$ Given any subset $S\subseteq[n]$ with $|S|=2s$, write $Z_S$ to be equal to the random vector $Z$ but with coordinates $i\in S$ set to $0$ and coordinates $i\not\in S$ left unchanged. Note that if $y,\nu$ are both $s$-sparse, then $y-\nu$ is at most $2s$-sparse. Thus,
        \begin{align*}
            \EE_{Z}\bigg[\sup_{y\in B(\nu, \delta)\cap K}Z\T y\bigg] &= \EE_{Z}\bigg[\sup_{y\in B(\nu, \delta)\cap K}Z\T (y-\nu)\bigg]  \\
            &= \EE_{Z}\bigg[\sup_{\substack{S\subseteq[n]\\|S|=2s}} \sup_{y\in B(\nu, \delta)\cap K}Z_S\T (y-\nu)\bigg] \\
            &\le \EE_{Z} \sup_{\substack{S\subseteq[n]\\|S|=2s}} \sup_{y\in B(\nu, \delta)\cap K}\|Z_S\|\|y-\nu\| \\
            &\le  \delta \cdot\EE_{Z}\sup_{\substack{S\subseteq[n]\\|S|=2s}} \|Z_S\| \\
            &\le \delta \cdot \EE_{Z}\sqrt{\sum\nolimits_{i=1}^{2s} Z_{(i)}^2},
        \end{align*} where we write the coordinates of $Z$ in order of non-increasing absolute value, say $|Z_{(1)}|\ge |Z_{(2)}|\ge\dots \ge |Z_{(n)}|$.

        Now, recall from Lemma \ref{lemma:R:tail:bound} that $\PP(\|W\| >R) \le 25^s\exp(-R^2/8)$ if $W$ is an $2s$-dimensional sub-Gaussian with parameter $1$. Now if $\sqrt{\sum_{i=1}^{2s} Z_{(i)}^2} > R$, that means one of the $\binom{n}{2s}$ many choices of coordinates of $Z$ was chosen, and the norm of those $2s$ coordinates exceeds $R$. Hence,  a union bound and binomial coefficient bound implies \begin{align*}
            \PP\left(\sqrt{\sum\nolimits_{i=1}^{2s} Z_{(i)}^2} \ge R\right) &\le \binom{n}{2s} 25^s \exp(-\tfrac{R^2}{8}) \\ &\le \left[\left(\tfrac{en}{2s}\right)^{2s} 25^s\exp(-\tfrac{R^2}{8})\right] \wedge 1,
        \end{align*} noting the bound is only relevant when $\le 1$. Moreover, note that for $R>R'$ where $R' = \sqrt{8 s\log 25 + 16s \log(en/(2s))}$, we have $\left(\tfrac{en}{2s}\right)^{2s}25^s \exp(-{R}^2/8)\le 1$.
        Then \begin{align*}
            \MoveEqLeft \EE\sqrt{\sum\nolimits_{i=1}^{2s} Z_{(i)}^2} \\ &= \int_0^{\infty} \PP\left(\sqrt{\sum\nolimits_{i=1}^s Z_{(i)}^2} \ge R\right) \mathrm{d}R \\ 
            &= \int_0^{R'} \PP\left(\sqrt{\sum\nolimits_{i=1}^{2s} Z_{(i)}^2} \ge R\right) \mathrm{d}R  + \int_{R'}^{\infty} \PP\left(\sqrt{\sum\nolimits_{i=1}^{2s} Z_{(i)}^2} \ge R\right) \mathrm{d}R \\
            &\le  \int_0^{R'} \mathrm{d}R  + \int_{R'}^{\infty}\left(\tfrac{en}{2s}\right)^{2s} 25^s\exp(-R^2/8) \mathrm{d}R  \\
            &= R' + \left(\tfrac{en}{2s}\right)^{2s} 25^s\int_{R'}^{\infty} \exp(-R^2/8) \mathrm{d}R \\
            &\lesssim R' + \left(\tfrac{en}{2s}\right)^{2s} 25^s \exp(-{R'}^2/8) \lesssim R'+1 \lesssim R'.
        \end{align*} Thus \begin{align*}
            \log \cMKloc(\delta,c) &\lesssim \sup_{\nu\in K}\frac{[w_{\nu}(\delta)]^2}{\delta^2} \le \frac{\left[\delta\cdot \EE_{Z}\sqrt{\sum_{i=1}^{2s} Z_{(i)}^2}]\right]^2}{\delta^2} \\ &\lesssim {R'}^2 \asymp s\log(1+\tfrac{n}{2s}).
        \end{align*}
    \end{proof}

\end{document}